\documentclass[12pt,leqno]{article}
\usepackage{amssymb,amsfonts,amsmath,amsthm,amscd,mathrsfs}
\usepackage{MnSymbol}
\usepackage{PSFigure,psfrag}
\usepackage{graphicx}
\usepackage{centernot}
\usepackage{stmaryrd}
\def\Vu{\stackrel{\cV^u}{\longleftrightarrow} \hspace{-2.8ex} \mbox{\f /}\;\;\;}
\def\Va{\stackrel{\cV^a}{\longleftrightarrow} \hspace{-2.8ex} \mbox{\f /}\;\;\;}
\def\Vv{\stackrel{\cV^v}{\longleftrightarrow} \hspace{-2.8ex} \mbox{\f /}\;\;\;}

\def\IE{{\mathbb E}}

\def\IP{{\mathbb P}}
\def\IR{{\mathbb R}}
\def\IN{{\mathbb N}}
\def\IL{{\mathbb L}}

\def\IQ{{\mathbb Q}}
\def\IZ{{\mathbb Z}}

\def\n{\noindent}
\def\dsl{\textstyle\sum\limits}

\def\dis{\displaystyle}
\def\o{\omega}
\def\fr{\mbox{\footnotesize $\dis\frac{1}{2}$}}

\def\fd{\mbox{\footnotesize $\dis\frac{1}{2d}$}}

\def\faa{\mbox{\footnotesize $\dis\frac{a}{L^2}$}}
\def\ov{\overline}
\def\ve{\varepsilon}
\def\f{\footnotesize}

\def\r{\rightarrow}
\def\point{{\mbox{\large $.$}}}
\def\wh{\widehat}
\def\wt{\widetilde}

\def\cA{{\cal A}}
\def\cB{{\cal B}}
\def\cC{{\cal C}}
\def\cD{{\cal D}}
\def\cE{{\cal E}}

\def\cE{{\cal E}}
\def\cI{{\cal I}}

\def\cF{{\cal F}}

\def\cG{{\cal G}}

\def\cV{{\cal V}}

\dimendef\dimen=0

\newtheorem{theorem}{Theorem}[section]
\newtheorem{lemma}[theorem]{Lemma}
\newtheorem{corollary}[theorem]{Corollary}
\newtheorem{proposition}[theorem]{Proposition}
\newtheorem{remark}[theorem]{Remark}
\newtheorem{example}[theorem]{Example}

\begin{document}

\noindent
~

\bigskip
\begin{center}
{\bf ON BULK DEVIATIONS FOR THE LOCAL BEHAVIOR \\ OF RANDOM INTERLACEMENTS}
\end{center}

\begin{center}
Alain-Sol Sznitman
\end{center}

\begin{abstract}
We investigate certain large deviation asymptotics concerning random interlacements in $\IZ^d, d \ge 3$. We find the principal exponential rate of decay for the probability that the average value of some suitable non-decreasing local function of the field of occupation times, sampled at each point of a large box, exceeds its expected value. We express the exponential rate of decay in terms of a constrained minimum for the Dirichlet energy of functions on $\IR^d$ that decay at infinity. An application concerns the excess presence of random interlacements in a large box. Our findings exhibit similarities to some of the results of van~den~Berg-Bolthausen-den~Hollander in their work on moderate deviations of the volume of the Wiener sausage. An other application relates to recent work of the author on macroscopic holes in connected components of the vacant set.

\medskip
\begin{center}
{\bf R\'esum\'e}
\end{center}
Nous \'etudions certaines asymptotiques de grandes d\'eviations pour les entrelacs al\'eatoires sur $\IZ^d, d \ge 3$. Nous d\'eterminons le taux principal de d\'ecroissance exponentielle pour la probabilit\'e que la valeur moyenne d'une fonction croissante au sens large du champ des temps d'occupation, \'echantillonn\'ee en chaque point d'une grande bo\^ite, d\'epasse son esp\'erance. Nous exprimons le taux de d\'ecroissance exponentielle en termes d'un minimum sous contrainte de l'\'energie de Dirichlet de fonctions sur $\IR^d$ qui s'annulent \`a l'infini. Cela s'applique au cas d'une pr\'esence excessive des entrelacs dans une grande bo\^ite. Nos r\'esultats dans cet exemple pr\'esentent des similarit\'es avec certains de ceux de van~den~Berg-Bolthausen-den~Hollander dans leur article sur les d\'eviations mod\'er\'ees du volume de la saucisse de Wiener. Une autre application a trait aux travaux r\'ecents de l'auteur concernant les trous macroscopiques dans les composantes de l'ensemble vacant.

\end{abstract}

\vspace{2 cm} 
\n
Departement Mathematik  
\\
ETH Z\"urich\\
CH-8092 Z\"urich\\
Switzerland

\vfill

~
\newpage
\thispagestyle{empty}
~

\newpage
\setcounter{page}{1}

\setcounter{section}{-1}
\section{Introduction}

Random interlacements have deep links with random walks and with the Gaussian free field, see \cite{Szni10}, \cite{DrewPrevRodr}, and in some respects behave as models of statistical mechanics having a continuous symmetry, such as the (massless) Gaussian free field. 
In the present work we investigate certain large deviation asymptotics related to the occupation times of the continuous-time random interlacements on $\IZ^d, d \ge 3$. In essence, we aim at finding the principal exponential rate of decay of the probability that the average value of a non-decreasing local function of the field of occupation times of random interlacements at level $u > 0$, sampled at each point of a large box of $\IZ^d$ centered at the origin, exceeds its expected value. For instance, if $\cI^u \subseteq \IZ^d$ stands for the random interlacements at level $u$, we establish a formula for the principal exponential rate of decay of the probability that the fraction of sites of $\cI^u$ in a large box centered at the origin exceeds a value $\nu$ that is bigger than $1 - e^{-u/g(0,0)}$ (with $g(\cdot,\cdot)$ the Green function of the simple random walk on $\IZ^d$), i.e.~bigger than the probability that the origin lies in $\cI^u$. As an other illustration, given any integer $R$, we establish a formula for the principal exponential rate of decay of the probability that in a large box centered at the origin, the fraction of sites $x$ that are disconnected by $\cI^u$ from the sphere $S(x,R)$ of sites at sup-distance $R$ from $x$ exceeds a value $\nu$ that is bigger than the probability $\IP[0 \stackrel{\cV^u}{\longleftrightarrow} \hspace{-2.8ex} \mbox{\f /}\;\;\;S(0,R)]$ that $0$ gets disconnected from $S(0,R)$ by $\cI^u$ ($\cV^u = \IZ^d \backslash \cI^u$ is the so-called vacant set of random interlacements). The principal exponential rates of decay that we obtain are expressed in terms of constrained minima for the Dirichlet energy of functions on $\IR^d$ that decay at infinity. The first example mentioned above exhibits similarities with some of the results obtained by van~den~Berg-Bolthausen-den~Hollander in their work \cite{BoltHollVanb01} on moderate deviations of the volume of Wiener sausage. The second example is related to the recent work of the author concerning macroscopic holes in the connected components of the vacant set of random interlacements in the strongly percolative regime, see \cite{Szni}.

\medskip
We now describe our results in more details. We consider $\IZ^d, d \ge 3$, and, given $u \ge 0$, denote by $(L^u_x)_{x \in \IZ^d}$ the field of occupation times of continuous-time random interlacements at level $u$, by $\cI^u$ the random interlacements at level $u$, and by $\cV^u = \IZ^d \backslash \cI^u$ the corresponding vacant set at level $u$ (so that $\cI^u = \{x \in \IZ^d$; $L^u_x > 0\}$ and $\cV^u = \{x \in \IZ^d$; $L^u_x = 0\}$). We denote by $\IP$ the probability measure governing these objects, and by $\IE$ the corresponding expectation. We refer to \cite{CernTeix12}, \cite{DrewRathSapo14c}, Section 1 of \cite{Szni15}, and below (\ref{1.11}) for further details and references.

\medskip
We probe the field of occupation times with the help of local functions. Each local function $F$ comes with a non-negative integer $R$ (its range), and is a function on $[0,\infty)^{B(0,R)}$ (with $B(0,R) = \{x \in \IZ^d; |x|_\infty \le R\}$ the closed ball in sup-distance with center $0$ and radius $R$). Throughout this article, we assume that $F$ satisfies (\ref{2.1}). In essence, this condition requires that $F$ is non-decreasing in each variable, $F(0) = 0$, $F$ satisfies a sub-linear growth condition (which incidentally is automatically fulfilled when $F$ is bounded), and it then follows from (\ref{2.1}) (see below (\ref{2.1})) that the map
\begin{equation}\label{0.1}
u \ge 0 \mapsto \theta(u) = \IE \big[F\big((L^u_y)_{|y|_\infty \le R}\big)\big] \ge 0, \;\mbox{is continuous}
\end{equation}
($\theta$ is finite, non-decreasing, and vanishes for $u = 0$, as a result of the previously mentioned requirements on $F$). Some examples of local functions $F$ of interest can be found in Example \ref{exam2.1} of Section 2. For instance, $F_0(\ell) = \ell$ (with $R = 0$ and $\theta(u) = u$, see (\ref{1.12})) satisfies (\ref{2.1}). More pertinent examples for our purpose are for instance
\begin{equation}\label{0.2}
\begin{array}{l}
\mbox{$F(\ell) = 1\{\ell > 0\}$ (with $R = 0$ and $\theta(u) = 1-e^{-u/g(0,0)}$ where $g(\cdot,\cdot)$ is the Green}
\\
\mbox{function of the simple random walk, see the beginning of Section 1).}
\end{array}
\end{equation}
and for $R \ge 0$,
\begin{equation}\label{0.3}
\begin{array}{l}
\mbox{$F(\ell) = 1\{$any path in $B(0,R)$ from $0$ to $S(0,R)$ meets a $y$ with $\ell_y > 0\}$}
\\
\mbox{(with $\ell \in [0,\infty)^R$ and $\theta(u) = \IP[0 \Vu S(0,R)]$ the probability that}
\\
\mbox{$\cI^u$ disconnects $0$ from $S(0,R)$).}
\end{array}
\end{equation}

\medskip
Given $u > 0$ and a local function $F$ as above (that is, satisfying (\ref{2.1})), one knows by the spatial ergodic theorem (see Theorem 2.8, p.~205 of \cite{Kren85}, and also (4.4) of \cite{LiSzni15}) that
\begin{equation}\label{0.4}
\mbox{$\IP$-a.s., $\;\mbox{\f $\dis\frac{1}{|B(0,N)|}$} \; \dsl_{x \in B(0,N)} F\big((L^u_{x+ \point})\big) \underset{N \r \infty}{\longrightarrow} \theta(u) \quad (= \IE[F((L^u_\point))])$},
\end{equation}
where $|B(0,N)|$ stands for the number of sites in $B(0,N)$.

\medskip
Actually, it is expedient to consider slightly more general sequences than $B(0,N)$, namely sequences obtained as the discrete blow-up of a model shape $D$ in $\IR^d$:
\begin{equation}\label{0.5}
D_N = (ND) \cap \IZ^d, \;N \ge 1,
\end{equation}
where we assume that
\begin{equation}\label{0.6}
\begin{array}{l}
\mbox{$D \subseteq \IR^d$ is the closure of a smooth bounded domain containing $0$,}
\\
\mbox{or of an open $|\cdot|_\infty$-ball, which contains $0$}.
\end{array}
\end{equation}
We are interested by large deviation events of the {\it excess} type:
\begin{equation}\label{0.7}
\cA_N = \big\{ \dsl_{x \in D_N} F\big((L^u_{x + \point})\big) > \nu \,|D_N|\big\}, \; N \ge 1,
\end{equation}
where $\nu \in (0,\infty)$ is chosen bigger than $\theta (u)$. We have less to say on {\it deficit} deviation events, where the inequality in (\ref{0.7}) is reversed and $\nu < \theta(u)$, see Remarks \ref{rem4.6} and \ref{rem5.5}.

\medskip
We derive asymptotic lower and upper bounds on $\frac{1}{N^{d-2}} \; \log \IP[\cA_N]$ in our main Theorems \ref{theo4.2} and \ref{theo5.1}. As an application of these theorems, we show in Corollary \ref{cor5.9}, that when $F$ satisfying (\ref{2.1}) is bounded, non-constant (this implies that $\theta$ is bounded and strictly increasing, see (\ref{2.2})), then, for $u > 0$, and $\theta(u) < \nu < \theta_\infty = \lim_{v \r \infty} \theta(v)$, one has
\begin{equation}\label{0.8}
\begin{array}{l}
\lim\limits_N \; \mbox{\f $\dis\frac{1}{N^{d-2}}$} \;\log \IP[\cA_N]  
\\[2ex]
= - \inf\Big\{ \fd \;\dis\int_{\IR^d} |\nabla \varphi |^2 \,dz; \varphi \ge 0, \,\varphi \in C^\infty_0 (\IR^d) \;\mbox{and} \; \strokedint_D \theta\big((\sqrt{u} + \varphi)^2\big) \,dz > \nu\Big\}
\\[2ex]
=  -\min\Big\{  \fd \;\dis\int_{\IR^d} |\nabla \varphi |^2 \,dz; \varphi \ge 0, \,\varphi \in D^1 (\IR^d) \;\mbox{and} \;\strokedint_D \theta\big((\sqrt{u} + \varphi)^2\big) \,dz = \nu\Big\}
\end{array}
\end{equation}
where $\strokedint_D dz$ refers to the normalized integral $\frac{1}{|D|} \int_D \dots dz$, with $|D|$ the Lebesgue measure of $D$, and $D^1(\IR^d)$ stands for the space of locally integrable functions that vanish at infinity, with finite Dirichlet integral (see Chapter 8~\S 2 in \cite{LiebLoss01}). One can actually omit the condition $\varphi \ge 0$ both in the infimum and the minimum in (\ref{0.8}) (this is expedient when perturbing around a minimizer $\varphi$). Further, when $d \ge 5$, one can replace $D^1(\IR^d)$ by the more traditional Sobolev space $H^1(\IR^d)$, see Remark \ref{rem5.10} 1). Also, when $D$ is a Euclidean ball, there is a spherically symmetric minimizer in the variational problem on the second line of (\ref{0.8}), see Remark \ref{rem5.10} 2). Moreover, it is conceivable that the boundedness assumption made on $F$ might be relaxed. Here it mainly plays a role in Theorem \ref{theo5.1} for the derivation of the asymptotic upper bound, see Remarks \ref{rem4.6} 2) and \ref{rem5.2} 2).

\medskip
Simple random walk is known to informally correspond to the ``singular limit $u \r 0$'' for random interlacements, see for instance Section 7 of \cite{Szni15}. The statement (\ref{0.8}) naturally leads to the asymptotic upper bound (see Corollary \ref{cor5.11} and (\ref{5.88})): for each $y \in \IZ^d$,
\begin{equation}\label{0.8a}
\begin{array}{l}
\limsup\limits_N \mbox{\f $\dis\frac{1}{N^{d-2}}$} \; \log P_y [\cA^0_N] \le
\\
- \min  \Big\{ \mbox{\f $\dis\frac{1}{2d}$}  \;\dis\int_{\IR^d} |\nabla \varphi |^2 \,dz; \varphi \ge 0, \varphi \in D^1(\IR^d) \; \mbox{and} \; \dis\strokedint_D \theta(\varphi^2) \, dz = \nu\Big\},
\end{array}
\end{equation}
where $P_y$ stands for the law of the simple random walk with unit jump rate starting at $y$, $\cA^0_N$ is defined as $\cA_N$ with $L^u_\point$ is replaced by the field of occupation times of the walk, and $\nu$ belongs to $(0, \theta_\infty)$. The lower bound corresponding to (\ref{0.8a}) remains open and we refer to Remark \ref{rem5.12} for more on this topic.

\medskip
The statement (\ref{0.8}) can for instance be applied to the number of sites of the interlacements $\cI^u$ present in $D_N$. Namely, with $F$ as in (\ref{0.2}), we show in Theorem \ref{theo6.1} that for $u > 0$ and $1 - e^{-u/g(0,0)} < \nu < 1$,
\begin{equation}\label{0.9}
\begin{array}{l}
\lim\limits_N \; \mbox{\f $\dis\frac{1}{N^{d-2}}$} \;\log \IP \big[| \cI^u \cap D_N| > \nu \,|D_N|\big]
\\[2ex]
= - \inf\Big\{  \fd \;\dis\int_{\IR^d} |\nabla \varphi |^2 \,dz; \varphi \ge 0, \,\varphi \in C^\infty_0 (\IR^d) \;\mbox{and}  \;\strokedint_D 1 - e^{(\sqrt{u} + \varphi)^2 / g(0,0)} dz > \nu\Big\}
\\[2ex]
= -\min\Big\{  \fd \;\dis\int_{\IR^d} |\nabla \varphi |^2 \,dz; \varphi \ge 0, \,\varphi \in D^1 (\IR^d) \;\mbox{and} \;
\strokedint_D 1 - e^{(\sqrt{u} + \varphi)^2 / g(0,0)} dz = \nu\Big\}
\end{array}
\end{equation}
(and one can drop the condition $\varphi \ge 0$ in the above expressions).

\medskip
The above statement bears some flavor of the asymptotics on the moderate deviations of the volume of the Wiener sausage, obtained in Theorem 1 of  \cite{BoltHollVanb01}, see Remark \ref{rem6.2} 1).

\begin{samepage}
\medskip
Corollary \ref{cor5.9} (i.e.~(\ref{0.8})) can also be applied to the number of sites in $D_N$, such that the ball of radius $r$ at these points is disconnected by $\cI^u$ from the concentric sphere of radius $R( > r)$. Namely, as shown in Theorem \ref{theo6.3}, when $0 \le r < R $ are integers, for $u > 0$, and $\IP [B(0,r) \Vu S(0,R)] < \nu < 1$, one has with $\theta_{r,R}(a) = \IP[B(0,r) \Va S(0,R)]$,
\begin{equation}\label{0.10}
\begin{array}{l}
\lim\limits_N \; \mbox{\f $\dis\frac{1}{N^{d-2}}$} \;\log \IP \big[| \{x \in D_N; \, B(x,r) \Vu S(x,R)\}| > \nu \,|D_N|\big]
\\[2ex]
= - \inf\Big\{  \fd \;\dis\int_{\IR^d} |\nabla \varphi |^2 \,dz;\;\varphi \ge 0, \; \varphi \in C^\infty_0 (\IR^d) \;\mbox{and}  \;\strokedint_D  \theta_{r,R} \big((\sqrt{u} + \varphi)^2\big)\,dz > \nu \Big\}
\\[2ex]
= -\min\Big\{  \fd \;\dis\int_{\IR^d} |\nabla \varphi |^2 \,dz; \;\varphi \ge 0, \;\varphi \in D^1 (\IR^d) \;\mbox{and} \;
\strokedint_D  \theta_{r,R} \big((\sqrt{u} + \varphi)^2\big)\,dz = \nu\Big\}
\end{array}
\end{equation}
(again, one can drop the condition $\varphi \ge 0$ in the above expressions).

\end{samepage}

\medskip
The above asymptotics can be linked to the results in \cite{Szni}. There, the connected component $\cC^u_N$ of $S(0,N)$ in $\cV^u \cup S(0,N)$ was considered, together with its thickening $\wt{\cC}^u_N$ (the $\wt{L}_0(N)$-neighborhood of $\cC^u_N$, for a suitably chosen $\wt{L}_0(N) = o(N)$). 
The asymptotic behavior of the probability of a ``macroscopic hole'' of volume $\nu \,|B(0,N)|$ left by $\wt{\cC}^u_N$ in $B(0,N)$, i.e.~of the event $|B(0,N) \backslash \wt{\cC}^u_N| > \nu \,|B(0,N)|$, was investigated. The spirit of the question was similar to that of ``phase separation and the emergence of a macroscopic Wulff shape'' for Bernoulli percolation or for the Ising model, see for instance \cite{Cerf00}, \cite{Bodi99} (however, with capacity replacing perimeter in the exponential costs). Specifically it was shown in Theorem 3.1 of \cite{Szni}  that when the vacant set $\cV^u$ is in the strongly percolating regime (i.e.~when $0 < u< \ov{u}$, which informally corresponds to a regime of local presence and uniqueness for the infinite cluster of $\cV^u$), one has
\begin{equation}\label{0.11}
\limsup\limits_N \; \mbox{\f $\dis\frac{1}{N^{d-2}}$} \;\log \IP\big[|B(0,N) \backslash \wt{\cC}^u_N| > \nu\,|B(0,N)|\big] \le - \mbox{\f $\dis\frac{1}{d}$} \; (\sqrt{\ov{u}} - \sqrt{u})^2 \,{\rm cap}_{\IR^d} (B_{\wt{\nu}})
\end{equation}
where $\wt{\nu} = 2^d \nu$, $B_{\wt{\nu}}$ is a Euclidean ball of volume $\wt{\nu}$, and ${\rm cap}_{\IR^d} (B_{\wt{\nu}})$ its Brownian capacity (see for instance p.~58 of \cite{PortSton78}, for its definition). In addition, when $\wt{\nu} < \o_d$ ($=$ the volume of a Euclidean ball or radius $1$), it was shown in Theorem 3.2 of \cite{Szni} that:
\begin{equation}\label{0.12}
\liminf\limits_N \; \mbox{\f $\dis\frac{1}{N^{d-2}}$} \; \log \IP\big[ |B(0,N) \backslash \wt{\cC}^u_N| > \nu\,|B(0,N)|\big] \ge - \frac{1}{d} \;(\sqrt{u}_{**} - \sqrt{u})^2 \,{\rm cap}_{\IR^d} (B_{\wt{\nu}}),
\end{equation}
where $u_{**}$ ($\ge u_* \ge \ov{u}$) is the threshold for the strongly non-percolative behavior of $\cV^u$ (informally $u > u_{**}$ corresponds to the regime where the probability that the cluster of the origin in $\cV^u$ contains a point at distance $L$ decays rapidly with $L$).  It is plausible, but presently open, that $\ov{u} = u_* = u_{**}$ (the corresponding equalities for level-set percolation of the Gaussian free field have recently been shown in \cite{DumGosRodrSev}). In this case the right members in (\ref{0.11}) and (\ref{0.12}) coincide with $ - \frac{1}{d}\;(\sqrt{u}_{*} - \sqrt{u})^2 \,{\rm cap}_{\IR^d} (B_{\wt{\nu}})$. Here, we show in Proposition \ref{prop6.5} that for $D = [-1,1]^d$,
\begin{equation}\label{0.13}
\begin{array}{l}
\mbox{when $0 < u < u_*$, and $\wt{\nu} = |D| \,\nu < \o_d$, if one successively let $R$ and $r$ tend to}
\\
\mbox{infinity, the right member of (\ref{0.10}) tends to $-\frac{1}{d} \; (\sqrt{u}_* - \sqrt{u})^2 \,{\rm cap}_{\IR^d} (B_{\wt{\nu}})$}.
\end{array}
\end{equation}

\n
Thus, the plausible but presently unproven equalities $\ov{u} = u_* = u_{**}$ have a remarkable consequence. When $u<u_*$ and $|D| \,\nu < \o_d$, if one looks at the asymptotic rate of exponential decay of the probability that a proportion bigger than $\nu$ of sites in $B(0,N)$ have the $|\cdot |_\infty$-ball of radius $r$ around them disconnected by $\cI^u$ within sup-distance $R$, the following happens: 

\n
in the limit where $R$ and $r$ successively tend to infinity, this rate recovers the asymptotic rate of exponential decay found in \cite{Szni} for the probability that the thickened component $\wt{\cC}^u_N$ leaves a ``hole'' $B(0,N) \backslash \wt{\cC}^u_N$, which occupies a proportion bigger than $\nu$ of the sites of $B(0,N)$. (This hole is then nearly spherical, see Remark 3.1 3) of \cite{Szni}.)

\medskip
The statement (\ref{0.13}) brings some heuristic insight into the role of the thickening of $\cC^u_N$, and one can naturally wonder about what happens in the absence of thickening, i.e.~when one considers $B(0,N) \backslash \cC^u_N$ in place of $B(0,N) \backslash \wt{\cC}^u_N$. Is it the case that for $0 < u < u_*$ and $\IP[0 \Vu \infty] < \nu < 1$, one has (with $D = [-1,1]^d$ and $\theta_0(a) = \IP[ 0 \Va \infty]$)
\begin{equation}\label{0.14}
\begin{array}{l}
\lim\limits_N \; \mbox{\f $\dis\frac{1}{N^{d-2}}$} \;\log \IP \big[ |B(0,N) \backslash \cC^u_N| > \nu\, |B(0,N)| \big] 
\\
= - \inf\Big\{  \fd \;\dis\int_{\IR^d} |\nabla \varphi |^2 \,dz;\;\varphi \ge 0, \; \varphi \in C^\infty_0 (\IR^d), \;\strokedint_D  \theta_0 \big((\sqrt{u} + \varphi)^2\big)\,dz > \nu \Big\} ?
\end{array}
\end{equation}
We refer to Remark \ref{rem6.6} 2) for more on this topic. We also refer to Remark \ref{rem6.6} 1) where (\ref{0.13}) is used to produce examples for which the right member of (\ref{0.8}) is not a convex function of $\nu$.

\medskip
Also, it is perhaps useful to describe the link between the results of the present article and that of \cite{LiSzni15}. In \cite{LiSzni15} a large deviation principle was derived for the occupation time profile  $\frac{1}{N^d} \;\sum_{x \in D_N} L_{x,u} \;\delta_{\frac{x}{N}}$(viewed as a positive measure), when $D$ is a closed box in $\IR^d$. Recast in the present set-up, via approximation, this result amounts, in essence, to the derivation of large deviation principles, as $N$ goes to infinity, for the random vectors made of the averages of $L_{x,u}$ over boxes $D_N^1, \dots, D_N^m$ defined as in (\ref{0.5}) (with $D^1 \dots, D^m$ arbitrary closed boxes in $\IR^d$). In the present framework this only involves the {\it linear} local function $F_0(\ell) = \ell$ (with $R = 0$, see above (\ref{0.2})). On the other hand, here, we consider excess deviations for the average over $D_N$ corresponding to a {\it non-linear}, {\it finite-range}, non-decreasing local function $F$.

\medskip
We now provide a brief overview of the proofs. An important feature underpinning the large deviation asymptotics analyzed in this work is a loose decomposition of random interlacements, when looked at in a box $B$ of large size $L$, into an {\it undertow} that heuristically reflects a local value for the box of the level of the interlacement, and a {\it wavelet} part that captures local information, so that between well-separated such boxes the mutual dependence is ``channeled'' via the ``undertow'', but the ``wavelet parts'' are independent. This type of decomposition is more delicate for random interlacements than in the case of the Gaussian free field, see for instance Section 5 of \cite{Szni15}, or Section 5 of \cite{Szni}. Incidentally, it would be of interest to see if such decompositions could be supplied in further models  with continuous symmetry. Here, we access the above mentioned decomposition in Section 2, with the coupling based on the soft local time technique of \cite{PopoTeix15}, and more specifically for our purpose in the form developed in \cite{ComeGallPopoVach13}. Informally, the wavelet part in a box $B$ corresponds to the collection of excursions in the trajectories of the full interlacements from $B$ to the boundary of a large concentric box, and the undertow to the number of such excursions in the interlacements up to level $u$, scaled by the capacity of $B$. Actually, we derive in Proposition \ref{prop2.2} of Section 2 super-polynomial decay estimates that show that in a box $B$ of large size $L$ various quantities, in particular averages involving the local functions $F$, concentrate near a value dictated by the undertow in the box $B$. This fact brings into play the notion of {\it good-boxes}, which also provides flexibility in tracking the undertow, see Lemma \ref{lem2.4}. Combined with the independence of the ``wavelet parts'', these super-polynomial decay estimates lead to super-exponential estimates in Section 3. In essence, once we choose the size $L$ of boxes $B$ so that there are approximately $N^{d-2}/\log N$ such boxes in $B(0,N)$, see (\ref{4.8}), except on an event $\cB$, see (\ref{3.10}), of negligible probability for our purpose, most boxes $B$ meeting $B(0,N)$ are good. In good boxes we have flexibility on how to track the undertow, and quantities involving the wavelet parts concentrate near a value dictated by the undertow. These features are used both for the lower and upper bounds of the large deviation asymptotics in Sections 4 and 5 respectively.

\medskip
For the lower bound we use the change of probability method, and replace the interlacements (governed by $\IP$), by {\it tilted interlacements} (governed by $\wt{\IP}_N$), as introduced in \cite{LiSzni14}. In essence, the tilted interlacements correspond to random interlacements with a spatially slowly varying parameter $(\sqrt{u} + \varphi (\frac{x}{N}))^2$, where $\varphi$ should be thought as coming from the variational formulas in (\ref{0.8}). Actually, we do not directly apply the relative entropy inequality to bound $\IP[\cA_N]$ from below, but rather replace the event $\cA_N$ with the help of the super-exponential bounds of Section 3, by an event $\cA^1_N$, see (\ref{4.22}), which solely involve constraints on the average value of $L^u_\point$ over the various boxes $B$ inside $B(0,N)$, see (\ref{4.21}). Importantly, this change makes it tractable to check that $\cA^1_N$ is typical under $\wt{\IP}_N$, see Lemma \ref{lem4.5}.

\medskip
For the upper bound we use a coarse graining procedure. With the super-exponential estimates of Section 3, we first replace $\cA_N$ by a not too large collection of events that mostly cover $\cA_N$. Each such event imposes a specific constraint on the averages $\langle \ov{e}_B, L^u\rangle$ of the occupation time $L^u_\point$ with respect to the normalized equilibrium measure of the box $B$, simultaneously over a collection of well-separated boxes $B$ of size $L$ intersecting $B(0,N)$ and at mutual distance of order $KL$, see (\ref{5.12}) and Proposition \ref{prop5.3}. One then derives bounds on the probability of all such events. For this purpose, a central step is to introduce for each such event a well-adapted linear combination of the $\langle \ov{e}_B, L^u\rangle$, and bound its exponential moment. This is performed in Proposition \ref{prop5.6} and crucially relies on the bounds developed in Proposition \ref{prop5.4}. Dealing with an ``excess event'' plays an important role, see Remark \ref{rem5.5}. The desired bounds on the probability of the events in the coarse graining now come via an exponential Chebyshev inequality and appear in Proposition \ref{prop5.6}. They involve the discrete Dirichlet forms of certain discrete superharmonic functions. Then, one has to relate these exponential bounds involving the discrete Dirichlet forms and a quite porous constraint, see (\ref{5.48}), with the continuous variational problems, similar to (\ref{0.8}), where the constraint involving $\theta$ is present in full. This is performed in Proposition \ref{prop5.7}. The combination of the lower bound in Theorem \ref{theo4.2} and the upper bound in Theorem \ref{theo5.1} acquires its full strength in Corollary \ref{cor5.9}, see also (\ref{0.8}). Several properties of the variational problem in (\ref{0.8}) are discussed in Remark \ref{rem5.10}. The upper bounds of Theorem \ref{theo5.1} also lead to upper bounds in the case of the occupation of the simple random walk (that formally corresponds to the level $u=0$) in Corollary \ref{cor5.11}, see also (\ref{0.8a}).

\medskip
We will now describe the organization of this article. Section 1 recalls various facts about the simple random walk, random interlacements, and tilted interlacements. In Section 2 we specify the assumptions on the local functions, and recall some facts about the coupling attached to the soft local-time technique. We introduce the notion of {\it good boxes} and prove the important super-polynomial estimates in Proposition \ref{prop2.2}. In addition, Lemma \ref{lem2.4} provides some useful features of good boxes. In Section 3 we obtain super-exponential estimates in Proposition \ref{prop3.1} that show that up to a negligible probability for our purpose, most boxes that we have to consider, are good. Section 4 is devoted to the proof of Theorem \ref{theo4.2} that contains our main lower bound on the exponential decay of $\IP[\cA_N] $. With the help of the super-exponential estimates of Section 3, we first replace $\cA_N$ by $\cA^1_N$ in Proposition \ref{prop4.3}. Then, the exponential lower bound on $\IP[\cA_N]$ is derived with the help of the change of probability method and tilted interlacements in Proposition \ref{prop4.4}. In Section 5 we derive the main upper bound on the exponential decay of $\IP[\cA_N] $ in Theorem \ref{theo5.1}. The coarse graining procedure relies on the super-exponential estimates of Section 3 and appears in Proposition \ref{prop5.3}. The exponential Chebyshev bound on the probability of events entering the coarse graining is provided in Proposition \ref{prop5.6}, and heavily relies on the controls derived in Proposition \ref{prop5.4}. The last main step linking the upper bound on the exponential rate of decay obtained in Proposition \ref{prop5.6}  to a constrained variational problem in the continuum appears in Proposition \ref{prop5.7}. Then, Corollary \ref{cor5.9} combines Theorems \ref{theo4.2} and \ref{theo5.1}, and proves (\ref{0.8}). Various properties of the constrained variational problem and the corresponding minimizers appear in Remark \ref{rem5.10}. An application to the simple random walk is given in Corollary \ref{cor5.11}. Section 6 contains applications to the interlacement sausage in Theorem \ref{theo6.1}, and to finite pockets in the vacant set in Theorem \ref{theo6.3}. Proposition \ref{prop6.5} provides a link to the results concerning ``macroscopic holes'' from \cite{Szni}, further explained in Remark \ref{rem6.6}.

\medskip
Finally, let us state our convention about constants. Throughout the article we denote by $c,c',\wt{c}$ positive constants changing from place to place that simply depend on the dimension $d$. Numbered constants $c_0,c_1,c_2$ refer to the value corresponding to their first appearance in the the text. Dependence on additional parameters appears in the notation.

\bigskip
\section{Notation and some useful facts}
\setcounter{equation}{0}

In this section we introduce some notation and recall various results concerning continuous time simple random walks, random interlacements, and tilted interlacements on $\IZ^d$, $d\ge3$.

\medskip
We begin with some notation. For $s,t$ real numbers, we write $s \wedge t$ and $s \vee t$ for the minimum and the maximum of $s$ and $t$, and denote by $[s]$ the integer part of $s$, when $s$ is non-negative. We write $| \cdot |_2$ and $| \cdot |_\infty$ for the Euclidean and the supremum norms on $\IR^d$. Given $x,y$ in $\IZ^d$, we write $x \sim y$ to state that $x$ and $y$ are neighbors, i.e.~$|y - x|_2 = 1$. For $x$ in $\IZ^d$ and $r \ge 0$, we let $B(x,r) = \{y \in \IZ^d$; $|y - x|_\infty \le r\}$ stand for the closed ball of radius $r$ and center $x$ for the sup-norm. When $z \in \IR^d$ and $r \ge 0$, we write $B_{\IR^d}(z,r)$ for the corresponding closed-ball in $\IR^d$ with center $z$ and radius $r$ in the sup-norm.  Given $A, A'$ subsets of $\IZ^d$, we denote by $d(A, A') = \inf\{|x - x'|_\infty$; $x \in A, x' \in A'\}$ the mutual distance between $A$ and $A'$. When $A = \{x\}$, we write $d(x,A')$ in place of $d(\{x\}, A')$ for simplicity. We say that a subset $B$ of $\IZ^d$ is a box when it is a translate of some set $\IZ^d \cap [0,L)^d$, with $L \ge 1$. We often write $[0,L)^d$ in place of $\IZ^d \cap [0,L)^d$, when this causes no confusion. When $A$ is a subset of $\IZ^d$, we let $|A|$ stand for the cardinality of $A$ and write $A \subset \subset \IZ^d$ to state that $A$ is a finite subset of $\IZ^d$. We denote by $\partial A = \{y \in \IZ^d \backslash A$; $\exists x \in A$, $|y - x| = 1\}$ and $\partial_i A = \{x \in A$; $\exists y \in \IZ^d \backslash A$; $|y - x| = 1\}$ the boundary and the internal boundary of $A$. Given $h, h'$ functions on $\IZ^d$, we write $h_+ = \max \{h,0\}$ and $h_- = \max\{- h,0\}$ for the positive and negative part of $h$ as well as $\langle h, h'\rangle = \sum_{x \in \IZ^d} h(x) \,h'(x)$ when the sum is absolutely convergent. We also use the notation $\langle \rho, f \rangle$ for the integral of a function $f$ (on an arbitrary space) with respect to a measure (on the same space) when this quantity is meaningful.

\medskip
Given a non-empty $U \subseteq \IZ^d$, we write $\Gamma(U)$ for the space of right-continuous, piecewise constant functions from $[0,\infty)$ to $U \cup \partial U$, with finitely many jumps on any finite time interval that remain constant after their first visit to $\partial U$. The space $\Gamma(U)$ is endowed with the canonical coordinate process denoted by $(X_t)_{t \ge 0}$. If $A \subseteq \IZ^d$, we will refer to $H_A = \inf\{t \ge 0; X_t \in A\} (\le \infty)$ as the entrance time of $X$ in $A$, and $T_A = \inf\{t \ge 0$; $X_t \notin A\}( \le \infty)$, as the exit time of $X$ from $A$. For $U \subset \subset \IZ^d$, the space $\Gamma (U)$ will be convenient to record certain excursions in the interlacement trajectories, see (\ref{1.21}). When $U = \IZ^d$, we view the law $P_x$ of the simple random walk with unit jump rate starting at $x$, as a measure on $\Gamma(\IZ^d)$, and denote by $E_x$ the corresponding expectation. We write $g_U(x,y)$, $x,y \in \IZ^d$, for the Green function of the simple random walk killed outside $U \subset \IZ^d$ (it is symmetric and vanishes if $x$ or $y$ does not belong to $U$). When $U = \IZ^d$ we drop the subscript and simply write $g(x,y)$ in place of $g_{\IZ^d}(x,y)$. For $A \subset \subset  \IZ^d$, we denote by $e_A$ the equilibrium measure of $A$ (i.e.~$e_A(x) = 0$ if $x \notin A$ and $e_A(x) = P_x \,[$after its first jump $X_\point$ never enters $A]$ if $x \in A$). We let $h_A$ stand for the equilibrium potential of $A$, so that for all $x$ in $\IZ^d$, $h_A(x) = \sum_{y \in \IZ^d} g(x,y) \,e_A(y) = P_x[H_A < \infty]$. We write ${\rm cap}(A)$ for the capacity of $A$, i.e.~the total mass of $e_A$. When in addition $A$ is not empty, we write $\ov{e}_A$ for the normalized equilibrium measure of $A$ and $m_A$ for the normalized counting measure on $A$.

\medskip
We denote by $L$ the generator of the simple random walk on $\IZ^d$, namely
\begin{equation}\label{1.1}
Lh(x) = \mbox{\f $\dis\frac{1}{2d}$} \; \dsl_{y \sim x} \; h(y) - h(x), \; \mbox{for $h$: $\IZ^d \rightarrow \IR$ and $x \in \IZ^d$}.
\end{equation}
The Dirichlet form is denoted by
\begin{equation}\label{1.2}
\cE(h,h) = \fr \;\dsl_{y \sim x} \; \fd \;\big(h(y) - h(x)\big)^2 ( \le \infty), \; \mbox{for $h$: $\IZ^d \r \IR$}.
\end{equation}

\n
When $h_1,h_2$ are functions with finite Dirichlet form, $\cE(h_1,h_2)$ is defined by polarization. In addition, the resolvent operator is defined as
\begin{equation}\label{1.3}
G h(x) = E_x\Big[\dis\int^\infty_0 h(X_s)\,ds\Big] = \dsl_{y \in \IZ^d} g(x,y)\,h(y),\; \mbox{for $x \in \IZ^d$},
\end{equation}
whenever the function $h$ is such that the last sum is absolutely convergent.

\medskip
The exponential bounds in the lemma below, for the time spent by the simple random walk in a box $B$, and for the number of sites in $B$ visited by the simple random walk, will be used in the proof of the super-polynomial decay in Proposition \ref{prop2.2} of Section 2, and in the implementation of the change of probability method to derive the main asymptotic lower bound, in Lemma \ref{lem4.5} of Section 4. We recall the convention concerning constants stated at the end of the Introduction.

\begin{lemma}\label{lem1.1}
There is a constant $c_0$ such that for any $L \ge 1$ integer and $B = z + [0,L)^d$, $z \in \IZ^d$, one has
\begin{align}
&E_x \Big[\exp\Big\{ \mbox{\f $\dis\frac{c_0}{L^2}$} \;\dis\int^\infty_0 1\{X_s \in B\}\,ds\Big\}\Big] \le 2, \; \mbox{for all $x \in \IZ^d$, and} \label{1.4}
\\[1ex]
&E_x \Big[\exp\Big\{ \mbox{\f $\dis\frac{c_0}{L^2}$} \; | {\rm range}(X) \cap B|\Big\}\Big] \le 2, \; \mbox{for all $x\in \IZ^d$} \label{1.5}
\end{align}
(with ${\rm range}(X)$ the set of points visited by the trajectory $X$).
\end{lemma}

\begin{proof}
The proof of (\ref{1.4}) is classical: it follows from Kac's moment formula, see \cite{MarcRose06}, p.~74, 116, and the estimate $\sup_{y \in \IZ^d} (G1_B)(y) \le c \,L^2$. As for (\ref{1.5}), it follows by considering the discrete skeleton of the walk, and using Jensen's inequality when integrating out the exponential variables describing the time spent between consecutive jumps of the walk, to dominate the left member of (\ref{1.5}) by the left member of (\ref{1.4}).
\end{proof}

In conjunction with tilted interlacements (see \cite{LiSzni14} and Lemma \ref{lem1.2} below), we will consider in Section 4 certain random walks with drifts that we now describe. We are given
\begin{equation}\label{1.6}
\mbox{$\wt{f}: \IZ^d \rightarrow (0,\infty)$ such that $\wt{f} = 1$ except on a finite set},
\end{equation}
(informally $\wt{f}\,^{\!2}$ can be viewed as the density with respect to the counting measure on $\IZ^d$ of a reversibility measure for the random walk with drift corresponding to the generator $\wt{L}$ in (\ref{1.10}) below), and we introduce the finitely supported function on $\IZ^d$:
\begin{equation}\label{1.7}
\wt{V}(x) = - \mbox{\f $\dis\frac{L \wt{f}}{\wt{f}}$}\;(x), \; \mbox{for $x \in \IZ^d$},
\end{equation}
as well as the measure $\wt{\lambda}$ on $\IZ^d$ such that
\begin{equation}\label{1.8}
\wt{\lambda}(x) = \wt{f}^2(x), \;\mbox{for $x \in \IZ^d$}.
\end{equation}

\n
For $h_1,h_2$ functions on $\IZ^d$ we write $\langle h_1,h_2 \rangle_{\wt{\lambda}}$ for $\sum_{x \in \IZ^d} \,h_1(x)\,h_2(x)\,\wt{\lambda}(x)$, when this sum is absolutely convergent (note that when $\wt{f}$ is identically equal to $1$ the quantity $\langle h_1,h_2\rangle_{\wt{\lambda}}$ coincides with $\langle h_1,h_2\rangle$). 

\medskip
The walks under consideration correspond to the measures on $\Gamma(\IZ^d)$:
\begin{equation}\label{1.9}
\wt{P}_x = \mbox{\f $\dis\frac{1}{\wt{f}(x)}$} \; \exp\Big\{ \dis\int^\infty_0 \wt{V}(X_s) \,ds\Big\} \,P_x, \; \mbox{for $x \in \IZ^d$}.
\end{equation}

\n
By Corollary 1.3 of \cite{LiSzni14}, one knows that $\wt{P}_x$ is a probability measure for each $x$ in $\IZ^d$, and that under $\wt{P}_x$, $(X_t)_{t \ge 0}$ is a reversible Markov chain on $\IZ^d$ with reversible measure $\wt{\lambda}$ and its semi-group on $L^2(\wt{\lambda})$ has the bounded generator (with $h \in L^2(\wt{\lambda})$):
\begin{equation}\label{1.10}
\wt{L} \,h = \mbox{\f $\dis\frac{1}{\wt{f}}$} \; L\,(\wt{f}\,h) + \wt{V} \,h, \;\mbox{so that} \; \wt{L} \,h(x) = \fd \;\dsl_{y \sim x} \;\mbox{\f $\dis\frac{\wt{f}(y)}{\wt{f}(x)}$} \;\big(h(y) - h(x)\big), \; \mbox{for $x \in \IZ^d$}.
\end{equation}

\n
We let $\wt{E}_x$ stand for the $\wt{P}_x$-expectation and also consider the resolvent operator:
\begin{equation}\label{1.11}
\wt{G} \,h(x) = \wt{E}_x\Big[\dis\int^\infty_0 h(X_s)\,ds\Big], \; x \in \IZ^d,
\end{equation}

\n
for $h$: $\IZ^d \rightarrow \IR$ such that the expectation corresponding to $|h|$ is finite. When $\wt{f} = 1$ identically, $\wt{P}_x$ coincides with $P_x$ and $\wt{G}$ coincides with $G$ in (\ref{1.3}). The resolvent operator $\wt{G}$ will enter the formula for the Laplace transform of the occupation times of tilted interlacements in Lemma \ref{lem1.2} below.

\medskip
We now turn to continuous-time random interlacements on $\IZ^d$, $d \ge 3$. We introduce some notation, recall some properties, but mostly refer to Section 1 of \cite{Szni17} for further details. As an aside, random interlacements have deep links with the Gaussian free field, and we refer to \cite{DrewPrevRodr} where the role of random interlacements as efficient highways in the sub-level sets of the Gaussian free field is highlighted. We let $\Omega$ stand for the sample space on which random interlacements are defined, $\IP$ stand for the probability governing random interlacements, and $\IE$ for the corresponding expectation (see (1.36) of \cite{Szni17} for details).
The occupation time $L^u_x$ at site $x$ and level $u$ records the total time spent at $x$ by the trajectories with label at most $u$ in the cloud $\omega$ of interlacement trajectories. When $V$ is a finitely supported function on $\IZ^d$, we write $\langle L^u, V\rangle$ for $\sum_{x \in \IZ^d} L^u_x \,V(x)$, and likewise when $A \subset \subset \IZ^d$ is not empty we write $\langle m_A, L^u \rangle = \frac{1}{|A|} \, \sum_{x \in A} L^u_x$. One knows that
\begin{equation}\label{1.12}
\IE [L^u_x] = u, \; \mbox{for all $x \in \IZ^d$ and $u \ge 0$}.
\end{equation}

\n
In addition, if $V$: $\IZ^d \rightarrow \IR$ is finitely supported and such that $\sum_{n \ge 0} (G|V|)^n 1$ is a finite function (where $|V|$ is to be understood as the multiplication operator by the function $|V|$), one knows that (see Theorem 2.1 of \cite{Szni12d})
\begin{equation}\label{1.13}
\IE[\exp\{ \langle L^u, V \rangle \}] = \exp\big\{u \big\langle V, \dsl_{n \ge 0} (GV)^n 1\big\rangle\big\}.
\end{equation}

\n
Actually, a more general identity than (\ref{1.13}) is known under the assumption of ``finiteness of the gauge'', see (2.2) and (1.10) of \cite{LiSzni15}. It has close links to the Dirichlet form, see Corollary 4.2 of \cite{LiSzni15}. We also refer to Lemma 3.1 of \cite{ChiaNitz19} for perturbation identities for the expression under the exponential in the right member of (\ref{1.13}). Formula (\ref{1.13}) will suffice for our purpose here, and we will use it recurrently throughout this article.

\medskip
We now briefly recall some facts about {\it tilted interlacements} from \cite{LiSzni14}. We will mainly use tilted interlacements in Section 4, when implementing the change of probability method to derive our main asymptotic lower bound on $\IP[\cA_N]$ (with $\cA_N$ as in (\ref{0.7})). With $\wt{f}$ and $\wt{V}$ as in (\ref{1.6}), one knows from the proof of (2.7) of \cite{LiSzni14} that 
\begin{equation}\label{1.14}
\wt{\IP} = e^{\langle L^u, \wt{V}\rangle} \IP \;\mbox{is a probability measure}.
\end{equation}
We recall the notation from (\ref{1.8}), (\ref{1.11}). In Lemma \ref{lem4.5} of Section 4 we will need the following

\begin{lemma}\label{lem1.2}
If $V$ is finitely supported on $\IZ^d$ and $\sup_{x \in \IZ^d} (\wt{G} |V|) (x) < 1$, then for $u \ge 0$
\begin{equation}\label{1.15}
\wt{\IE} [e^{\langle L^u, V\rangle}] = \exp\{u \langle V, (I - \wt{G} V)^{-1} 1 \rangle_{\wt{\lambda}}\}
\end{equation}
($\wt{\IE}$ stands for the $\wt{\IP}$-expectation).
\end{lemma}

\begin{proof}
Consider $\wt{U} \supseteq \{\wt{f} \not= 1\} \cup \{V \not= 0\}$ a large enough finite subset of $\IZ^d$ such that $d (\partial_i \,\wt{U}, \{\wt{f} \not= 1\}) \ge 2$. As in (2.13) of \cite{LiSzni14}, one finds that
\begin{equation}\label{1.16}
\begin{split}
\wt{\IE} [e^{\langle L^u, V\rangle}] & = \exp \Big\{ u \dsl_{x \in \wt{U}} e_{\wt{U}}(x) \, \big(\wt{E}_x\big[e ^{\int^\infty_0 V(X_s)ds}\big] - 1\big)\Big\}
\\
& =\exp \Big\{ u \dsl_{x \in \wt{U}} e_{\wt{U}}(x) \, \dsl_{k \ge 1} [(\wt{G} V)^k 1](x)\Big\}.
\end{split}
\end{equation}

\n
Now, by a similar identity as in (2.23) of \cite{LiSzni14}, we see that
\begin{equation}\label{1.17}
\dsl_{x \in \wt{U}} e_{\wt{U}}(x) \,\dsl_{k \ge 1} [(\wt{G} V)^k 1](x) = \big\langle V,\dsl_{\ell \ge 0} (\wt{G} V)^\ell 1\big\rangle_{\wt{\lambda}} = \langle V, (I - \wt{G} V)^{-1} 1\rangle_{\wt{\lambda}}.
\end{equation}

\n
The claim (\ref{1.15}) now follows from (\ref{1.16}) and (\ref{1.17}). This proves Lemma \ref{lem1.2}.
\end{proof}

As a next topic, given a scale $L$, an important role will be played by the decomposition of random interlacements into ``wavelet'' and ``undertow'' components, respectively carrying ``local'' and ``longer range'' information (with respect to the scale $L$). With this in mind, we consider positive integers $L \ge 1$ and $K \ge 100$, and define
\begin{equation}\label{1.18}
\IL = L \, \IZ^d (\subseteq \IZ^d).
\end{equation}

\n
Later, in Sections 4 and 5, we will choose $L$ of order $N^{\frac{2}{d}} (\log N)^{\frac{1}{d}}$, see (\ref{4.8}) so that $\IL \cap B(0,N)$ contains an order of $N^{d-2} / (\log N)$ points, when $N$ is large.

\medskip
For each $z \in \IL$ we consider the boxes in $\IZ^d$:
\begin{equation}\label{1.19}
B_z = z + [0,L)^d \subseteq U_z = z + [-K L + 1, KL - 1)^d.
\end{equation}

\n
Given $z \in \IL$, $B = B_z$, $U = U_z$, as well as $u \ge 0$, we write (see (1.42), (2.14) of \cite{Szni17})
\begin{align}
x_B  = &\;z, \label{1.20}
\\[1ex]
N_u(B)  = &\; \mbox{the total number of excursions from $B$ to $\partial U$ in all trajectories of}\label{1.21}
\\[-0.5ex]
& \; \mbox{the interlacements with labels at most $u$},\nonumber
\end{align}
and denote by
\begin{equation}\label{1.22}
\begin{array}{l}
\mbox{$Z^B_\ell, \ell \ge 1$, the successive excursions going from $B$ to $\partial U$ in the}
\\
\mbox{trajectories of the interlacements}
\end{array}
\end{equation}
(the $Z^B_\ell, \ell \ge 1$, are $\Gamma(U)$-valued variables, where $\Gamma(U)$ has been defined above (\ref{1.1})).

\medskip
In essence, the $(Z^B_\ell)_{\ell \ge 1}$, play the role of ``wavelet'' variables capturing local information, and the $N_u(B)$ (or more precisely the $N_u(B)/{\rm cap}(B)$) the role of an ``undertow'' in the informal scheme described above. In the next two sections we will develop alternative ways of ``tracking the undertow'', see Remark \ref{rem2.5}.

\medskip
Given $a \ge 0$, $x$ in $\IZ^d$, and $B, U$ as above, we write
\begin{equation}\label{1.23}
L^B_{a,x} = \dsl_{1 \le \ell \le a} \;\dis\int_0^{T_U (Z^B_\ell)} 1\{Z^B_\ell(s) = x\} \,ds \; \mbox{($=0$, when $a < 1$ or $x \notin U$)},
\end{equation}

\n
where $T_U(Z^B_\ell)$ denotes the exit time of $Z^B_\ell$ from $U$. So, for $x$ in $U$, $L^B_{a,x}$ records the time spent at $x$ by the first $[a]$ excursions from $B$ to $\partial U$ in the interlacements. When $\rho$ is some measure on $\IZ^d$, we use the notation
\begin{equation}\label{1.24}
\langle \rho, L^B_a \rangle \; = \; \dsl_{x \in \IZ^d} \;\rho(x) \, L^B_{a,x} \quad \mbox{(a finite sum by (\ref{1.23}))}.
\end{equation}
We will also recurrently use the bounds (see (2.16), p.~53 of \cite{Lawl91})
\begin{equation}\label{1.25}
c\,L^{d-2} \le {\rm cap} ([0,L)^d) \le c'\, L^{d-2}, \; \mbox{for $L \ge 1$}.
\end{equation}

\section{Good boxes and coupling of excursions}
\setcounter{equation}{0}

In this section, we specify the type of local functions of the field of occupation times of random interlacements that enter the large deviation estimates to be proven in Sections 4 and 5. We then recall some facts about the soft local time technique, as developed in \cite{PopoTeix15}, and, actually, more specifically for our purpose, in \cite{ComeGallPopoVach13}, see also Section 4 of \cite{Szni17}. It provides a coupling of the excursions $Z^B_\ell, \ell \ge 1$, when $B$ runs over a finite collection of well-separated boxes (of side-length $L$), with independent excursions (see below (\ref{2.15})). This coupling grants us with the tool to prove the rarity of bad-boxes, when combined with the super-polynomial bounds of Proposition \ref{prop2.2}, which rely on concentration estimates. The actual definition of good and bad boxes $B$ is stated in (\ref{2.76}) towards the end of the section. It pertains to the good or bad behavior of the excursions $Z^B_\ell, \ell \ge 1$, as measured via certain functionals depending over the first $\alpha \,{\rm cap}(B)$ excursions, when $\alpha$ ranges over a finite subset $\Sigma$ of $(0,\infty)$, see (\ref{2.11}). It involves information on the ``wavelet part'' of the interlacement, in the terminology from the end of Section 1. In Lemma \ref{lem2.4}, we will see how good boxes provide alternative ways to track the ``undertow'' in the interlacements, see also Remark \ref{rem2.5}. Combined with the super-exponential bounds of Proposition \ref{prop3.1} in the next section, this flexibility will be an important asset, allowing the replacement of the event $\cA_N$ of (\ref{0.7}) by more convenient events, see in particular (\ref{4.21}), (\ref{4.22}), and (\ref{5.12}), (\ref{5.15}), in the respective derivations of the asymptotic lower bounds, and upper bounds on $\IP[\cA_N]$, in Sections 4 \hbox{and 5}.

\medskip
We first record the requirements that we will impose on the local functions of the field of occupation times of random interlacement. There is an integer $R \ge 0$, controlling the range, and a non-negative function $F$ on $[0,\infty)^{B(0,R)}$ such that
\begin{equation}\label{2.1}
\left\{ \begin{array}{rl}
{\rm i)} & \mbox{$F$: $[0,\infty)^{B(0,R)} \rightarrow [0,\infty)$ is measurable and non-decreasing in}\\
&\quad \;\; \mbox{each variable},
\\[1ex]
{\rm ii)} & F(0) = 0,
\\[2ex]
{\rm iii)} & \mbox{for $\ell, \ell' \in [0,\infty)^{B(0,R)}, \;F(\ell + \ell') \le F(\ell) + c(F)  \big(1_{ \{ \ell' \neq 0 \}}+ \sum_{|x|_\infty \le R} \ell'_x$\big)}.
\end{array}\right.
\end{equation}

\n
As we immediately note: 
\begin{equation*}
\begin{array}{l}
\mbox{under (\ref{2.1}), $u \ge 0 \r \theta(u) = \IE\big[F\big((L^u_y)_{|y|_\infty \le R}\big)\big] \ge 0$ is a non-decreasing}
\\
\mbox{continuous function, and $\theta(0) = 0$.}
\end{array}
\end{equation*}

\medskip\n
Indeed $\theta$ is finite by (\ref{1.12}) and ii), iii), as well as non-decreasing by i) and $\theta(0) = 0$ by ii). In addition by iii), for $u \ge 0$, $h > 0$, $\theta(u+h) \le \theta(u) + c(F) |B(0,R)| (\IP[0 \in \cI^h] + \IE[L_0^h]) = \theta(u) + c(F) |B(0,R)| (1 - e^{-h/g(0,0)} + h)$, and the continuity of $\theta$ follows. 

\medskip
We view $R$ as a function of $F$, so, $c(F)$ in iii) is a positive constant depending on $d$ and $F$ (and hence, possibly on $R$). Incidentally, iii) is automatic when $F$ is bounded: one simply chooses $c(F) = \|F\|_\infty$. Note that $F$ naturally extends to a non-decreasing function from $[0,\infty]^{B(0,R)}$ into $[0,\infty]$ (there is no ambiguity in the order in which limits are taken) and that $F(\infty, \dots , \infty) = \|F\|_\infty$. We will use the shorthand notation $F((L^u_\point))$ in place of $F((L^u_y)_{|y|_\infty \le R})$, and $F((L^u_{x+ \point}))$ in place of $F((L^u_{x+y})_{|y|_\infty \le R})$.

\medskip
We always assume (\ref{2.1}), and sometimes also require the additional condition
\begin{equation}\label{2.2}
\begin{array}{l}
\mbox{$\theta(\cdot)$ is strictly increasing},
\\
\mbox{(or equivalently under (\ref{2.1}): $F$ is non-constant).}
\end{array}
\end{equation}

\n
To see the above mentioned equivalence note that $\theta$ strictly increasing clearly implies $F$ non-constant. Conversely, pick $a, b > 0$ such that $\ell_y \ge a$ for all $y \in B(0,R)$ ensures $F(\ell) \ge b$. Then, for $u \ge 0$, $h > 0$, $\theta (u + h) \ge \theta(u) + \IE [F((L_\point^{u + h}))$, $L^u_y = 0$ and $L^{u+h}_y \ge a$ for all $y \in B(0,R)] > \theta(u)$, and $\theta$ is strictly increasing.

\bigskip\n
\begin{example}\label{exam2.1} \rm ~

\medskip\n
1) With $R = 0$, we set (and keep the notation $F_0$ for the choice (\ref{2.3})):
\begin{equation}\label{2.3}
\mbox{$F_0(\ell) = \ell$ for $\ell \ge 0$, so that $\theta(u) = u$ for $u \ge 0$ (see (\ref{1.12}))}.
\end{equation}

\medskip\n
2) With $R = 0$, we set (this was (\ref{0.2}) in the Introduction)
\begin{equation}\label{2.4}
\mbox{$F(\ell) = 1\{ \ell > 0\}$, for $\ell \ge 0$, so that $\theta (u) = \IP[0 \in \cI^u] = 1 - e^{-u/g(0,0)}$, for $u \ge 0$}.
\end{equation}

\medskip\n
3) With $R \ge 0$, we set
\begin{equation}\label{2.5}
\left\{ \begin{array}{l}
F(\ell) = 1\Big\{ \dsl_{|y|_\infty \le R} \ell_y > 0\Big\}, \;\mbox{for $\ell \in [0,\infty)^{B(0,\ell)}$, so that}
\\
\theta(u) = \IP[\cI^u \cap B(0,R) \not= \emptyset] = 1 - e^{-u \,{\rm cap}(B(0,R))}, \;\mbox{for $u \ge 0$}.
\end{array}\right.
\end{equation}

\medskip\n
4) With $R \ge 0$, we set (this was (\ref{0.3}) in the Introduction)
\begin{equation}\label{2.6}
\left\{ \begin{split}
F(\ell) = &\;1\{\mbox{any path from $0$ to $S(0,R)$ in $B(0,R)$ meets a $y$ with $\ell_y > 0\}$}, 
\\
&\quad\; \mbox{for $\ell \in [0,\infty)^{B(0,R)}$, so that}
\\
\theta(u) =&\; \IP[0 \Vu S(0,R)], \;\mbox{for $u \ge 0$}. 
\end{split}\right.
\end{equation}

\medskip\n
5) With $0 \le r < R$ integers, we set
\begin{equation}\label{2.7}
\left\{ \begin{split}
F(\ell) = &\; 1\{\mbox{any path from $B(0,r)$ to $S(0,R)$ in $B(0,R)$ meets a $y$ with $\ell_y > 0\}$,} 
\\
&\quad\;  \mbox{so that}
\\ 
\theta(u) =&\; \mbox{$\IP[0$ does not belong to the $r$-neighborhood of the connected}
\\
&\quad\; \mbox{component of $S(0,R)$ in $\cV^u]= \IP[B(0,r) \Vu S(0,R)]$}.
\end{split}\right.
\end{equation}
\end{example}

This last example implicitly showed up above (\ref{0.10}) in the Introduction. It shares some flavor with the thickening of the component in the vacant set of the boundary of a large box as considered in (\ref{1.2}), (\ref{1.6}) of \cite{Szni}. Note that in all five examples above the function $\theta$ is analytic: by direct inspection in the first three examples, and in the case of the last two examples due to the fact that for any $A \subseteq B(0,R)$, $\IP[\cI^u \cap B(0,R) = A]$ depends analytically on $u$, see (\ref{2.17}) of \cite{Szni10}. In addition, in all five examples, the function $F$ is not identically equal to $0$ (actually, as $u \r \infty$, $\theta(u)$ tends to $\infty$ in the case of (\ref{2.3}) and to $1$ in the remaining four examples). Hence,  
\begin{equation}\label{2.8}
\mbox{in (\ref{2.3}) - (\ref{2.7}) the function $\theta$ is analytic and satisfies (\ref{2.2})}.
\end{equation}

\vspace{-3ex}
\hfill $\square$

\medskip
We now consider a local function $F$ (with range $R \ge 0$), satisfying (\ref{2.1}), as well as $L \ge 1$ and $K \ge 100$ integers (see above (\ref{1.18})), and for each $B = B_z$, $z \in \IL$, with corresponding $U = U_z$, we probe the occupation time left by the first $a$ excursions between $B$ and $\partial U$ in the interlacements by setting (see (\ref{1.23})):
\begin{align}
F^B_{a,x} & = F\big((L^B_{a,x+ \point})\big), \; \mbox{for $a \ge 0$, $x \in \IZ^d$, and}\label{2.9}
\\[1ex]
F^B_a & = \dsl_{x \in B} F^B_{a,x}, \;\mbox{for $a \ge 0$}. \label{2.10}
\end{align}

\n
We then introduce the main ingredient of what will enter the definition of a bad box in (\ref{2.76}) below. Informally, for $\alpha > 0$, a given fixed level, we have in mind that for large $L$, and $a$ of order $\alpha \,{\rm cap}(B)$, for a ``good box'', the spatial average $\frac{1}{|B|} F^B_{\alpha\,{\rm cap}(B)}$ should be close to $\theta(\alpha)$ (and Proposition \ref{prop2.2} will make this statement more precise). So, with this idea in mind, given $\alpha > 0$, $0 < \kappa < 1$, $\mu \ge 0$ and $B = B_z$, with $z \in \IL$, we introduce the (bad) event (see (\ref{1.24}) for notation)
\begin{equation}\label{2.11}
\begin{split}
\cB^{B,F}_{\alpha, \kappa, \mu} = &\; \big\{ \langle \ov{e}_B, L^B_{\alpha\,{\rm cap}(B)} \rangle \notin \big(\alpha(1 - \kappa), \alpha (1 + \kappa)\big)\big\} \;\cup
\\
&\; \Big\{ \mbox{\f $\dis\frac{1}{|B|}$} F^B_{\alpha \, {\rm cap}(B)} \notin \big(\theta(\alpha(1- \kappa)\big) - \mu , \theta\big(\alpha (1 + \kappa)\big) + \mu\big)\Big\}.
\end{split}
\end{equation}

\n
When there is no ambiguity about which $F$ we use, we will drop $F$ from the notation and when $\mu = 0$, we will also drop $\mu$ from the notation.

\medskip
We are now going to recall some facts about the soft local time technique, as developed in \cite{PopoTeix15} and \cite{ComeGallPopoVach13}, very much in the spirit of Section 4 of \cite{Szni17}. This will be the tool to control the events $\cB^{B,F}_{\alpha, \kappa, \mu}$ introduced above. In particular, as a result of Proposition \ref{prop2.2} below, we will see that when $K$ is sufficiently large, their probability has a super-polynomial decay in $L$.

\medskip
We thus consider $L \ge 1$, $K \ge 100$, and a non-empty finite subset $\cC$ of $\IL$ satisfying
\begin{equation}\label{2.12}
\mbox{for $z \not= z'$ in $\cC$, one has $|z - z'|_\infty \ge \ov{K} \,L$, where $\ov{K} = 2 K + 3$},
\end{equation}
so that for $z \not= z'$ in $\cC$ (see (\ref{1.19}) for notation)
\begin{equation}\label{2.13}
d(U_z,U_{z'}) \ge 3 L + 2.
\end{equation}

\n
We will often write $B \in \cC$ as a shorthand for $B_z, z\in \cC$. We use the soft local time technique, see \cite{ComeGallPopoVach13} especially, to couple the excursions $Z^B_\ell$, $\ell \ge 1$, $B \in \cC$ of the random interlacements, with independent excursions $\wt{Z}^B_\ell$, $\ell \ge 1$, $B \in \cC$, respectively distributed as $X_{\point \wedge T_U}$ under $P_{\ov{e}_B}$, as $B$ varies over $\cC$. We introduce the subsets of $\IZ^d$
\begin{equation}\label{2.14}
C = \textstyle{\bigcup\limits_{z \in \cC}} B_z \subseteq W = \bigcup\limits_{z \in \cC} U_z.
\end{equation}

\n
For $x$ in $\IZ^d$, we denote by $Q_x$ the probability measure governing two independent  conti\-nuous-time walks $X^1_\point$ and $X^2_\point$ on $\IZ^d$, respectively starting from $x$ and from the initial distribution $\ov{e}_C$ (i.e.~the normalized equilibrium measure of $C$). Letting $H_C$ denote the entrance time of $X^1_\point$ in $C$, we consider
\begin{equation}\label{2.15}
Y = \left\{ \begin{array}{ll}
X^1_{H_C}, & \mbox{on $\{H_C < \infty\}$}
\\[1ex]
X^2_0, & \mbox{on $\{H_C =  \infty\}$}.
\end{array}\right.
\end{equation}

\n
The soft local time technique constructs a coupling $\IQ^\cC$ of the law $\IP$ of the random interlacements with collections of independent right-continuous Poisson counting functions, with unit intensity, vanishing at $0$, $(n_{B_z}(0,t))_{t \ge 0}$, $z \in \cC$, and with independent collections of i.i.d. excursions $\wt{Z}^{B_z}_\ell$, $\ell \ge 1$, $z \in \cC$, having for each $z \in \cC$ the same law on $\Gamma (U_z)$ as $X_{\point \wedge T_{U_z}}$ under $P_{\ov{e}_{B_z}}$. We further define for $B \in \cC$
\begin{equation}\label{2.16}
n_B(a,b) = n_B(0,b) - n_B(0,a),\; \mbox{for $0 \le a \le b$}
\end{equation}
(the notation is consistent when $a = 0$). One then has
\begin{equation}\label{2.17}
\left\{ \begin{array}{l}
\mbox{under $\IQ^\cC$, as $B$ varies over $\cC$, the $((n_B(0,t))_{t \ge 0}, \wt{Z}^B_\ell, \ell \ge 1)$ are independent} 
\\
\mbox{collections of independent processes with $(n_B(0,t))_{t \ge 0}$ distributed as a}
\\
\mbox{Poisson counting process of intensity $1$ and $\wt{Z}^B_\ell$, $\ell \ge 1$, as i.i.d. $\Gamma(U)$-valued}
\\
\mbox{variables with same law as $X_{\point \wedge T_U}$ under $P_{\ov{e}_B}$}.
\end{array}\right.
\end{equation}

\n
The coupling measure $\IQ^\cC$ has the following crucial property, see Lemma 2.1 of \cite{ComeGallPopoVach13}. If for some $\delta \in (0,1)$ and all $B \in \cC$, $y \in B$ and $x \in \partial W$ (see (\ref{2.14})), one has
\begin{equation}\label{2.18}
\Big(1 - \mbox{\f $\dis\frac{\delta}{3}$}\Big) \;\ov{e}_B(y) \le Q_x [Y = y \, | \,Y \in B] \le \Big(1 +\mbox{\f $\dis\frac{\delta}{3}$}\Big)\;\ov{e}_B(y),
\end{equation}

\n
then, for any $B \in \cC$ and $m_0 \ge 1$, on the event (in the auxiliary space governed by $\IQ^\cC$):
\begin{equation}\label{2.19}
\begin{array}{l}
\wt{U}^{m_0}_B =
\\[1ex]
\{n_B(m,(1+ \delta)m) < 2 \delta m, (1-\delta) \,m < n_B(0,m) < (1 + \delta) \,m, \;\mbox{for all $m \ge m_0$}\},
\end{array}
\end{equation}

\n
one has for all $m \ge m_0$ the following inclusions among subsets of $\Gamma(U)$
\begin{align}
\{\wt{Z}^B_1,\dots,\wt{Z}^B_{(1-\delta)m}\} & \subseteq \{Z^B_1,\dots,Z^B_{(1+ 3 \delta)m}\} \; \mbox{and} \label{2.20}
\\[1ex]
\{Z^B_1,\dots, Z^B_{(1-\delta)m}\} & \subseteq \{\wt{Z}^B_1,\dots,\wt{Z}^B_{(1+ 3 \delta)m}\} ,\label{2.21}
\end{align}

\n
where $\wt{Z}^B_v$ and $Z^B_v$ respectively stand for $\wt{Z}^B_{[v]}$ and $Z^B_{[v]}$ when $v \ge 1$, and the sets in the left members of (\ref{2.20}) and (\ref{2.21}) are empty if $(1-\delta) \,m < 1$. Importantly, the favorable event $\wt{U}^{m_0}_B$ is solely defined in terms of $(n_B(0,t))_{t \ge 0}$. We now choose $m_0$ as a function of $L$ via
\begin{equation}\label{2.22}
m_0 = [(\log L)^2] + 1.
\end{equation}
Since $(n_B(0,t))_{t \ge 0}$ is a Poisson counting process of unit intensity, a standard exponential Chebyshev inequality yields that
\begin{equation}\label{2.23}
\lim\limits_{L \rightarrow \infty} \; \mbox{\f $\dis\frac{1}{\log L}$} \; \log \IQ^\cC \big[(U_B^{m_0})^c \big] = - \infty, \; \mbox{for all $B \in \cC$}
\end{equation}

\n
(the above probability does not depend on $\cC$, or on the choice of $B \in \cC$).

\medskip
We need some further notation. For $B \in \cC$, we define
\begin{equation}\label{2.24}
\mbox{$\wt{L}^B_{a,x}$ as well as $\wt{F}^B_{a,x}$ and $\wt{F}^B_a$, for $a \ge 0$ and $x$ in $\IZ^d$},
\end{equation}

\n
as in (\ref{1.23}), (\ref{2.9}), (\ref{2.10}), with the $\wt{Z}^B_\ell$, $\ell \ge 1$, in place of the $Z^B_\ell$, $\ell \ge 1$. Moreover, given $\alpha > 0$ and $0 < \delta < 1$, we set
\begin{equation}\label{2.25}
\alpha_0 = \mbox{\f $\dis\frac{1- \delta}{1 + 4 \delta}$} \; \alpha \le \alpha \le \alpha_1 = \mbox{\f $\dis\frac{1 + 4 \delta}{1 - \delta}$} \;\alpha.
\end{equation}

\n
From now on, we choose $\delta$ as a function of $\kappa$ (see above (\ref{2.11})) so that
\begin{equation}\label{2.26}
\delta = \delta(\kappa) \in \Big(0, \mbox{\f $\dis\frac{1}{2}$}\Big) \; \;\mbox{is such that}\;\; \mbox{\f $\dis\frac{1+4\delta}{1 + \delta}$} < 1 + \mbox{\f $\dis\frac{\kappa}{2}$} \;\mbox{and}\; \mbox{\f $\dis\frac{1-\delta}{1 + 4 \delta}$} \;\Big(1 - \mbox{\f $\dis\frac{\delta}{100}$}\Big) > 1 - \mbox{\f $\dis\frac{\kappa}{2}$}\;.
\end{equation}

\n
The next step is to introduce an event (on the auxiliary space governed by $\IQ^\cC$) that will give us a way to control the bad event in (\ref{2.11}) when $B \in \cC$,~cf.~Proposition \ref{prop2.2}. Thus, for $\alpha > 0$, $0 < \kappa < 1$, $\mu \ge 0$, and $B \in \cC$, $F$ satisfying (\ref{2.1}), we set
\begin{equation}\label{2.27}
\begin{array}{l}
\wt{\cB}^{B,F}_{\alpha, \kappa, \mu} = \wt{A}_1 \cup \wt{A}_2 \cup \wt{A}_3, \;\mbox{where}
\\[1ex]
\wt{A}_1 = (\wt{U}^{m_0}_B)^c, \wt{A}_2 = \big\{\langle \ov{e}_B, \wt{L}^B_{\alpha_1\,{\rm cap}(B)} \rangle \ge \alpha(1 + \kappa) \; \mbox{or $\langle \ov{e}_B, \wt{L}^B_{\alpha_0\,{\rm cap}(B)}\rangle \le \alpha(1 - \kappa) \}$, and}
\\[1ex]
\wt{A}_3 = \{ \wt{F}^B_{\alpha_1\,{\rm cap}(B)} \ge (\theta(\alpha(1 + \kappa)) + \mu) \,|B|\;\mbox{or} \;\wt{F}^B_{\alpha_0\,{\rm cap}(B)} \le (\theta(\alpha(1 - \kappa)) - \mu) \,|B|\}. 
\end{array}
\end{equation}

\n
We will write $\wt{\cB}^{B,F}_{\alpha,\kappa}$ when $\mu = 0$, and often drop the superscript $F$ when this causes no confusion. Incidentally, note that due to (\ref{2.17})
\begin{equation}\label{2.28}
\mbox{the events $\wt{\cB}^{B,F}_{\alpha, \kappa, \mu}$ as $B$ varies over $B_z$, $z \in \cC$, are i.i.d.}.
\end{equation}

\n
We now come to the key control of this section. It shows that when $K$ is large for $B$ in $\cC$, we can dominate the occurrence of the bad event in (\ref{2.11}) by the event in (\ref{2.27}), and when $\mu > 0$, the probability of the latter has super-polynomial decay in $L$. When, in addition, $\theta(\cdot)$ is strictly increasing, i.e.~satisfies (\ref{2.2}), this super-polynomial decay holds for $\mu = 0$ as well.

\begin{proposition}\label{prop2.2} (super-polynomial decay)

\smallskip\n
Assume that $F$ satisfies (\ref{2.1}). Then, for any $\alpha > 0$, $0 < \kappa < 1$, $\mu \ge 0$, $K \ge c_1(\alpha, \kappa, \mu, F)$, for large $L$, for any $\cC$ as in (\ref{2.12}) and $B \in \cC$, under $\IQ^\cC$
\begin{equation}\label{2.29}
\cB^{B,F}_{\alpha, \kappa, \mu} \subseteq \wt{\cB}^{B,F}_{\alpha, \kappa, \mu}\;.
\end{equation}
Moreover, when $\mu > 0$,
\begin{equation}\label{2.30}
\lim\limits_{L \rightarrow \infty} \; \mbox{\f $\dis\frac{1}{\log L}$} \; \log \IQ^\cC [\wt{\cB}^{B,F}_{\alpha, \kappa, \mu}] = - \infty
\end{equation}

\n
(the above probability does not depend on the choice of $\cC$ or of $B \in \cC$). In addition,
\begin{equation}\label{2.31}
\mbox{if (\ref{2.2}) holds, then (\ref{2.30}) holds for $\mu = 0$ as well.}
\end{equation}
\end{proposition}

\begin{proof}
Recall that $\delta$, as chosen in (\ref{2.26}), is a function of $\kappa$. As in (4.11) of \cite{Szni17} (see also below (4.16) of the same reference), we can choose $c_2(\kappa) \ge 100$ such that
\begin{equation}\label{2.32}
\mbox{when $K \ge c_2(\kappa)$, for any $L \ge 1$, $\cC$ as in (\ref{2.12}), the bounds (\ref{2.18}) hold}.
\end{equation}

\n
From now, we will tacitly assume that $K \ge c_2(\kappa)$, except when explicitly stated otherwise, as in Lemma \ref{lem2.3} below.

\medskip
We first prove (\ref{2.29}). To this end we note that by (\ref{1.25}), for large $L$, the integer $m = [\frac{\alpha}{1 - \delta}\, {\rm cap}(B)] + 1$ satisfies $m \ge m_0$ (with $m_0$ as in (\ref{2.22})), and that $(1-\delta) m \ge \alpha\, {\rm cap}(B)$ together with $(1 + 3 \delta)\,m \le \frac{1 + 4 \delta}{1 - \delta} \, \alpha \,{\rm cap}(B)$. Hence, for any $\cC$ as in (\ref{2.12}) and $B \in \cC$, we find by (\ref{2.21}) and (\ref{2.25}) that
\begin{equation}\label{2.33}
\mbox{on} \; \wt{U}^{m_0}_B, \{Z^B_1,\dots , Z^B_{\alpha \,{\rm cap}(B)}\} \subseteq \{\wt{Z}^B_1,\dots, \wt{Z}^B_{\alpha_1 \,{\rm cap}(B)}\}.
\end{equation}

\n
Similarly, for large $L$, the integer $m = [\frac{\alpha}{(1 + 3 \delta)} \;{\rm cap}(B)]$ satisfies $m \ge m_0$, as well as $(1 + 3 \delta)\,m \le \alpha \,{\rm cap}(B)$ and $(1 - \delta) \,m \ge \frac{(1 - \delta)}{1 + 4 \delta} \;\alpha\; {\rm cap}(B)$. Hence, for any $\cC$ as in (\ref{2.12}) and $B \in \cC$, we find by (\ref{2.20}) and (\ref{2.25}) that
\begin{equation}\label{2.34}
\mbox{on} \; \wt{U}^{m_0}_B, \{\wt{Z}^B_1,\dots ,\wt{Z}^B_{\alpha_0 \,{\rm cap}(B)} \} \subseteq \{Z^B_1,\dots, Z^B_{\alpha \,{\rm cap}(B)}\}.
\end{equation}

\n
As a result, for large $L$, for any $\cC$ as in (\ref{2.12}) and $B \in \cC$,
\begin{equation}\label{2.35}
\begin{split}
\mbox{on $\wt{U}^{m_0}_B$, one has} &\;\langle \ov{e}_B, \wt{L}^B_{\alpha_0 \,{\rm cap} (B)}\rangle \le \langle \ov{e}_B, L^B_{\alpha \, {\rm cap}(B)} \rangle \le \langle \ov{e}_B, \wt{L}_{\alpha_1 \,{\rm cap}(B)} \rangle \;\mbox{and} 
\\
&\;\wt{F}^B_{\alpha_0 \,{\rm cap}(B)} \le F^B_{\alpha \, {\rm cap}(B)} \le \wt{F}^B_{\alpha_1 \, {\rm cap}(B)}.
\end{split}
\end{equation}

\n
Looking at the respective definitions in (\ref{2.11}) and (\ref{2.27}) the claim (\ref{2.29}) follows.

\medskip
We now turn to the proof of (\ref{2.30}). In view of (\ref{2.23}), we only need to show the super-polynomial decay in $L$ of the $\IQ^\cC$-probability of $\wt{A}_2$ and $\wt{A}_3$ in the notation of (\ref{2.27}). Given the i.i.d. nature of the $\wt{Z}^B_\ell$, $\ell \ge 1$, the proof will be based on concentration. We write $\wt{Z}_\ell$ as a shorthand for $\wt{Z}^B_\ell$, when $\ell \ge 1$. Our next step is to show that when $K \ge c'_2(\kappa)$ (from (\ref{2.41}) below),
\begin{equation}\label{2.36}
\mbox{$\IQ^\cC[\wt{A}_2]$ has super-polynomial decay in $L$.}
\end{equation}

\n
To this end, we define for $n \ge 0$ (with $T_U(\wt{Z}_\ell)$ the exit time of $\wt{Z}_\ell$ from $U$):
\begin{equation}\label{2.37}
T_n = \dsl^n_{\ell = 1} \;\dis\int_0^{T_U(\wt{Z}_\ell)} e_B\big(\wt{Z}_\ell (s)\big) \,ds \;\; \mbox{and} \; \;T'_n = \dsl^n_{\ell = 1} \;\Big(\dis\int_0^{T_U(\wt{Z}_\ell)} e_B\big(\wt{Z}_\ell(s)\big)\,ds\Big) \wedge L^{\frac{1}{4}}.
\end{equation}

\n
Using the convention $T_v = T_{[v]}$ and $T'_v = T'_{[v]}$, for $v \ge 0$, we see that
\begin{equation}\label{2.38}
T'_{\alpha_0 \,{\rm cap}(B)} \le T_{\alpha_0 \,{\rm cap}(B)} = \langle e_B, \wt{L}^B_{\alpha_0 \,{\rm cap}(B)}\rangle \;\;\mbox{and} \;\; T_{\alpha_1 \,{\rm cap}(B)} = \langle e_B, \wt{L}^B_{\alpha_1 \,{\rm cap}(B)}\rangle .
\end{equation}

\n
Next, we remark that under $\IQ^\cC$ the summands entering $T_n$ are i.i.d. and
\begin{equation}\label{2.39}
\begin{array}{l}
\mbox{$\dis\int_0^{T_U(Z_1)} e_B\big(\wt{Z}_1(s)\big)\,ds$ is stochastically dominated by}
\\
\mbox{$\tau$ distributed as $\dis\int^\infty_0 e_B(X_s) \, ds$ under $P_{\ov{e}_B}$}.
\end{array}
\end{equation}

\n
By Lemma 1.1 of \cite{Szni17} one knows that $\tau$ has an exponential distribution with parameter $1$. Since $\alpha_1 \le \alpha (1 + \frac{\kappa}{2})$, see (\ref{2.25}), (\ref{2.26}), it follows by the exponential Chebyshev inequality and (\ref{1.25}) that
\begin{equation}\label{2.40}
\IQ^\cC [ \langle e_B, \wt{L}^B_{\alpha_1\,{\rm cap}(B)}\rangle \ge \alpha (1 + \kappa) \,{\rm cap}(B)] \; \mbox{has super-polynomial decay in $L$}.
\end{equation}

\n
Further, by assuming $K \ge c'_2(\kappa) ( \ge c_2(\kappa)$ in (\ref{2.32})), we can make sure that
\begin{equation}\label{2.41}
g_U(x,y) \ge \Big(1 - \mbox{\f $\dis\frac{\delta}{200}$}\Big) \;g(x,y) \; \mbox{for all $x,y \in B$}.
\end{equation}
We then see that for large $L$, for any $n \ge 0$ (with $\tau$ as in (\ref{2.39}))
\begin{equation}\label{2.42}
\begin{split}
E^{\IQ^\cC} [T'_n]  \ge&\; E^{\IQ^\cC} [T_n] - n\, E[\tau \,1\{\tau \ge L^{\frac{1}{4}}\}] \stackrel{(\ref{2.41}), (\ref{2.37})}{\ge}
\\
&\; n  \Big(1 - \mbox{\f $\dis\frac{\delta}{200}$}\Big) \,\dsl_{x,y \in B} \ov{e}_B(x) \,g(x,y) \,e_B(y) - n \dis\int^\infty_{L^{\frac{1}{4}}} s \,e^{-s} ds \ge n  \Big(1 - \mbox{\f $\dis\frac{\delta}{150}$}\Big) .
\end{split}
\end{equation}
Now, for large $L$, we have $[\alpha_0 \,{\rm cap}(B)] (1 - \frac{\delta}{150}) \ge \alpha_0 (1 - \frac{\delta}{100}) \,{\rm cap}(B) \ge \alpha (1 - \frac{\kappa}{2})$ by (\ref{2.26}), and we thus find that
\begin{equation}\label{2.43}
\begin{array}{l}
\IQ^\cC [\langle e_B, \wt{L}^B_{\alpha_0\,{\rm cap}(B)} \rangle \le \alpha (1 - \kappa) \,{\rm cap}(B)] \le \IQ^\cC [T'_{\alpha_0\,{\rm cap}(B)}  \le \alpha (1 - \kappa) \,{\rm cap}(B)] \le
\\[0.5ex]
\IQ^\cC \Big[T'_{\alpha_0\,{\rm cap}(B)} - E^{\IQ^\cC} [T'_{\alpha_0\,{\rm cap}(B)} ] \le - \mbox{\f $\dis\frac{\alpha \kappa}{2}$} \;{\rm cap}(B)\Big] \le \exp\Big\{ - c \;\mbox{\f $\dis\frac{(\alpha \kappa)^2}{\alpha L^{d-2}}$} \; L^{2(d-2) - \frac{1}{2}}\Big\}
\end{array}
\end{equation}

\medskip\n
using Azuma-Hoeffding's inequality in the last step (see Lemma 4.1 on p.~68 of \cite{Ledo01}), and (\ref{1.25}). Together with (\ref{2.40}), this completes the proof of (\ref{2.36}).

\medskip
We will now bound $\IQ^\cC[\wt{A}_3]$ (see (\ref{2.27})). As in (\ref{2.37}) above, we introduce a truncation scheme in order to apply the Azuma-Hoeffding inequality for martingales (see Lemma 4.1, p.~68 of \cite{Ledo01}), and obtain concentration, see (\ref{2.53}) below. Recall that $R$ from (\ref{2.1}) controls the range of $F$. We introduce the notation
\begin{equation}\label{2.44}
B^R = \{x \in \IZ^d; d(x,B) \le R\} \; \; \mbox{and} \;\;B_R = \{x \in B; \, B(x,R) \subseteq B\}
\end{equation}
(so $B_R \subseteq B \subseteq B^R)$.

\medskip
For $x \in \IZ^d$, $A \subseteq \IZ^d$ and $\ell \ge 1$ we set
\begin{equation}\label{2.45}
L_x(\wt{Z}_\ell) = \dis\int_0^{T_U(\wt{Z}_\ell)} 1\{\wt{Z}_\ell(s) = x\}\,ds \; \; \mbox{and} \;\; L_A (\wt{Z}_\ell) = \dsl_{y \in A} \;L_y(\wt{Z}_\ell).
\end{equation}
Note that $L_A(\wt{Z}_\ell) = 0$ if $A \cap U = \emptyset$.

\medskip
To perform the truncation we define for each $\ell \ge 1$,
\begin{equation}\label{2.46}
\sigma_\ell = \inf\Big\{s \ge 0; \,|\wt{Z}_\ell  \,([0,s]) \cap B^R| \ge L^{2 + \frac{1}{4}} \; \mbox{or}\; \dis\int^{s \wedge T_U(\wt{Z}_\ell)}_0 1\{\wt{Z}_\ell(t) \in B^R\} \, dt \ge L^{2 + \frac{1}{4}}\Big\},
\end{equation}
and for $x \in \IZ^d$ or $A \subseteq \IZ^d$,
\begin{equation}\label{2.47}
\mbox{$L'_x(\wt{Z}_\ell)$ and $L'_A(\wt{Z}_\ell)$  as in (\ref{2.45}) with $T_U (\wt{Z}_\ell) \wedge \sigma_\ell$ in place of $T_U(\wt{Z}_\ell)$}.
\end{equation}

\n
With a similar spirit as in (\ref{2.37}), we define for $n \ge 0$
\begin{equation}\label{2.48}
\left\{ \begin{array}{l}
V_n = \dsl_{x \in B} \; F\Big(\Big(\dsl_{1 \le \ell \le n} \,L_{x+y} (\wt{Z}_\ell)\Big)_{|y|_\infty \le R}\Big) \big(\stackrel{(\ref{2.24})}{=}  \wt{F}^B_n\big), \;\mbox{and}
\\[3ex]
V'_n = \dsl_{x \in B} \; F\Big(\Big(\dsl_{1 \le \ell \le n} \,L'_{x+y} (\wt{Z}_\ell)\Big)_{|y|_\infty \le R}\Big)
\end{array}\right.
\end{equation}
(we also write $V_v = V_{[v]}$ and $V'_v = V'_{[v]}$ for $v \ge 0$).

\medskip
We then consider the filtration $(\cF_n)_{n \ge 0}$: with $\cF_0$ trivial and $\cF_n = \sigma(\wt{Z}_1,\dots, \wt{Z}_n)$ for $n \ge 1$, as well as the bounded martingale
\begin{equation}\label{2.49}
M'_n = E^{\IQ^\cC} [V'_{\alpha_1 \,{\rm cap}(B)} \,| \,\cF_n], \; n \ge 0.
\end{equation}

\n
Note that $M'_n$ is obtained by integrating out the variables $\wt{Z}_\ell$, $\ell > n$ among the i.i.d. variables $\wt{Z}_i$, $1 \le i \le \alpha_1 \,{\rm cap}(B)$ entering the definition of $V'_n$. We write $E^n$ for the integration over the variables $\wt{Z}_\ell$, $\ell > n$, so that for $n \ge 0$ such that $n + 1 \le \alpha_1\,{\rm cap}(B)$, we have
\begin{equation}\label{2.50}
M'_{n+1} - M'_n = E^{n+1} [V'_{\alpha_1 \,{\rm cap}(B)}] - E^n [V'_{\alpha_1 \,{\rm cap}(B)}].
\end{equation}

\n
If $\ov{Z}_{n+1}$ denotes an independent copy of $\wt{Z}_{n+1}$ and $\ov{E}^{n+1}$ stands for the integration over the $\wt{Z}_\ell$, $\ell > n+1$, and $\ov{Z}_{n+1}$, we have (with hopefully obvious notation)
\begin{equation}\label{2.51} 
\begin{split}
|M'_{n+1} - M'_n| = \Big|\ov{E}\,^{n+1} \Big[\dsl_{x \in B} \Big\{ & F\Big(\Big(\dsl_{\ell \le \alpha_1 \, {\rm cap}(B)} L'_{x+y} (\wt{Z}_\ell)\Big)_{|y|_\infty \le R}\Big)
\\
&\!\!\!-F\Big(\Big(\dsl_{\ell \le \alpha_1 \, {\rm cap}(B) \atop \ell \not= n+1} L'_{x+y} (\wt{Z}_\ell) + L'_{x+y} (\ov{Z}_{n+1})\Big)_{|y|_\infty \le R}\Big)\Big\}\Big]\Big|.
\end{split}
\end{equation}
Note that in the above sum, only those $x$ in $B$ within sup-distance at most $R$ from the range of $\wt{Z}_{n+1} ( \cdot \wedge \sigma_{n+1})$ or the range of $\ov{Z}_{n+1} ( \cdot \wedge \ov{\sigma}_{n+1})$ (where $\ov{\sigma}_{n+1}$ is defined as in (\ref{2.46}) with $\ov{Z}_{n+1}$ in place of $\wt{Z}_\ell$) have a possibly non-vanishing contribution. There are at most $c(R) \,L^{2 + \frac{1}{4}}$ such $x$ by (\ref{2.46}). Moreover, for any such $x$ the term inside the accolade in (\ref{2.51}) is by (\ref{2.1}) iii) in absolute value at most $c(F) \{1 + \sum_{|y|_\infty \le R} L'_{x+y} (\wt{Z}_{n+1}) + L'_{x+y} (\ov{Z}_{n+1})\}$. We thus find that for $n \ge 0$ with $n+1 \le \alpha_1\, {\rm cap}(B)$, one has (recall that we view $R$ as a function of $F$)
\begin{equation}\label{2.52}
\begin{split}
|M'_{n+1} - M'_n | & \le c'(F) \, L^{2 + \frac{1}{4}} + c''(F) \,\dsl_{x \in B^R} (L'_x (\wt{Z}_{n+1}) + \ov{E}^{n+1} [L'_x(\ov{Z}_{n+1})])
\\[-2ex]
&\!\!\!\stackrel{(\ref{2.46})}{\le} c(F) \, L^{2 + \frac{1}{4}}.
\end{split}
\end{equation}

\n
With this bound on the increments of the martingale $M'_n$, $n \ge 0$, we can apply the Azuma-Hoeffding inequality (see \cite{Ledo01}, p.~68) and find:
\begin{equation}\label{2.53}
\begin{array}{l}
\IQ^\cC \big[| V'_{\alpha_1 \,{\rm cap}(B)} - E^{\IQ^\cC} [V'_{\alpha_1 \,{\rm cap}(B)} ] | \ge L^{d-\frac{1}{10}}\big] \le 2\exp\Big\{ - \mbox{\f $\dis\frac{L^{2d-\frac{1}{5}}}{c'(F) \,\alpha_1 \,{\rm cap}(B)\, L^{4 + \frac{1}{2}}}$}\Big\}
\\
\stackrel{(\ref{1.25})}{\le} 2 \,\exp\Big\{- \mbox{\f $\dis\frac{c(F)}{\alpha_1}$} \;L^{\frac{1}{4}}\Big\}.
\end{array}
\end{equation}
A similar inequality holds with $\alpha_0$ in place of $\alpha_1$.

\medskip
We will now see that for positive $\beta$, the difference of the expectations of $V_{\beta \, {\rm cap}(B)}$ and $V'_{\beta\,{\rm cap}(B)}$ has super-polynomial decay in $L$, and if $K$ is large, $\frac{1}{|B|} \,E^{\IQ^\cC}[V_{\beta\, {\rm cap}(B)}]$ as $L$ tends to infinity, is not far from $\theta(\beta)$. We refer to (\ref{2.48}) for notation.

\begin{lemma}\label{lem2.3}
For $\beta > 0$, $K \ge 100$, $\cC$ as in (\ref{2.12}) and $B \in \cC$,
\begin{align}
&\mbox{$0 \le E^{\IQ^\cC}[V_{\beta \,{\rm cap}(B)}] - E^{\IQ^\cC}[V'_{\beta \,{\rm cap}(B)}]$ has super-polynomial decay in $L$, and} \label{2.54}
\\
&\theta(\beta) \ge \underset{L \r \infty}{\overline{\lim}} \,\mbox{\f $\dis\frac{1}{|B|}$} \;E^{\IQ^\cC} [V_{\beta \,{\rm cap}(B)}] \ge \underset{L \r \infty}{\underline{\lim}} \,\mbox{\f $\dis\frac{1}{|B|}$} E^{\IQ^\cC}[V_{\beta \,{\rm cap}(B)}] \ge \theta(\beta) - c(\beta, F)\,K^{-\frac{(d-2)}{2}}  \label{2.55}
\end{align}

\n
(the expectations in (\ref{2.54}) and (\ref{2.55}) do not depend on $\cC$ or on $B$ in $\cC$).
\end{lemma}

\begin{proof}
We first prove (\ref{2.54}). The difference in (\ref{2.54}) is clearly non-negative by the monotonicity assumption (\ref{2.1}) i) on $F$ and (\ref{2.48}). In addition, one has
\begin{equation}\label{2.56}
\begin{split}
E^{\IQ^\cC} [V_{\beta \,{\rm cap}(B)} - V'_{\beta \,{\rm cap}(B)}] & \le \dsl_{1 \le \ell \le \beta \, {\rm cap}\,(B)} \,E^{\IQ^\cC}[\sigma_\ell < \infty, V_{\beta \,{\rm cap}(B)}]
\\
&\!\!\!\!\!\!\! \stackrel{\rm symmetry}{\le} [\beta \,{\rm cap}(B)] \;E^{\IQ^\cC} [\sigma_1 < \infty, V_{\beta \,{\rm cap}(B)}].
\end{split}
\end{equation}
By (\ref{2.1}) ii) and iii), we find that (in the notation of (\ref{2.44}), (\ref{2.45}))
\begin{equation}\label{2.57}
V_{\beta \,{\rm cap}(B)} \le c'(F) \dsl_{1 \le \ell \le \beta \,{\rm cap}(B)} \big(|{\rm range} (\wt{Z}_\ell) \cap B^R| + L_{B^R}(\wt{Z}_\ell)\big) \stackrel{\rm def}{=} A.
\end{equation}

\n
By Lemma \ref{lem1.1} applied to $B^R$, we see in particular that
\begin{equation}\label{2.58}
E^{\IQ^\cC} [V_{\beta \,{\rm cap}(B)}] \le E^{\IQ^\cC}[A] \le \wt{c} \, (F) \, \beta \,{\rm cap}(B) (L + R)^2 \stackrel{(\ref{1.25})}{\le} c(F)\, \beta(L+R)^d.
\end{equation}

\n
Coming back to (\ref{2.56}) and inserting (\ref{2.57}) in the last expectation, we find
\begin{equation}\label{2.59}
\begin{array}{l}
E^{\IQ^\cC} [V_{\beta \,{\rm cap}(B)} - V'_{\beta \,{\rm cap}(B)}]  \le 
\\[1ex]
c'(F) \beta \,{\rm cap}(B) \big(E^{\IQ^\cC}[\sigma_1 < \infty, |{\rm range} (\wt{Z}_1) \cap B^R| + L_{B^R}(\wt{Z}_1)] \; +
\\[1ex]
\IQ^\cC [\sigma_1 < \infty] \,E^{\IQ^\cC} \big[\dsl_{2 \le \ell \le \beta\,{\rm cap}(B)} \big( |{\rm range} (\wt{Z}_\ell) \cap B^R |  + L_{B^R}(\wt{Z}_\ell) \big)\big].
\end{array}
\end{equation}

\n
By the Cauchy-Schwarz inequality and Lemma \ref{lem1.1}, we find that
\begin{equation}\label{2.60}
\begin{split}
E^{\IQ^\cC}  [\sigma_1 < \infty, |{\rm range} (\wt{Z}_1) \cap B^R| + L_{B^R}(\wt{Z}_1)] & \le c(L + R)^2 \, \IQ^\cC[\sigma_1 < \infty]^{\frac{1}{2}}
\\
& \le c(L+R)^2 \,e^{-c'\,L^{2 + \frac{1}{4}} / (L+R)^2}.
\end{split}
\end{equation}

\n
In a similar fashion, by Lemma \ref{lem1.1}, the second inequality of (\ref{2.58}), and (\ref{1.25}),
\begin{equation}\label{2.61}
\begin{array}{l}
\IQ^\cC[\sigma_1 < \infty] \,E^{\IQ^\cC} \big[\dsl_{2 \le \ell \le \beta \,{\rm cap}(B)} \big(|{\rm range} (\wt{Z}_\ell) \cap B^R| + L_{B^R}(\wt{Z}_\ell)\big)\big] \le 
\\
c\,\beta(L+R)^d \,e^{-c \,L^{2 + \frac{1}{4}} / (L+R)^2}.
\end{array}
\end{equation}
Inserting (\ref{2.60}) and (\ref{2.61}) into (\ref{2.59}) readily yields (\ref{2.54}).

\medskip
We now turn to the proof of (\ref{2.55}). We will first show that
\begin{equation}\label{2.62}
\underset{L \r \infty}{\overline{\lim}} \; \mbox{\f $\dis\frac{1}{|B|}$} \;E^{\IQ^\cC} [V_{\beta\,{\rm cap}(B)}] \le \theta(\beta).
\end{equation}

\n
To this end, we consider $Y$ an independent Poisson variable with parameter $\beta_1\,{\rm cap}(B)$, where $\beta_1 > \beta$, as well as $\wh{Z}_\ell$, $\ell \ge 1$, i.i.d. $P_{\ov{e}_B}$-distributed (that is, distributed as the simple random walk with initial distribution, $\ov{e}_B$), and independent of $Y$. Then (with hopefully obvious notation), $(\sum_{1 \le \ell \le Y} \,L_x (\wt{Z}_\ell))_{x \in B^R}$ is stochastically dominated by $(\sum_{1 \le \ell \le Y} L_x(\wh{Z}_\ell))_{x \in B^R}$, which itself is stochastically dominated by the restriction to $B^R$ of the field of occupation times of random interlacements at level $\beta_1$, i.e.~$(L^{\beta_1}_x)_{x \in B^R}$. Thus, by the monotonocity of $F$, we find that (using stochastic domination in the second inequality)
\begin{equation}\label{2.63}
\begin{split}
E^{\IQ^\cC} [V_{\beta\,{\rm cap}(B)}] & \le E^Y \Big[E^{\IQ^\cC}\Big[\dsl_{x \in B} \,F\Big(\big(\dsl_{1 \le \ell \le Y} L_{x+y} (\wt{Z}_\ell)\big)_{|y|_\infty \le R}\Big)\Big] \,\Big|\, Y \ge \beta \, {\rm cap}(B)\Big] 
\\
& \le \IE \Big[\dsl_{x \in B} F\big((L^{\beta_1}_{x+y})_{|y|_\infty \le R}\big)\Big] \;P[Y \ge  \beta \, {\rm cap}(B)]^{-1}
\\
&\!\!\!\!\stackrel{{\rm def.\, of} \,\theta}{=} |B| \,\theta(\beta_1) \, P[Y \ge  \beta \, {\rm cap}(B)]^{-1}.
\end{split}
\end{equation}

\n
Since $\beta_1 > \beta$, the last probability tends to $1$ as $L$ goes to infinity. Hence, the left member of (\ref{2.62}) is at most $\theta(\beta_1)$. Letting $\beta_1$ decrease to $\beta$ yields (\ref{2.62}) in view of the continuity of $\theta$, see below (\ref{2.1}).

\medskip
We will now prove that
\begin{equation}\label{2.64}
\underset{L \r \infty}{\underline{\lim}} \; \mbox{\f $\dis\frac{1}{|B|}$} \; E^{\IQ^\cC} [V_{\beta \, {\rm cap}(B)}] \ge \theta(\beta) - c (\beta, F)\,K^{-\frac{(d-2)}{2}} .
\end{equation}

\n
We consider $\beta_0 \le \beta$ and $Y$ an independent Poisson variable with parameter $\beta_0 \, {\rm cap}(B)$. By the monotonicity of the non-negative function $F$ we find (see (\ref{2.48}) for notation):
\begin{equation}\label{2.65}
\begin{array}{l}
E^{\IQ^\cC} [V_{\beta\,{\rm cap}(B)}] \ge 
\\[1ex]
E^Y \Big[E^{\IQ^\cC}\Big[\dsl_{x \in B_R} \,F\Big(\big(\dsl_{1 \le \ell \le Y} L_{x+y} (\wt{Z}_\ell)\big)_{|y|_\infty \le R}\Big)\Big] \,\Big|\, Y < \beta \, {\rm cap}(B)\Big]  \ge
\\[2ex]
\dsl_{x \in B_R} E^Y \Big[E^{\IQ^\cC}\Big[F\Big(\big(\dsl_{1 \le \ell \le Y} L_{x+y} (\wt{Z}_\ell)\big)_{|y|_\infty \le R}\Big)\Big] \,\Big] - 
\\[2ex]
E^Y \big[E^{\IQ^\cC} [V_Y], Y \ge \beta \, {\rm cap}(B)\big] \,P[Y < \beta \, {\rm cap}(B)]^{-1} \ge 
\\[1ex]
\dsl_{x \in B_R}  E^Y \Big[E^{\IQ^\cC} \Big[F\Big(\big(\dsl_{1 \le \ell \le Y} L_{x+y}(\wt{Z}_\ell)_{|y|_\infty \le R}\big)\Big)\Big] -
\\[2ex]
\wt{c}(F) (L+R)^2 \,E^Y[Y,Y \ge \beta \, {\rm cap}(B)] \,P[Y < \beta \, {\rm cap}(B)]^{-1},
\end{array}
\end{equation}
where we used a similar bound as in (\ref{2.58}) on the first term of the fourth line.

\medskip
Since $Y$ is Poisson distributed with parameter $\beta_0 \,{\rm cap}(B)$, where $\beta_0 < \beta$, using the Cauchy-Schwarz inequality and usual exponential bounds, one sees that the term in the last line of (\ref{2.65}) goes to $0$ as $L$ tends to infinity (see also (\ref{1.25})). Let us now focus on the first term after the last inequality of (\ref{2.65}). For $x \in B_R$, we denote by $A^{\beta_0}_x$ the event stating that one trajectory of the interlacement at level $\beta_0$ enters $B$ and after exiting $U$ visits $B(x,R)$. Then, the first term after the last inequality of (\ref{2.65}) is at least
\begin{equation}\label{2.66}
\begin{array}{l}
\dsl_{x \in B_R} \big(\IE\big[ F\big((L^{\beta_0}_{x+y})_{|y|_\infty \le R}\big)\big] - \IE\big[F\big( (L^{\beta_0}_{x+y}\big)_{|y|_\infty \le R}\big), A^{\beta_0}_x\big]\big) \stackrel{(\ref{2.1})}{\ge}
\\
|B_R|\,\theta(\beta_0) - \dsl_{x \in B_R} \,c(F) \;\IE\big[1 + \dsl_{|y|_\infty \le R} \,L^{\beta_0}_{x + y}, \, A_x^{\beta_0}\big] \stackrel{\rm Cauchy-Schwarz}{\ge}
\\
|B_R|\,\theta(\beta_0) - \dsl_{x \in B_R} \,c(F) \;\IE\Big[\Big(1 + \dsl_{|y|_\infty \le R} \,L^{\beta_0}_y\Big)^2\Big]^{\frac{1}{2}} \, \IP[A^{\beta_0}_x]^{\frac{1}{2}} \ge
\\
|B_R|\,\theta(\beta_0) - |B_R| \,c(F) \big(1 + |B(0,R)| \,\IE[(L_0^{\beta_0})^2]^{\frac{1}{2}}\big) \,\sup\limits_{x \in B_R} \IP[A_x^{\beta_0}]^{\frac{1}{2}} \stackrel{(\ref{1.13}), \beta_0 < \beta}{\ge}
\\
|B_R| \big(\theta(\beta_0) - c(\beta,F) \sup\limits_{x \in B_R} \IP[A_x^{\beta_0}]^{\frac{1}{2}}\big).
\end{array}
\end{equation}

\n
Then, see for instance (1.20) of \cite{Szni17}, one has (with the notation from below (\ref{1.19})):
\begin{equation}\label{2.67}
a = \sup\limits_{y \in B, x \in B_R} P_y[\mbox{$X$ enters $B(x,R)$ after exiting $U$}] \le c(1 + R)^{d-2} / (KL)^{d-2},
\end{equation}
so that keeping in mind that $\beta_0 \le \beta$, one finds that for any $x \in B_R$,
\begin{equation}\label{2.68}
\IP[A_x^{\beta_0}] \le 1 - e^{-\beta_0 \,{\rm cap}(B)\,a} \stackrel{(\ref{1.25})}{\le} c\,\beta(1 + R)^{d-2}/K^{d-2}.
\end{equation}

\n
Hence, the first term after the last inequality of (\ref{2.65}) is bigger or equal to $|B_R| (\theta(\beta_0) - c'(\beta,F)$ $K^{-\frac{(d-2)}{2}})$, and since, as mentioned above, the term in the last line of (\ref{2.65}) tends to zero as $L$ goes to infinity, the claim (\ref{2.64}) now follows by first letting $L$ go to infinity, and then $\beta_0$ to $\beta$ with the help of the continuity of $\theta$ (see below (\ref{2.1})). This concludes the proof of (\ref{2.55}) and hence of Lemma \ref{lem2.3}.
\end{proof}

We will now conclude the proof of Proposition \ref{prop2.2}. Consider $\rho > 0$ such that
\begin{equation}\label{2.69}
\theta\big(\alpha (1 + \kappa)\big) + \mu \ge \theta\Big(\alpha \big(1 + \mbox{\f $\dis\frac{3}{4}$} \;\kappa\big)\Big) +  \mbox{\f $\dis\frac{\mu}{2}$} + \rho
\end{equation}

\n
(such a $\rho$ can be chosen when $\mu > 0$, or when $\mu = 0$ and (\ref{2.2}) holds as well). Then,
\begin{equation}\label{2.70}
\begin{array}{l}
\IQ^\ell \big[\wt{F}^\beta_{\alpha_1 \,{\rm cap}(B)} \ge  \big(\theta\big(\alpha (1 + \kappa)\big) + \mu\big) \,|B|\big] \le \IQ^{\cC} \big[V_{\alpha_1 \, {\rm cap}(B)} - V'_{\alpha_1 \, {\rm cap}(B)} \ge \rho \,|B|\big] \;+
\\[1ex]
\IQ^\cC \Big[V'_{\alpha_1 \, {\rm cap}(B)} \ge \Big( \theta\big(\alpha \big(1 + \mbox{\f $\dis\frac{3}{4}$} \;\kappa\big)\Big) +  \mbox{\f $\dis\frac{\mu}{2}$}\Big) \,|B|\Big] .
\end{array}
\end{equation}
The first term in the right member of (\ref{2.70}) has super-polynomial decay in $L$ by Markov inequality and (\ref{2.54}) of Lemma \ref{lem2.3}. As for the second term, using the first inequality of (\ref{2.55}) with $\beta = \alpha_1 < (1 + \frac{\kappa}{2})\,\alpha$ (by (\ref{2.25}), (\ref{2.26})), for large $L$, it is bounded by
\begin{equation}\label{2.71}
\IQ^\cC \Big[V'_{\alpha_1 \, {\rm cap}(B)} - E^{\IQ^\cC}  [V'_{\alpha_1 \, {\rm cap}(B)} ] \ge 
\Big\{\theta\Big(\alpha \big(1 + \mbox{\f $\dis\frac{3}{4}$} \;\kappa\big)\Big) +  \mbox{\f $\dis\frac{\mu}{2}$} - \theta \Big(\big(1 + \mbox{\f $\dis\frac{\kappa}{2}$} \big)\;\alpha\Big) - \mbox{\f $\dis\frac{\mu}{4}$} \Big\} \,|B|\Big].
\end{equation}
The expression between the accolades is strictly positive (recall that when $\mu = 0$, we assume (\ref{2.2})), and by the concentration bound (\ref{2.53}), it has super-polynomial decay in $L$. This shows that
\begin{equation}\label{2.72}
\lim\limits_{L \r \infty} \; \mbox{\f $\dis\frac{1}{\log L}$}\; \log \IQ^\cC [\wt{F}^B_{\alpha_1 \, {\rm cap}(B)} \ge \big(\theta\big(\alpha(1 + \kappa)\big) + \mu\big)\big] = - \infty.
\end{equation}

\n
On the other hand, we also have
\begin{equation}\label{2.73}
\begin{array}{l}
\IQ^\cC \big[\wt{F}^B_{\alpha_0 \, {\rm cap}(B)} \le \big(\theta\big(\alpha(1 - \kappa)\big) - \mu\big) \,|B|\big] \le
\\[1ex]
 \IQ^{\cC} [V'_{\alpha_0 \, {\rm cap}(B)} \le \big(\theta\big(\alpha(1 - \kappa)\big) - \mu\big)\;|B|\big] =
\\[1ex]
\IQ^{\cC} \Big[V'_{\alpha_0 \, {\rm cap}(B)} - E^{\IQ^\cC} [V'_{\alpha_0 \, {\rm cap}(B)}] \le  
\Big\{ \theta\big(\alpha(1 - \kappa)\big) - \mu - \mbox{\f $\dis\frac{1}{|B|}$} \;E^{\IQ^\cC}[V'_{\alpha_0 \, {\rm cap}(B)}]\Big\} \;|B|\Big].
\end{array}
\end{equation}

\n
Keeping in mind that $\alpha_0 > (1 - \frac{\kappa}{2})\,\alpha$ by (\ref{2.25}), (\ref{2.26}), it now follows from (\ref{2.54}) and the rightmost inequality in (\ref{2.55}) of Lemma \ref{lem2.3} applied with $\beta = \alpha_0$, that when $K \ge c(\alpha, \kappa, \mu, F)$, then for large $L$ the expression between the accolades in the second line of (\ref{2.73}) is at most $\theta(\alpha(1 - \kappa)) - \mu - (\theta(\alpha (1 - \frac{\kappa}{2})) - \frac{\mu}{2}) < 0$. The concentration bound (\ref{2.53}), with now $\alpha_0$ in place of $\alpha_1$, shows that
\begin{equation}\label{2.74}
\lim\limits_{L \r \infty} \; \mbox{\f $\dis\frac{1}{\log L}$}\; \log \IQ^\cC [\wt{F}^B_{\alpha_0 \, {\rm cap}(B)} \le \big(\theta\big(\alpha(1 - \kappa)\big) - \mu\big)  |B| \big] = - \infty.
\end{equation}

\n
Together with (\ref{2.72}), this proves that when $K \ge c(\alpha, \kappa, \mu, F)$, 
\begin{equation}\label{2.75}
\lim\limits_{L \r \infty} \; \mbox{\f $\dis\frac{1}{\log L}$}\; \log \IQ^\cC [\wt{A}_3] = - \infty .
\end{equation}

\n
Combined with (\ref{2.36}) and (\ref{2.23}), this concludes the proof of (\ref{2.30}), (\ref{2.31}), and hence of Proposition \ref{prop2.2}.
\end{proof}

We will now specify what we mean by a ``good'' or ``bad'' box $B$. We consider a local function $F$ satisfying (\ref{2.1}), with range $R$ and associated function $\theta$, together with $\sum \subseteq (0,\infty)$ a non-empty finite subset, $0 < \kappa < 1$, $\mu \ge 0$, as well as $L \ge 1$, $K \ge 100$ and $z \in \IL$. With $B = B_z$, and $F_0(\ell) = \ell$ for $\ell \ge 0$, as in (\ref{2.3}), we introduce the event (see (\ref{2.11}) for notation)
\begin{equation}\label{2.76}
\cB^{B,F}_{\Sigma, \kappa, \mu} = \textstyle{\bigcup\limits_{\alpha \in \Sigma}} \big(\cB^{B,F}_{\alpha,\kappa,\mu} \cup \cB^{B,F_0}_{\alpha,\kappa, \mu =0}\big),
\end{equation}
and say that the box $B$ is {\it ($\Sigma,\kappa,\mu$)-bad} if the above event occurs and that it is {\it ($\Sigma,\kappa,\mu$)-good} otherwise. In other words, for a $(\Sigma,\kappa,\mu)$-good box $B$, for any $\alpha$ in $\Sigma$, $\langle \ov{e}_B, L^B_{\alpha\, {\rm cap}(B)}\rangle$ lies in $(\alpha(1-\kappa),\alpha(1 + \kappa))$, $\langle m_B, L^B_{\alpha\, {\rm cap}(B)}\rangle$ lies in $(\alpha(1-\kappa),\alpha(1 + \kappa))$ and $\frac{1}{|B|} \,F^B_{\alpha\, {\rm cap}(B)}$ lies in $(\theta(\alpha(1-\kappa)) - \mu$, $\theta(\alpha(1+\kappa)) + \mu)$, see (\ref{2.10}) for notation.

\medskip
In the next section, the above Proposition \ref{prop2.2} and the coupling measure $\IQ^\cC$ from the soft local time technique will endow us with a tool to ensure that up to a ``negligible probability'', ``most boxes'' of relevance are good.

\medskip
We now come to the last lemma of this section, which can be viewed as a convenient device to track the ``undertow'' of random interlacements in a ``good box'', see also Remark \ref{rem2.5} below. We recall the notation (\ref{1.21}), and (\ref{2.44}).

\begin{lemma}\label{lem2.4}(good boxes and the undertow)

\smallskip\n
Consider $F,R,\theta,\Sigma,\kappa,\mu$ as above, as well as $L \ge 1$, $K \ge 100$, $z \in \IL$ and $B = B_z$. Then, for any $u > 0$ and $\alpha \in \Sigma$,
\begin{equation}\label{2.77}
\begin{array}{l}
\mbox{if $B$ is $(\Sigma,\kappa,\mu)$-good and $\frac{N_u(B)}{{\rm cap}(B)} \le \alpha$, then $\langle \ov{e}_B, L^u \rangle$, $\langle m_B, L^u\rangle$ are smaller}
\\[0.5ex]
\mbox{than $(1 + \kappa)\,\alpha$ and $\frac{1}{|B|} \,\sum_{x \in B_R} F((L^u_{x + \point}))$ is smaller than $\theta((1 + \kappa)\alpha) + \mu$.}
\\[0.5ex]
\mbox{Similarly, if $\frac{N_u(B)}{{\rm cap}(B)} \ge \alpha$, then $\langle \ov{e}_B, L^u \rangle$, $\langle m_B, L^u\rangle$ are bigger than $(1-\kappa)\,\alpha$,}
\\[0.5ex]
\mbox{and $\frac{1}{|B|} \;\sum_{x \in B} \;F((L^u_{x + \point}))$ is bigger than $\theta((1- \kappa)\,\alpha) - \mu$}. 
\end{array}
\end{equation}
If $\Sigma$ satisfies the condition
\begin{equation}\label{2.78}
\mbox{for each $\alpha \in \Sigma$, $\Sigma  \cap \big((1 - \kappa)\,\alpha, (1 + \kappa)\,\alpha\big) = \{\alpha\}$},
\end{equation}
then,
\begin{equation}\label{2.79}
\begin{array}{l}
\mbox{if $B$ is $(\Sigma,\kappa,\mu)$-good, for all $u > 0$, at most two elements of $\Sigma$ lie in }
\\[0.5ex]
\mbox{the closed intervals with endpoints consisting of a pair of values among}
\\
\frac{N_u(B)}{{\rm cap}(B)}, \langle \ov{e}_B,L^u\rangle, \langle m_B,L^u\rangle.
\end{array}
\end{equation}
\end{lemma}

\begin{proof}
We first prove (\ref{2.77}). We consider $u > 0$, $\alpha \in \Sigma$ and a $(\Sigma,\kappa,\mu)$-good box $B$. If we have $N_u(B) \le \alpha\,{\rm cap}(B)$, then in the notation of (\ref{1.23}) and above (\ref{1.1})
\begin{equation}\label{2.80}
|B| \,\langle m_B,L^u\rangle = \dsl_{x \in B} L^u_x = \dsl_{x \in B} \,L^B_{N_u(B),x} \le |B| \,\langle m_B, L^B_{\alpha \, {\rm cap}(B)}\rangle  < |B| \,\alpha(1 + \kappa),
\end{equation}

\n
and likewise $\langle \ov{e}_B,L^u\rangle \le \langle \ov{e}_B, L^B_{\alpha \, {\rm cap}(B)} \rangle < \alpha(1 + \kappa)$. In addition, $\sum_{x \in B_R} F((L^u_{x + \point})) \le F^B_{\alpha \, {\rm cap}(B)} < |B|(\theta((1 + \kappa) \,\alpha)  + \mu$).

\medskip
In a similar fashion, if $N_u(B) \ge \alpha \, {\rm cap}(B)$, one finds that $\langle m_B,L^u\rangle > (1 -\kappa)\,\alpha$, $\langle \ov{e}_B,L^u\rangle > (1-\kappa)\,\alpha$ and $\sum_{x \in B} F(L^u_{x + \point})) > |B|(\theta((1-\kappa)\,\alpha) - \mu$), and this completes the proof of (\ref{2.77}).

\medskip
We now prove (\ref{2.79}). We assume that (\ref{2.78}) holds, that $B$ is $(\Sigma,\kappa,\mu)$-good, and consider $u > 0$. We denote by $\ov{\Sigma}$ the finite set $\{0,\infty\} \cup \Sigma$. Let $[\underline{\alpha}, \ov{\alpha}]$ be the unique interval between consecutive values of $\ov{\Sigma}$ that contains $N_u(B)/{\rm cap}(B)$. By (\ref{2.77}), both $\langle m_B,L^u\rangle$ and $\langle \ov{e}_B,L^u\rangle$ are smaller than $(1 + \kappa)\,\ov{\alpha}$, and bigger than $(1-\kappa)\,\underline{\alpha}$, if $\underline{\alpha} > 0$. By (\ref{2.78}), when $\underline{\alpha} > 0$, the interval $((1-\kappa)\,\underline{\alpha}$, $(1 + \kappa)\,\ov{\alpha})$ contains at most two values of $\Sigma$. If instead $\underline{\alpha} = 0$, then $[0,(1+ \kappa)\, \ov{\alpha})$ contains at most one value of $\Sigma$. The claim (\ref{2.79}) follows, and this concludes the proof of Lemma \ref{lem2.4}.
\end{proof}

\begin{remark}\label{rem2.5} \rm 
In essence, the above Lemma \ref{lem2.4} permits to track the ``undertow'' $N_u(B)/$ ${\rm cap}(B)$ created by the interlacements at level $u$ in the box $B$, when $B$ is $(\Sigma,\kappa,\mu)$-good by means of the quantities $\langle \ov{e}_B,L^u\rangle$ and $\langle m_B,L^u\rangle$ (informally, the undertow in the box $B$  corresponds to a ``local value'' of the parameter of the interlacements for the box $B$). Also, in the case where (\ref{2.2}) holds (that is, when~$\theta$ is strictly increasing), one could track the undertow in the box $B$, when the box is $(\Sigma,\kappa,\mu = 0)$-good by means of the quantity $\theta^{-1} (\frac{1}{|B|} \, \sum_{x \in B} F((L^u_{x + \point}))$), but we will not need this fact in what follows.  \hfill $\square$
\end{remark}

\section{Super-exponential decay}
\setcounter{equation}{0}

In this section, we consider a local function $F$ satisfying (\ref{2.1}), as well as a non-empty finite subset $\Sigma$ of $(0,\infty)$, $\kappa \in (0,1)$, and $\mu \ge 0$ (with $\mu = 0$ only in the case when (\ref{2.2}) additionally holds, i.e.~when~$\theta$ is strictly increasing). We introduce large boxes centered at the origin $[-N_L,N_L]^d$, with $N_L$ of order $L^{d/2} / \sqrt{\log L}$, and show in the main Proposition \ref{prop3.1} that most boxes $B_z$, $z \in \IL$ (see (\ref{1.19})) within such large boxes are $(\Sigma, \kappa, \mu)$-good, except on a set with exponentially decaying probability at a rate faster than $N_L^{d-2}$, as $L$ tends to infinity. The main preparation has already taken place in the previous section, and the arguments in this section are similar to the proofs of Proposition 5.4 of \cite{Szni15} or of Theorem 5.1 of \cite{Szni17}.

\medskip
More precisely, we consider $F$ as in (\ref{2.1}) as well as
\begin{equation}\label{3.1}
\begin{array}{l}
\mbox{$\Sigma$ a non-empty finite subset of $(0,\infty)$, $\kappa \in (0,1)$ and $\mu \ge 0$,}
\\
\mbox{where $\mu = 0$ is allowed if $F$ satisfies (\ref{2.2}), and $\mu > 0$ otherwise.}
\end{array}
\end{equation}

\n
We recall that $F_0$ denotes the local function $F_0(\ell) = \ell$ for $\ell \ge 0$, see (\ref{2.3}), and in the notation of Proposition \ref{prop2.2} we introduce the constant 
\begin{equation}\label{3.2}
c_3 (\Sigma, \kappa, \mu, F) = \sup\limits_{\alpha \in \Sigma} \,c_1(\alpha, \kappa,\mu, F) \vee c_1(\alpha, \kappa, \mu = 0, F_0) ( \ge 100).
\end{equation}
In this section, from now on, we assume that
\begin{equation}\label{3.3}
K \ge c_3(\Sigma,\kappa,\mu, F).
\end{equation}
For $L \ge 1$ and $\cC$ a non-empty finite subset of $\IL$ as in (\ref{2.12}), and $B \in \cC$ (i.e.~$B = B_z$ with $z \in \cC$), we define (see (\ref{2.27}) for notation)
\begin{equation}\label{3.4}
\wt{\cB}^{B,F}_{\Sigma,\kappa,\mu} = \textstyle{\bigcup\limits_{\alpha \in \Sigma}} \,\big(\wt{\cB}^{B,F}_{\alpha,\kappa,\mu} \cup \wt{B}^{B,F_0}_{\alpha, \kappa, \mu = 0}\big),
\end{equation}
as well as
\begin{equation}\label{3.5}
\eta = \IQ^\cC [\wt{\cB}^{B,F}_{\Sigma,\kappa,\mu}].
\end{equation}

\medskip\n
(this probability does not depend on the choice of $\cC$ or of $B \in \cC$). By (\ref{2.30}) of Proposition \ref{prop2.2} and (\ref{3.3}) above, we know that
\begin{equation}\label{3.6}
\lim\limits_{L \r \infty} \; \mbox{\f $\dis\frac{1}{\log L}$} \; \log \eta = - \infty .
\end{equation}
We define $\rho$ (that depends on $F, \Sigma,\kappa,\mu, K$ and $L$) as
\begin{equation}\label{3.7}
\mbox{$\rho = \sqrt{\mbox{\f $\dis\frac{\log L}{|\log \eta |}$}}$,  so that $\lim\limits_{L \r \infty} \; \rho =0$}.
\end{equation}

\n
We sometimes write $\eta_L$ and $\rho_L$ to underline the $L$-dependence of these quantities. In addition, we note that by (\ref{2.17}) (and just as in (\ref{2.28}))
\begin{equation}\label{3.8}
\mbox{under $\IQ^\cC$, the events $\wt{B}^{B,F}_{\Sigma,\kappa,\mu}$, $B \in \cC$ are i.i.d. with probability $\eta$}.
\end{equation}

\n
We then choose $N_L$ for $L \ge 2$ (this size will roughly correspond to $N$ in Sections 4 and 5):
\begin{equation}\label{3.9}
\mbox{$N_L = L^{\frac{d}{2}} (\log L)^{-\frac{1}{2}}$, so that $(\frac{N_L}{L})^d \asymp N_L^{d-2} / \log N_L$, as $L \r \infty$},
\end{equation}

\medskip\n
where ``$\asymp$'' means that the ratio of both sides remains bounded away from $0$ and $\infty$, as $L$ goes to infinity. We then introduce the ``bad event'' (see (\ref{2.76}) for notation)
\begin{equation}\label{3.10}
\begin{array}{l}
\cB = \{\mbox{there are at least $\rho_L(\frac{N_L}{L})^d$ boxes $B_z$ intersecting $[-N_L,N_L]^d$,}
\\
\mbox{which are $(\Sigma,\kappa,\mu)$-bad}\}.
\end{array}
\end{equation}

\n
The main result of this section shows that $\cB$ is a negligible event for our (later) purpose. The proof is similar to that of Proposition 5.4 of \cite{Szni15} or Theorem 5.1 of \cite{Szni17}, but is briefly sketched for the reader's convenience.

\begin{proposition}\label{prop3.1}
Assume that $F$ satisfies (\ref{2.1}), $\Sigma, \kappa, \mu$ satisfies (\ref{3.1}) and $K \ge c_3(\Sigma, \kappa, \mu, F)$ as in (\ref{3.2}). Then, we have
\begin{equation}\label{3.11}
\lim\limits_{L \r \infty} \;N_L^{-(d-2)} \log \,\IP[\cB] = - \infty.
\end{equation}
\end{proposition}

\begin{proof}
We write $\IL$ as the disjoint union of the sets $L \,\tau + \ov{K} \,L\,\IZ^d$, where $\tau$ runs over $\{0,\dots,\ov{K}-1\}^d$ (recall $\ov{K} = 2 K + 3$, see (\ref{2.12})). For such a $\tau$ we write
\begin{equation}\label{3.12}
\cC^\tau = \{z \in L \,\tau + \ov{K} \, L \, \IZ^d; \;B_z \cap [-N_L,N_L]^d \not= \emptyset\}.
\end{equation}

\n
Making use of (\ref{2.29}) (and $K \ge c_3$), we see that for large $L$ and any $\tau \in \{0,\dots,\ov{K}-1\}^d$, and $B \in \cC^\tau$, one has $\cB^{B,F}_{\Sigma, \kappa, \mu} \subseteq \wt{\cB}^{B,F}_{\Sigma, \kappa,\mu}$, so that
\begin{equation}\label{3.13}
\begin{split}
\IP[\cB] & \le \dsl_\tau \, \IP[ \mbox{at least $\ov{K}\,^{-d} \rho_L\big(\frac{N_L}{L}\big)^d$ boxes $B \in \cC^\tau$ are $(\Sigma, \kappa, \mu)$-bad}]
\\[-1ex]
&\!\!\! \stackrel{(\ref{2.29})}{\le} \dsl_\tau \, \IQ^{\cC^\tau} [ \wt{\cB}^{B,F}_{\Sigma, \kappa, \mu} \; \mbox{occurs for at least $\ov{K}\,^{-d} \rho_L\big(\frac{N_L}{L}\big)^d$ boxes $B \in \cC^\tau] \stackrel{\rm def}{=} A$}.
\end{split}
\end{equation}

\n
We now make use of (\ref{3.8}). We let $U_i, i \ge 1$, stand for i.i.d. Bernoulli variables with success probability $\eta_L$. Then, $m = [\ov{c}_4 (\frac{N_L}{L})^d]$ is an upper bound on $|\cC^\tau|$ for each $\tau \in \{0,\dots,\ov{K}-1\}^d$, and by usual exponential bounds on the sums of i.i.d. variables, we see that
\begin{equation}\label{3.14}
A \le \ov{K}\,^d \,P\Big[\dsl_{i=1}^m \,U_i \ge \ov{K}\,^{-d} \,\rho_L \Big(\mbox{\f $\dis\frac{N_L}{L}$}\Big)^d \Big] \le \ov{K}\,^d \,\exp\{ -m \, I_L\},
\end{equation}
where we have $I_L = 0$ if $\eta_L = 1$ and otherwise
\begin{equation}\label{3.15}
I_L = \wt{\rho} \log \;\mbox{\f $\dis\frac{\wt{\rho}}{\eta}$} + (1 - \wt{\rho}) \; \log \;\mbox{\f $\dis\frac{1- \wt{\rho}}{1 - \eta}$}, \;\mbox{with}\;\; \wt{\rho} = \min  \Big\{\mbox{\f $\dis\frac{1}{m\,\ov{K}\,\!^d}$} \;\rho \,\Big(\mbox{\f $\dis\frac{N_L}{L}$}\Big)^d, \,1\Big\}.
\end{equation}
Note that for large $L$, $\wt{\rho} \sim \frac{1}{\ov{c}_4 \,\ov{K}\,^{\!d}}\, \rho$.

\medskip
We will now see that $N_L^{d-2}$ is small compared to $m\,I_L$, as $L \r \infty$, see (\ref{3.19}) below. Indeed, for large $L$, one has
\begin{equation}\label{3.16}
\begin{split}
\log \;\mbox{\f $\dis\frac{\wt{\rho}}{\eta}$} & = \log \rho + \log \mbox{\f $\dis\frac{1}{\eta}$} + \log \Big(\Big(\mbox{\f $\dis\frac{N_L}{L}$}\Big)^d\; \mbox{\f $\dis\frac{1}{m\,\ov{K}\,\!^d}$}\Big)
\\
&\!\!\!\stackrel{(\ref{3.7})}{=} \mbox{\f $\dis\frac{1}{2}$} \;\log \log L - \mbox{\f $\dis\frac{1}{2}$}\;\log \log \; \mbox{\f $\dis\frac{1}{\eta}$} + \log \;\mbox{\f $\dis\frac{1}{\eta}$} + \log \Big(\Big(\mbox{\f $\dis\frac{N_L}{L}$}\Big)^d\; \mbox{\f $\dis\frac{1}{m\,\ov{K}\,\!^d}$}\Big) \stackrel{(\ref{3.6})}{\sim} \log \mbox{\f $\dis\frac{1}{\eta}$}.
\end{split}
\end{equation}
Hence, for $L \r \infty$, we have
\begin{align}
I_L \sim &\;\wt{\rho} \;\log \Big(\mbox{\f $\dis\frac{\wt{\rho}}{\eta}$}\Big) \sim \wt{\rho} \;\log \;\mbox{\f $\dis\frac{1}{\eta}$}, \;\mbox{so that} \label{3.17}
\\[1ex]
m\,I_L \sim &\; \ov{K}\,^{-d} \rho \Big(\mbox{\f $\dis\frac{N_L}{L}$}\Big)^d \; \log \; \mbox{\f $\dis\frac{1}{\eta}$} \stackrel{(\ref{3.7})}{=}  \ov{K}\,^{-d} \sqrt{\log L \, \log \textstyle{\frac{1}{\eta}}} \;\;\Big(\mbox{\f $\dis\frac{N_L}{L}$}\Big)^d \stackrel{(\ref{3.9})}{\ge} \label{3.18}
\\
&  c\,\ov{K}\,^{-d}  \sqrt{\log L \, \log \textstyle{\frac{1}{\eta}}}\;\; \mbox{\f $\dis\frac{N_L^{d-2}}{\log N_L}$}  \stackrel{(\ref{3.9})}{\ge}   c'\,\ov{K}\,^{-d} \Big(\mbox{\f $\dis\frac{\log \frac{1}{\eta}}{\log L}$}\Big)^{\frac{1}{2}} \,N^{d-2}_L.\nonumber
\end{align}
By (\ref{3.6}) it follows  that
\begin{equation}\label{3.19}
N_L^{d-2} = o(m\,  I_L), \;\mbox{as $L \r \infty$}.
\end{equation}

\medskip\n
Inserting this information in (\ref{3.14}) and (\ref{3.13}) yields the claim (\ref{3.11}). This concludes the proof of Proposition \ref{prop3.1}.
\end{proof}

\section{Large deviation lower bounds}
\setcounter{equation}{0}

The principal object of this section is the proof of Theorem \ref{theo4.2} that states our main asymptotic lower bound on the probability of the excess deviation event $\cA_N$, see (\ref{0.7}) and (\ref{4.6}) below. The proof proceeds in a number of steps. First, with the help of the super-exponential estimates of the previous section, we replace the event $\cA_N$ by an event $\cA^1_N$ that imposes lower bounds on the average value $\langle m_B, L^u \rangle$ of the occupation time in the boxes $B$ of size $\sim N^{\frac{2}{d}} (\log N)^{\frac{1}{d}}$ inside $B(0,N)$, see Proposition \ref{prop4.3}. Then, we obtain the main asymptotic lower bound on $\IP[\cA^1_N]$ with the help of the change of probability method and {\it tilted interlacements} that were recalled in Section 1, see Proposition \ref{prop4.4}. The central estimate showing that $\cA^1_N$ is a typical event for the tilted interlacements appears in Lemma \ref{lem4.5}. The proof of the Theorem \ref{theo4.2} is completed at the end of this section and Remark \ref{rem4.6} describes some extensions.

\medskip
We now introduce our assumptions. We consider a local function $F$ satisfying (\ref{2.1}), with associated range $R \ge 0$ and function $\theta(v) = \IE[F((L^v_\point))]$, $v \ge 0$. Recall that $\theta$ is non-decreasing, continuous, and $\theta(0) = 0$. We additionally require that
\begin{equation}\label{4.1}
\mbox{$F$ is bounded and not identically equal to $0$}.
\end{equation}
In particular (\ref{2.2}) holds, and we write
\begin{equation}\label{4.2}
\theta_\infty = \lim\limits_{v \r \infty} \;\theta(v) \in (0, \infty).
\end{equation}

\begin{remark}\label{rem4.1} \rm 1) The examples (\ref{2.4}), (\ref{2.5}), (\ref{2.6}), (\ref{2.7}) all satisfy (\ref{4.1}), and $\theta_\infty = 1$ in all four cases.

\bigskip\n
2) Since $\theta$ is a non-decreasing, continuous, bounded function, it readily follows that
\begin{equation}\label{4.3}
\mbox{$\theta$ is a uniformly continuous function on $\IR_+$}.
\end{equation}

\n
In addition, by monotonicity of the non-negative local function $F$, we have $\|F\|_\infty = \lim_{v \r \infty} F((L^v_\point))$, $\IP$-a.s., see also below (\ref{2.1}). By monotone convergence, one finds with (\ref{4.2}) that $F$ is bounded with sup-norm $\theta_\infty$:
\begin{equation}\label{4.4}
\| F\|_\infty = \lim\limits_{v \r \infty} \;\theta(v) = \theta_\infty < \infty.
\end{equation}

\vspace{-3ex}
\hfill $\square$
\end{remark}

Given a level $u > 0$, we will sample $F((L^u_{x + \point}))$ at the points $x$ of the discrete blow-up $D_N = (DN) \cap \IZ^d$, $N \ge 1$, of a compact model shape $D \subseteq \IR^d$ satisfying (\ref{0.6}). We now come to the main result of this section.

\begin{theorem}\label{theo4.2} (large deviation lower bound)

\smallskip\n
Assume that the local function $F$ satisfies (\ref{2.1}) and (\ref{4.1}).  For $N \ge 1$, let $D_N = (ND) \cap \IZ^d$, where $D \subseteq \IR^d$ is a compact set satisfying (\ref{0.6}). Then, for any $u > 0$ and $\nu$ in $(\theta(u),\theta_\infty)$, one has
\begin{equation}\label{4.5}
\begin{array}{l}
\liminf_N \; \mbox{\f $\dis\frac{1}{N^{d-2}}$} \;\log \IP[\cA_N] 
\\[1ex]
\ge - \inf\Big\{ \fd \;\dis\int_{\IR^d} | \nabla \varphi|^2dz; \;\varphi \in C^\infty_0(\IR^d) \;\mbox{and}
\; \strokedint_D \theta\big((\sqrt{u} + \varphi)^2\big) \,dz > \nu\Big\}  = - I^D_\nu ,
\end{array}
\end{equation}
where we have set
\begin{equation}\label{4.6}
\cA_N = \Big\{\dsl_{x \in D_N} \;F\big((L^u_{x + \point})\big) > \nu \, |D_N|\Big\},
\end{equation}
and for $a \ge 0$,
\begin{equation}\label{4.7}
I^D_a = \inf \Big\{ \fd \; \dis\int_{\IR^d} | \nabla \varphi |^2 \,dz; \, \varphi \ge 0, \;\varphi \in C^\infty_0 (\IR^d) \;\mbox{and} \; \strokedint_D \theta\big((\sqrt{u} + \varphi)^2\big) \,dz > a\Big\}.
\end{equation}

\n
($\strokedint_D \dots dz$ refers to the normalized integral $\frac{1}{|D|} \int_D \dots dz$, as in (\ref{0.8}), and the equality in the second line of (\ref{4.5}) will be proved at the end of this section).
\end{theorem}

As mentioned at the beginning of this section, we will prove Theorem \ref{theo4.2} in a number of steps. We now keep $u > 0$ and $\nu \in (\theta(u), \theta_\infty)$ fixed. We intend to use the super-exponential estimates of Section 3, and for this purpose we relate the scale $L$ of the previous section to the scale $N$ of the present section. In essence, $L$ is chosen so that the number of boxes $B = B_z$, $z \in \IL$ (see (\ref{1.19})) that intersect $B(0,N)$ is of order $N^{d-2}/\log N$ for large $N$. More precisely, we set for $N \ge 1$
\begin{equation}\label{4.8}
L = \big[N^{\frac{2}{d}} (\log N)^{\frac{1}d}\big]\; \vee 2.
\end{equation}
Then, with $N_L$ as in (\ref{3.9}), we have
\begin{equation}\label{4.9}
N_L =\mbox{\f $\dis\frac{L^{\frac{d}{2}}}{(\log L)^{\frac{1}{2}}}$}  \sim {\textstyle \sqrt{\frac{d}{2}}} \; N, \;\;\mbox{as} \;\;N \r \infty.
\end{equation}
We assume that in addition to (\ref{0.6}), that $D$ satisfies
\begin{equation}\label{4.10a}
D \subseteq [-1,1]^d .
\end{equation}
We will later see at the end of this section that ``by scaling'', the general case can be reduced to this special case (see above Remark \ref{rem4.6}). We then consider a fixed
\begin{equation}\label{4.10}
\begin{array}{l}
\mbox{non-negative, smooth, compactly supported function $\varphi$ on $\IR^d$}, 
\\
\mbox{positive on $[-2,2]^d$},
\end{array}
\end{equation}
and such that
\begin{equation}\label{4.11}
\strokedint_D \;\theta\big(\psi^2(z)\big) \, dz > \nu, \; \mbox{with $\psi = \sqrt{u} + \varphi$}.
\end{equation}

\n
Our next step is to introduce an event $\cA^1_N$, see (\ref{4.21}) below, which up to a ``negligible probability'' is contained in $\cA_N$, and solely consists of constraints on the average occupation times $\langle m_B, L^u \rangle$ in boxes $B = B_z, z \in \IL$, contained in $B(0,N)$. We will later see that this event $\cA^1_N$ is typical for the tilted interlacements that we will introduce in (\ref{4.28}) below, see Lemma \ref{lem4.5}. We first need some notation.

\medskip
We introduce the parameter $\ve(D,\theta,\psi,\nu) > 0$ via (see (\ref{4.11})):
\begin{equation}\label{4.12}
\strokedint_D \;\theta(\psi^2) \, dz = \nu + 10 \ve.
\end{equation}
We then choose a finite non-empty set $\Sigma (\ve, \theta,\psi) \subseteq (0,\infty)$ such that
\begin{align}
&\mbox{$\theta(\cdot)$ varies at most by $\ve$ between consecutive points of $\Sigma$},\label{4.13}
\\[1ex]
&\mbox{$\Sigma \cap (0, \min\limits_{[-2,2]^d} \,\psi^2)$ and $\Sigma \cap (\sup\limits_{[-2,2]^d} \,\psi^2, \infty)$, each contain at least 5 points} \label{4.14}
\end{align}

\n
(the first condition will be used in (\ref{4.20}) below, and the second condition would come up in the context of Remark \ref{rem4.6} 1)\,).

\medskip\n
In essence, $\Sigma$ stands as a discretization of the range of $\psi^2$. We then introduce
\begin{equation}\label{4.15}
\mbox{$\Delta =$ the minimal distance between points of $\Sigma$},
\end{equation}
as well as $\kappa(\ve,\theta,\psi) \in (0,1)$ such that
\begin{equation}\label{4.16}
\mbox{for each $\alpha \in \Sigma$, $\big(\alpha(1-\kappa), \alpha (1 + \kappa)\big) \cap \Sigma = \{\alpha\}$ \;(i.e.~(\ref{2.78}) holds)},
\end{equation}
together with $\mu \ge 0$ specified via (since (\ref{2.2}) holds, so (\ref{3.1}) is fulfilled)
\begin{equation}\label{4.17}
\mu = 0 .
\end{equation}
With $\psi = \sqrt{u} + \varphi$ as in (\ref{4.11}), we introduce the functions on $\IZ^d$:
\begin{equation}\label{4.18}
f(x) = \psi \Big(\mbox{\f $\dis\frac{x}{N}$}\Big),\; \wt{f}(x) = \mbox{\f $\dis\frac{1}{\sqrt{u}}$} \;f(x) = 1 +  \mbox{\f $\dis\frac{1}{\sqrt{u}}$} \;\varphi\,\Big(\mbox{\f $\dis\frac{x}{N}$}\Big), \;\;\mbox{for $x \in \IZ^d$}.
\end{equation}
From now on, we assume that (with $L$ as in (\ref{4.8}))
\begin{equation}\label{4.19}
\begin{array}{l}
\mbox{$N$ is large enough so that $N \ge 10L$ and the oscillations of $\psi^2$ and}
\\
\mbox{$\theta(\psi^2)$ on a box of side-length $\frac{L}{N}$ in $\IR^d$ are respectively smaller}
\\
\mbox{than $\Delta$ in (\ref{4.15}) and $\ve$ in (\ref{4.12})}.
\end{array}
\end{equation}

\n
For $B = B_z, z \in \IL$, we recall the notation $x_B = z$ from (\ref{1.20}). When $B \subseteq B(0,N)$, we choose $f_{1,B} > f_{0,B} > 0$ such that (see (\ref{4.14}))
\begin{equation}\label{4.20}
\left\{ \begin{array}{rl}
{\rm i)} & f^2_{1,B} \in \Sigma \cap \big(0, f^2(x_B)\big) \;\;\mbox{and} \;\; | \Sigma \cap [f^2_{1,B}, f^2(x_B)] \,| = 3,
\\[2ex]
{\rm ii)} & f^2_{0,B} \in \Sigma \cap \big(0, f^2(x_B)\big) \;\;\mbox{and} \;\; | \Sigma \cap [f^2_{0,B}, f^2(x_B)] \,| = 5,
\end{array}\right.
\end{equation}
and introduce the (crucial) event
\begin{equation}\label{4.21}
\cA^1_N = \bigcap\limits_{B \subseteq B(0,N)} \{\langle m_B,L^u \rangle \ge f^2_{1,B}\}
\end{equation}

\n
(where $B$ is the form $B_z,z \in \IL$ in the above intersection).

\medskip
With the help of the super-exponential bounds from the last section, we will now reduce the task of bounding $\IP[\cA_N]$ from below to that of bounding $\IP[\cA^1_N]$ from below.

\begin{proposition}\label{prop4.3}
\begin{equation}\label{4.22}
\liminf\limits_N \;\mbox{\f $\dis\frac{1}{N^{d-2}}$} \;\log \IP[\cA_N] \ge \liminf\limits_N \; \mbox{\f $\dis\frac{1}{N^{d-2}}$} \; \log \IP[\cA_N^1]. 
\end{equation}
\end{proposition}

\begin{proof}
With the above choices of $F,\Sigma,\kappa,\mu$ and with $K \ge c_3(\Sigma,\kappa,\mu,F)$ from Proposition \ref{prop3.1}, we consider the bad event $\cB$ from (\ref{3.10}). Our main objective is to show in (\ref{4.25}) below that for large $N$, $\cA_N$ contains $\cA^1_N \backslash \cB$, and our claim (\ref{4.22}) will then quickly follow from the super-exponential bounds in Proposition \ref{prop3.1}.

\medskip
With this in mind, we note that for large $N$, by (\ref{4.9}), $B(0,N) \subseteq [-N_L ,N_L]^d$. As a result, for large $N$, on $\cA^1_N \backslash \cB$ there are at most $\rho_L(\frac{N_L}{L})^d$ boxes $B \subseteq B(0,N)$ that are $(\Sigma, \kappa,\mu)$-bad. By (\ref{2.79}) and (\ref{2.77}) of Lemma \ref{lem2.4}, and (\ref{4.20}), we find that for large $N$
\begin{equation}\label{4.23}
\begin{array}{l}
\mbox{on $\cA^1_N \backslash \cB$ there are at most $\rho_L(\frac{N_L}{L})^d$ boxes $B \subseteq B(0,N)$ such that}
\\
\langle m_B, L^u \rangle \ge f^2_{1,B} \; \mbox{but}\; \sum\limits_{x \in B} F\big((L^u_{x + \point})\big) < \big(\theta(f^2_{0,B})-\mu\big) |B| \stackrel{(\ref{4.17})}{=} \theta(f^2_{0,B}) \,|B|.
\end{array}
\end{equation}
As a result, we see that for large $N$ on $\cA^1_N \backslash \cB_N$, we have (recall (\ref{4.4})):
\begin{equation}\label{4.24}
\begin{split}
\sum\limits_{x \in D_N} \;F\big((L^u_{x + \point})\big) & \ge \dsl_{B \subseteq D_N} \;\theta(f^2_{0,B}) \, |B| - \theta_\infty \,\rho_L \,N_L^d
\\
&\hspace{-5ex} \stackrel{(\ref{4.13}),(\ref{4.20}), {\rm ii)}}{\ge}  \dsl_{B \subseteq D_N} \,\big(\theta(f^2(x_B)\big) - 5 \ve)_+ |B| - \theta_\infty \,\rho_L \,N_L^d
\\
&\!\!\!\!\stackrel{(\ref{4.19})}{\ge} \dsl_{z \in \IL;z + [0,L)^d \subseteq N D} \;\dis\int_{z + [0,L]^d} \Big(\theta\big(\psi^2 \big(\textstyle \frac{w}{N}\big)\big) - 6 \ve\Big)_+ \,dw - \theta_\infty \,\rho_L\,N^d_L
\\
&\!\!\!\!\stackrel{(\ref{4.12}), (\ref{0.6})}{\ge} (\nu + 9 \ve) \,N^d |D| - 6 \ve\,N^d |D| - \theta_\infty \,\rho_L\,N^d_L
\\
& \ge (\nu + \ve)\,N^d |D| \stackrel{(\ref{0.6})}{>} \nu \,|D_N|
\end{split}
\end{equation}
(where $|D|$ denotes the Lebesgue measure of $D$ but $|D_N|$ stands for the number of points in $D_N \subseteq \IZ^d$). This proves that
\begin{equation}\label{4.25}
\mbox{for large $N$, $\cA_N \supseteq \cA^1_N \backslash \cB$}.
\end{equation}
Taking into account the super-exponential decay of $\IP[\cB]$ proven in Proposition \ref{prop3.1}, the claim (\ref{4.22}) follows.
\end{proof}

Our next step towards the proof of Theorem \ref{theo4.2} is to derive an asymptotic lower bound on $\IP[\cA^1_N]$. Recall (\ref{4.10}) and (\ref{4.21}) for notation. 

\begin{proposition}\label{prop4.4}
\begin{equation}\label{4.26}
\liminf\limits_N \; \mbox{\f $\dis\frac{1}{N^{d-2}}$} \; \log \IP[\cA^1_N] \ge - \fd \; \dis\int_{\IR^d} |\nabla \varphi|^2 dz. 
\end{equation}
\end{proposition}

\begin{proof}
We will use the change of probability method, and for this purpose we consider tilted interlacements, see (\ref{1.14}). Namely, with $\wt{f}$ as in (\ref{4.18}), we set
\begin{equation}\label{4.27}
\wt{V} = -\mbox{\f $\dis\frac{L\,\wt{f}}{\wt{f}}$} \;\;\mbox{(a finitely supported function)},
\end{equation}

\n
and introduce the probability governing the tilted interlacements:
\begin{equation}\label{4.28}
\wt{\IP}_N = e^{\langle L^u,\wt{V}\rangle} \IP \;\; \mbox{(we denote by $\wt{\IE}_N$ the corresponding expectation)}.
\end{equation}
By the relative entropy inequality (see p.~76 of \cite{DeusStro89}), one knows that
\begin{equation}\label{4.29}
\IP[\cA^1_N] \ge \wt{\IP}_N[\cA^1_N] \; \exp\Big\{- \mbox{\f $\dis\frac{1}{\wt{\IP}_N[\cA^1_N]}$} \;\Big(H(\wt{\IP}_N \,| \, \IP) +\mbox{\f $\dis\frac{1}{e}$} \Big)\Big\},
\end{equation}
where
\begin{equation}\label{4.30}
\mbox{$H(\wt{\IP}_N \,| \, \IP) = \wt{\IE}_N \,\Big[\log \;\mbox{\f $\dis\frac{d\,\wt{\IP}_N}{d \,\IP}$}\Big]$ is the relative entropy of $\wt{\IP}_N$ with respect to $\IP$}.
\end{equation}
By the same calculation as in (\ref{2.23}) of \cite{LiSzni14}, we have
\begin{equation}\label{4.31}
\begin{split}
H(\wt{\IP}_N \, | \, \IP) & = \wt{\IE}_N [\langle L^u,\wt{V}\rangle] = - u \dsl_{x \in \IZ^d} \wt{f}(x) \,L \wt{f}(x)
\\[-1ex]
&\!\!\!\! \stackrel{(\ref{4.18})}{=} - \dsl_{x \in \IZ^d} f(x)\,Lf(x) = - \dsl_{x \in \IZ^d} \big(f(x) - \sqrt{u}\big) \,Lf(x)
\\
&  = \cE (f - \sqrt{u}, f - \sqrt{u}) \;\; \mbox{(with $\cE(\cdot,\cdot)$ the Dirichlet form, see (\ref{1.2}))}.
\end{split}
\end{equation}
In view of (\ref{4.18}), we thus find that
\begin{equation}\label{4.32}
\begin{split}
\lim\limits_N \;\mbox{\f $\dis\frac{1}{N^{d-2}}$} \; H(\wt{\IP}_N \, | \, \IP) & = \lim\limits_N \;\mbox{\f $\dis\frac{1}{N^{d-2}}$} \; \mbox{\f $\dis\frac{1}{4d}$}  \; \dsl_{y \sim x} \;\Big(\varphi \big(\frac{y}{N}\big) - \varphi \big(\frac{x}{N}\big) \Big)^2
\\
& = \fd \;\dis\int_{\IR^d} \,|\nabla \varphi |^2 dz ,
\end{split}
\end{equation}
where in the last step we have used the same Riemann-sum argument as in (\ref{2.27}) - (\ref{2.28}) of \cite{LiSzni14}. Taking into account (\ref{4.29}) and (\ref{4.32}), the proof of Proposition \ref{prop4.4} will be completed once we show
\begin{lemma}\label{lem4.5}
\begin{equation}\label{4.33}
\lim\limits_N \;\wt{\IP}_N [\cA^1_N] = 1.
\end{equation}
\end{lemma}

\begin{proof}
We recall the notation (\ref{1.11}) for $\wt{G}$ where $\wt{f}$ is now chosen as in (\ref{4.18}). The proof of (\ref{4.33}) relies on the following estimate (recall (\ref{4.8})):
\begin{equation}\label{4.34}
\mbox{for all $N \ge 1$, and $B = B_z$, $z \in \IL$, $c_4(u,\varphi) \,\|\wt{G} \, 1_B\|_\infty \le \fr \;L^2$}.
\end{equation}
Let us admit (\ref{4.34}) for the time being and explain how (\ref{4.33}) follows. By (\ref{4.34}), we see that for $N$ and $B$ as in (\ref{4.34}) and $|a| \le c_4$, one has:
\begin{equation}\label{4.35}
\begin{split}
\Big\| \Big(I -\mbox{\f $\dis\frac{a}{L^2}$} \;\wt{G} \,1_B\Big)^{-1} \,1 - 1\Big\|_\infty & \le \dsl_{k \ge 1} \;\Big(\mbox{\f $\dis\frac{|a|}{L^2}$} \;\|\wt{G} \,1_B\|_\infty\Big)^k
\\
&\!\!\!\! \stackrel{(\ref{4.34})}{\le} \dsl_{k \ge 1} \;\Big(\mbox{\f $\dis\frac{|a|}{2c_4}$}\Big)^k = \mbox{\f $\dis\frac{|a|}{2c_4}$} \; \Big(1 - \mbox{\f $\dis\frac{|a|}{2c_4}$}\Big)^{-1} \le  \mbox{\f $\dis\frac{|a|}{c_4}$} \;.
\end{split}
\end{equation}
It now follows from Lemma \ref{lem1.2} applied to $V = \frac{-a}{L^2} \;1_B$, where $0 \le a \le c_4$, that
\begin{equation}\label{4.36}
\begin{split}
\wt{\IE}_N \Big[\exp\Big\{- \faa\;\dsl_{x \in B} L^u_x\Big\}\Big] & = \exp\Big\{ - u \;  \faa \; \Big\langle 1_B, \Big(I + \faa \;\wt{G} \,1_B\Big)^{-1} \,1\Big\rangle_{\wt{\lambda}}\Big\}
\\[-1ex]
&\!\!\!\! \stackrel{(\ref{4.35})}{\le} \exp\Big\{- u \;  \faa\;\wt{\lambda}(B) \Big(1 - \mbox{\f $\dis\frac{a}{c_4}$}\Big)\Big\},
\end{split}
\end{equation}
where $\wt{\lambda}(B) = \sum_{x \in B} \wt{f}^2(x)$ (see (\ref{1.8})).

\medskip
Hence, by the exponential Chebyshev inequality, we find that for large $N$, $B \subseteq B(0,N)$ and $0 \le a \le c_4$, we have in the notation of (\ref{4.20})
\begin{equation}\label{4.37}
\wt{\IP}_N[\langle m_B, L^u \rangle  < f^2_{1,B}] \le \exp\Big\{\faa\;\dsl_{x \in B} \,\Big(f^2_{1,B} - \Big(1 -  \mbox{\f $\dis\frac{a}{c_4}$}\Big)\,\wt{f}^2(x)\Big)\Big\}.
\end{equation}

\n
Recall the definition of $\Delta$ in (\ref{4.15}). Then, by (\ref{4.20}) i), we have $f^2_{1,B} \le f^2(x_B) - 2 \Delta$, and by (\ref{4.19}), for all $x \in B$, $f^2(x) \ge f^2(x_B) - \Delta$. As a result, we have
\begin{equation}\label{4.38}
\begin{split}
\dsl_{x \in B} \, \Big(f^2_{1,B} - \Big(1 -  \mbox{\f $\dis\frac{a}{c_4}$}\Big)\,\wt{f}^2(x)\Big) & \le |B| \;\Big\{f^2(x_B) - 2 \Delta - \Big(1 -  \mbox{\f $\dis\frac{a}{c_4}$}\Big) (f^2(x_B) - \Delta )\Big\}  
\\
& = |B| \; \Big\{ - \Big(1 +  \mbox{\f $\dis\frac{a}{c_4}$}\Big) \;\Delta +   \mbox{\f $\dis\frac{a}{c_4}$} \;f^2(x_B)\Big\}
\\
&\!\!\!\! \stackrel{(\ref{4.18})}{\le} |B| \; \Big\{- \Big(1 + \mbox{\f $\dis\frac{a}{c_4}$}\Big)\;\Delta +  \mbox{\f $\dis\frac{a}{c_4}$} \;\|\psi \|^2_\infty\Big\}.
\end{split}
\end{equation}
We then choose $a$ via
\begin{equation}\label{4.39}
\mbox{\f $\dis\frac{a}{c_4}$} = \fr \; \mbox{\f $\dis\frac{\Delta}{\|\psi \|^2_\infty - \Delta}$} \; \mbox{($< \frac{1}{2}$ by (\ref{4.13}), (\ref{4.14}))}.
\end{equation}
It now follows from (\ref{4.37}) that for large $N$ and all $B \subseteq B(0,N)$
\begin{equation}\label{4.40}
\wt{\IP}_N \,[ \langle m_B, L^u \rangle < f^2_{1,B}] \le \exp\Big\{ -\mbox{\f $\dis\frac{c_4}{4}$}  \; \mbox{\f $\dis\frac{\Delta^2}{\|\psi \|^2_\infty - \Delta}$} \;L^{d-2}\Big\}.
\end{equation}

\n
Since the number of boxes $B \subseteq B(0,N)$ grows polynomially in $N$, with $L$ as in (\ref{4.8}), and $\cA^1_N$ defined by (\ref{4.21}), the above bound readily implies (\ref{4.33}).

\medskip
So, there remains to prove (\ref{4.34}). For this purpose we consider $t  > 0$, as well as $N$ and $B$ as in (\ref{4.34}). Recall $c_0$ from. Lemma \ref{lem1.1}. Then, for any $x \in \IZ^d$, by the relative entropy inequality, and the relative entropy calculations from (1.36) - (1.41) of \cite{LiSzni14}, we find that
\begin{align}
&\wt{E}_x \Big[\dis\int^t_0 \; \mbox{\f $\dis\frac{c_0}{L^2}$} \;1_B(X_s)\,ds\Big] \le \log E_x \big[e^{\int^t_0 \,\frac{c_0}{L^2} \;1_B(X_s)ds}\big] + \wt{E}_x \Big[\dis\int^t_0 H(X_s)\,ds\Big], \;\mbox{where} \label{4.41}
\\[1ex]
&H(y) = \fd \;\dsl_{|e| = 1} \;\Big\{\mbox{\f $\dis\frac{\wt{f}(y+e)}{\wt{f}(y)}$} \;\log \;\mbox{\f $\dis\frac{\wt{f}(y+e)}{\wt{f}(y)}$} - \mbox{\f $\dis\frac{\wt{f}(y+e) - \wt{f}(y)}{\wt{f}(y)}$}\Big\} \ge 0 \;\; \mbox{for $y \in \IZ^d$}. \label{4.42}
\end{align}
By inspection of $\wt{f}$ in (\ref{4.18}), we see that
\begin{equation}\label{4.43}
0 \le H(y) \le \mbox{\f $\dis\frac{c(u,\varphi)}{N^2}$} \;1 \{|y|_\infty \le  c_5(u,\varphi)\, N \}
\end{equation}

\n
Inserting this bound in (\ref{4.41}), and letting $t$ tend to infinity, we find with (\ref{1.4})  that
\begin{equation}\label{4.44}
\wt{E}_x\;  \Big[ \dis\int^\infty_0 \mbox{\f $\dis\frac{c_0}{L^2}$} \;1_B(X_s)\,ds\Big] \le \log 2  + \mbox{\f $\dis\frac{c(u,\varphi)}{N^2}$} \wt{E}_x \Big[\dis\int^\infty_0 1\{|X_s|_\infty \le c_5(u,\varphi)\,N\} ds\Big].
\end{equation}

\n
To bound the last term of (\ref{4.44}), we note that $\wt{V}$ in (\ref{4.27}) satisfies
\begin{equation}\label{4.46}
|\wt{V}(y)| \le \mbox{\f $\dis\frac{c(u, \varphi)}{N^2}$} \,1\{|y|_\infty \le c_6 (u,\varphi)\,N\}, \; \mbox{for all $y \in \IZ^d$},
\end{equation}

\n
where we choose $c_6(u,\varphi) \ge c_5(u,\varphi)$ in (\ref{4.44}) (and such that $\varphi$ is supported in $B_{\IR^d} (0,c_6)$). Then, setting $B_N = B(0,c_6\,N)$ and $\wh{B}_N = B(0,2 c_6\,N)$, we denote by $H_N$ the entrance time in $B_N$ and $\wh{T}_N$ the exit time from $\wh{B}_N$. We see that for all $y \in \wh{B}_N$
\begin{equation}\label{4.47}
\begin{split}
\wt{P}_y [\wh{T}_N < N^2] & \stackrel{(\ref{1.9})}{=} \mbox{\f $\dis\frac{1}{\wt{f}(y)}$} \;E_y \Big[\wt{f}(X_{N^2}) \;\exp\Big\{\dis\int^{N^2}_0 \wt{V}(X_s)\,ds\Big\}, \; \wh{T}_N < N^2\Big]
\\
&\hspace{-3ex} \stackrel{(\ref{4.18}), (\ref{4.46})}{\ge} c(u,\varphi) \; P_y[\wh{T}_N < N^2] \ge c'(u,\varphi) \in (0,1),
\end{split}
\end{equation}
using the invariance principle in the last step. It then follows that for all $y \in \wh{B}_N$, using the Markov property and (\ref{4.47}),
\begin{equation}\label{4.48}
\wt{E}_y [\wh{T}_N] \le N^2 \, \dsl_{k \ge 0} \;\wt{P}_y [\wh{T}_N > k N^2] \le N^2 \dsl_{k \ge 0} \big(1 - c'(u,\varphi)\big)^k = N^2 \, c'(u,\varphi)^{-1}.
\end{equation}
On the other hand, for all $y \notin \wh{B}_N$, we have
\begin{equation}\label{4.49}
\wt{P}_y [H_N < \infty] = P_y [H_N < \infty] \le c < 1.
\end{equation}

\n
We can then introduce the successive returns $R_n$ to $B_N$ and departures $D_n$ from $\wh{B}_N$, so that for $y \in B_N$, $0 = R_1 \le D_1 \le R_2 \le D_2 \le \dots \le \infty$, $\wt{P}_y$-a.s., and write
\begin{equation}\label{4.50}
\begin{split}
\wt{E}_y \;\Big[\int^\infty_0 \!\!1_{B_N}(X_s) \,ds\Big] & \le \wt{E}_y \big[\dsl_{n \ge 1} (D_n - R_n) \,1\{R_n < \infty\}\big] = \dsl_{n \ge 1} \wt{E}_y \big[R_n < \infty, \wt{E}_{X_{R_n}}[\wh{T}_N]\big]
\\
&\!\!\!\!\! \stackrel{(\ref{4.48})}{\le} \dsl_{n\ge 1} c(u,\varphi) \, N^2 \;\wt{P}_y [R_n < \infty] \stackrel{(\ref{4.49})}{\le} c(u,\varphi)\,N^2 \dsl_{n \ge 1} c^{n-1}
\\
& = \wt{c}\;(u,\varphi) \,N^2 .
\end{split}
\end{equation}

\medskip\n
Thus, coming back to (\ref{4.44}), we find with (\ref{4.50}) that for $N$ and $B$ as in (\ref{4.34}), for all $x \in \IZ^d$ one has
\begin{equation}\label{4.51}
\wt{E}_x\Big[\dis\int^\infty_0 \mbox{\f $\dis\frac{c_0}{L^2}$} \;1_B(X_s) \,ds\Big] \le  \; \log 2 + c'(u,\varphi).
\end{equation}
Hence, $\sup_{x \in \IZ^d} \wt{E}_x [\int^\infty_0 \frac{b}{L^2} \;1_B(X_s) \,ds] \le \frac{b}{c_0}\;(\log 2 + c'(u,\varphi))$, for $0 \le b \le c_0$, and (\ref{4.34}) follows. This completes the proof of Lemma \ref{lem4.5}.
\end{proof}

As explained above Lemma \ref{lem4.5}, this completes the proof of Proposition \ref{prop4.4}.
\end{proof}

We now have all the elements to complete the proof of Theorem \ref{theo4.2}.

\bigskip\n
{\it Proof of Theorem \ref{theo4.2}:} Observe that in the infimum that appears in the first line of (\ref{4.5}) the Dirichlet integral $\int_{\IR^d} |\nabla \varphi|^2 dz$ decreases if we replace $\varphi$ by $\varphi^+ = \max (\varphi,0)$, and $\strokedint_D \theta((\sqrt{u} + \varphi)^2)dz$ increases, when $\varphi$ is replaced by $\varphi^+$. Using regularization by convolution, we thus find that the infimum in (\ref{4.5}) remains the same if we additionally require that $\varphi \in C^\infty_0$ is non-negative. In the notation (\ref{4.7}), this proves the equality of the infimum on the second line of (\ref{4.5}) with $I^D_\nu$.

\medskip
Then, for $M \ge 1$ integer, we have $D_N = (\frac{1}{M} \;D)_{MN}$ for all $N \ge 1$, and $I^D_\nu = M^{d-2} \,I^{\frac{1}{M} \,D}_\nu$ by direct inspection. Hence, if we prove the claim (\ref{4.5}) for $\frac{1}{M}\,D$, the claim (\ref{4.5}) holds for $D$ as well. Without loss of generality, we can thus assume, as in (\ref{4.10}), that $D \subseteq [-1,1]^d$. Adding a small smooth compactly supported function, which is non-negative and positive on $[-2,2]^d$, we see that the infimum defining $I^D_\nu$ is unchanged if we additionally require that $\varphi > 0$ on $[-2,2]^d$. For such a $\varphi$, we know by Proposition \ref{prop4.3} and \ref{prop4.4} that $\liminf_N \;\frac{1}{N^{d-2}} \log \IP[\cA_N] \ge - \frac{1}{2d} \;\int_{\IR^d} |\nabla \varphi |^2 dz$. This proves (\ref{4.5}) and completes the proof of Theorem \ref{theo4.2}. \hfill $\square$

\begin{remark}\label{rem4.6}  \rm 1) When $0 < \nu < \theta(u)$ in place of $\theta(u) < \nu < \theta_\infty$, the proof of Theorem \ref{theo4.2} given in this section can be adapted to provide a lower bound on  $\liminf_N$ $\frac{1}{N^{d-2}} \log \IP[\cA'_N]$, where $\cA'_N$ denotes the deficit deviation event $\{\sum_{x \in D_N} F((L^u_{x + \point})) < \nu\, |D_N|\}$ (now replacing the condition $\strokedint_D \theta((\sqrt{u} + \varphi)^2) \,dz > \nu$ by $\strokedint_D  \theta((\sqrt{u} + \varphi)^2) \,dz < \nu$ in the infimum in (\ref{4.5}), as well the condition $\varphi \ge 0$ in (\ref{4.7}) by $-\sqrt{u} < \varphi \le 0$). We will see in the next section that the proof of the upper bound does not seem to accommodate such a change, see Remark \ref{rem5.5}. 

\medskip\n
2) Note that in Theorem \ref{theo4.2} one can straightforwardly remove the boundedness assumption on $F$. Indeed, for large enough $M > 0$, $F \wedge M$ satisfy the assumptions (\ref{2.1}) and (\ref{4.1}). The corresponding event $\cA^{(M)}_N$ is included in $\cA_N$ and by monotone convergence, as $M \r \infty$, the corresponding function $\theta^{(M)}$ increases to $\theta$, and for each $\varphi \in C^\infty_0 (\IR^d)$, $\strokedint_D \theta^{(M)} ((\sqrt{u} + \varphi)^2)\,dz$ increases to $\strokedint_D \theta((\sqrt{u} + \varphi)^2)\,dz$. The claim readily follows. However it is unclear whether the boundedness assumption can be relaxed in the case of Theorem \ref{theo5.1}, see Remark \ref{5.2} 2).  \hfill $\square$
\end{remark}

\section{Large deviation upper bounds}
\setcounter{equation}{0}

The main object of this section is the proof of Theorem \ref{theo5.1} that states our main asymptotic upper bound on the probability of the excess deviation event $\cA_N$, see (\ref{0.7}) or (\ref{4.6}). At the end of the section, as an application of Theorems \ref{theo4.2} and \ref{theo5.1}, we derive the principal exponential rate of decay of $\IP[\cA_N]$. We also provide an application to the simple random walk in Corollary \ref{cor5.11}. We keep the notation from the previous section.

\begin{theorem}\label{theo5.1} (large deviation upper bound)

\smallskip\n
Assume that the local function $F$ satisfies (\ref{2.1}), (\ref{4.1}), and that the compact set $D \subseteq \IR^d$ is of the form (\ref{0.6}). Then, for any $u > 0$ and $\nu$ in $(\theta (u), \theta_\infty)$, one has
\begin{equation}\label{5.1}
\limsup\limits_N \; \mbox{\f $\dis\frac{1}{N^{d-2}}$} \;\log \IP[\cA_N] \le - J^D_\nu ,
\end{equation}
where $\cA_N$ is as in (\ref{4.6}) (or (\ref{0.7})) and for $a \ge 0$ we set
\begin{equation}\label{5.2}
\begin{split}
J^D_a = \inf\Big\{ & \fd \;\dis\int_{\IR^d} |\nabla \varphi|^2 dz; \varphi \ge 0, \; \mbox{compactly supported, in $H^1(\IR^d)$, and}
\\
& \strokedint_D \theta\big((\sqrt{u} + \varphi)^2\big) \,dz \ge a\Big\}
\end{split}
\end{equation}
(with $\strokedint_D \dots dz$ the normalized integral $\frac{1}{|D|} \int_D \dots dz$).
\end{theorem}

\begin{remark}\label{rem5.2} \rm 1) We will see in Corollary \ref{cor5.9} that $I^D_\nu = J^D_\nu$, so that the upper and lower bounds of Theorems \ref{theo5.1} and \ref{4.2} are matching. 

\medskip\n
2) In the proof of Theorem \ref{theo5.1} the boundedness assumption on $F$ plays an important role, notably in (\ref{5.5}) and the coarse graining performed in Proposition \ref{prop5.3}. \hfill $\square$
\end{remark}

At this stage it is maybe helpful to provide  a short outline of the poof of Theorem \ref{theo5.1}. The proof is based on a coarse-graining procedure. By the same scaling argument as in the proof of Theorem \ref{theo4.2}, we can assume that $D \subseteq [-1,1]^d$. In the first main step, namely Proposition \ref{prop5.3}, we replace the event $\cA_N$ by a not too large collection of events (the coarse-graining) that covers $\cA_N$ up to a ``negligible'' part. This step relies on the super-exponential estimates of Section 3. The events entering this coarse-graining collection that essentially covers $\cA_N$, each impose a specific constraint on the occupation time of random interlacements in a collection of well-separated boxes $B$ of side-length $L$ intersecting $D_N\, (= (ND) \cap \IZ^d$). The next objective is to devise uniform bounds on the probability of such events. In essence, this goes via the application of the exponential Chebyshev inequality, see (\ref{5.50}) in the proof of Proposition \ref{prop5.6}. A crucial feature for this purpose is the introduction for each event of the coarse-graining of a well-adapted potential, see (\ref{5.25}). When tested against the field of occupation times, this potential behaves well, namely, one has a natural lower bound, see (\ref{5.26}), on the event of the coarse-graining to which it is attached. In addition, one has an efficient upper bound of the corresponding exponential moment, see (\ref{5.51}), (\ref{5.52}). The key estimate for this last feature is developed in Proposition \ref{prop5.4}. The bound that we derive on each event entering the coarse-graining of $\cA_N$  has a similar flavor to Section 3 of \cite{Szni17} about ``occupation-time bounds'', but the situation here is definitely more delicate. In Proposition \ref{prop5.7}, we then investigate the scaling limit of the discrete variational problems entering the asymptotic upper bounds on $\IP[\cA_N]$ obtained in Proposition \ref{prop5.6}, somewhat in the spirit of $\Gamma$-convergence (see \cite{Dalm93}). Some of the controls from Section 5 of \cite{LiSzni15} are very helpful at this stage. Then, the proof of Theorem \ref{theo5.1} can relatively quickly be completed.

\medskip
We now set this program in motion. We assume that $D$ in addition to (\ref{0.6}) also satisfies
\begin{equation}\label{5.3}
D \subseteq [-1,1]^d.
\end{equation}

\n
By the same scaling argument as in the proof of Theorem \ref{theo4.2} at the end of Section 4 (namely, for $M, N \ge 1$ integers $(\frac{1}{M}D)_{MN} = D_N$, and $M^{d-2} J^{\frac{1}{M}D}_\nu = J^D_\nu$), this causes no loss of generality in the proof that follows. Next, we wish to introduce a (not too large) collection of events that covers $\cA_N$ up to a negligible part (the ``coarse-graining'' of $\cA_N$). We first need some notation. We consider $\ve \in (0,1)$ (this parameter is unrelated to the choice in (\ref{4.12}), but in spirit plays a similar role), such that with $u > 0$ and $\nu > 0$ as in Theorem \ref{theo5.1},
\begin{equation}\label{5.4}
\theta(u) + 10 \ve < \nu.
\end{equation}

\n
We then consider (see (\ref{4.3})) a non-empty finite subset $\Sigma(\ve, \theta) \subseteq (0,\infty)$ such that
\begin{equation}\label{5.5}
\begin{array}{l}
\mbox{the non-decreasing function $\theta$ varies at most by $\frac{\ve}{10}$ between consecutive} 
\\
\mbox{points of $\{0,\infty\} \cup \Sigma$}.
\end{array}
\end{equation}
We then choose $\kappa(\ve,\theta)$ such that
\begin{equation}\label{5.6}
\mbox{$0 < \kappa < \mbox{\f $\dis\frac{\ve}{2}$}$ and for each $\alpha \in \Sigma, \big((1 - \kappa) \,\alpha, (1 + \kappa)\,\alpha\big) \cap \Sigma = \{\alpha\}$}
\end{equation}
(so, (\ref{2.78}) holds), and we pick $\mu$ via (so (\ref{3.1}) holds due to (\ref{4.1}))
\begin{equation}\label{5.7}
\mu = 0.
\end{equation}
The above choices now determine a constant (see Proposition \ref{prop3.1})
\begin{equation}\label{5.8}
c_7(F,\ve) = c_3(\Sigma,\kappa,\mu,F) (\ge 100),
\end{equation}
and we assume that (this will not be needed in Propositions \ref{prop5.4}, \ref{prop5.7} and (\ref{5.48}))
\begin{equation}\label{5.9}
K \ge c_7(F,\ve).
\end{equation}

\n
We keep the choice (\ref{4.8}) for $L$ as a function of $N$, and introduce the collections (recall that $\IL = L \, \IZ^d$, see (\ref{1.18})):
\begin{equation}\label{5.10}
\begin{split}
\cC & = \{z \in \IL; B_z \cap D_N \not= \phi\}, \;\mbox{and}
\\ 
\cC_\tau & = \{z \in \cC; z \in L \tau + \ov{K} \,\IL\}, \;\mbox{for $\tau \in \{0,\dots,\ov{K} - 1\}^d$ where $\ov{K} \stackrel{(\ref{2.12})}{=} 2 K + 3$}
\end{split}
\end{equation}
($\cC_\tau$ should not be confused with $\cC^\tau$ from (\ref{3.12})).

\medskip
For each $\tau \in \{0,\dots,\ov{K}-1\}^d$, we further define
\begin{equation}\label{5.11}
\begin{split}
\cF_\tau = \Big\{\ov{f} = (f_B)_{B \in \cC_\tau}; &\; f_B \ge 0, f^2_B \in \{0\} \cup \Sigma \;\mbox{for each $B \in \cC_\tau$, and}
\\ 
&\dsl_{B \in \cC_\tau} \theta(f^2_B) |B| \ge \mbox{\f $\dis\frac{\nu - \ve}{\ov{K}\,^d}$} \;|D_N|\Big\},
\end{split}
\end{equation}
as well as for each $\ov{f} \in \cF_\tau$ the event
\begin{equation}\label{5.12}
\cA_{\ov{f}} = \bigcap\limits_{B \in \cC_\tau} \{ \langle \ov{e}_B, L^u\rangle \ge (1- \ve) \, f^2_B\}
\end{equation}

\n
(by convention when $\cC_\tau$ is empty, $\cF_\tau$ is defined as a singleton, and for the unique $\ov{f} \in \cF_\tau$ the event $\cA_{\ov{f}}$ is the entire sample space of random interlacements). Of course, by (\ref{0.6}) and (\ref{4.8}),  for large $N$, all $\cC_\tau$, $\tau \in \{0,\dots\ov{K} - 1\}^d$ are non-empty. As we will now see, the collections $\cF_\tau$ are not too large, the union of events $\cA_{\ov{f}}$, as $\ov{f}$ runs over $\cF_\tau$ and $\tau$ over $\{0,\dots,\ov{K} = 1\}^d$ essentially covers $\cA_N$, and the task of bounding $\IP[\cA_N]$ is essentially reduced to bounding $\IP[\cA_{\ov{f}}]$ in a uniform fashion. We recall the definition of the (bad) event $\cB$ in (\ref{3.10}), where our current choice of $\Sigma, \kappa,\mu$, see (\ref{5.5}) - (\ref{5.7}), is now in force.

\begin{proposition}\label{prop5.3} (coarse-graining)

\smallskip\n
As $N$ goes to infinity,
\begin{equation}\label{5.13}
\mbox{for each $\tau \in \{0,\dots,\ov{K} - 1\}^d,\;|\cF_\tau | = \exp \{o(N^{d-2})\}$}.
\end{equation}
Moreover, for large $N$, with $\cB$ mentioned above, one has
\begin{equation}\label{5.14}
\cA_N \backslash \cB \subseteq \bigcup\limits_{\tau \in \{0,\dots,\ov{K} - 1\}^d} \;\bigcup\limits_{\ov{f} \in \cF_\tau} \; \cA_{\ov{f}},
\end{equation}
and
\begin{equation}\label{5.15}
\limsup\limits_N \; \mbox{\f $\dis\frac{1}{N^{d-2}}$} \;\log \IP[\cA_N] \le \limsup\limits_N \;\sup\limits_{\tau \in \{0,\dots,\ov{K} - 1\}^d} \;\; \sup\limits_{\ov{f} \in \cF_\tau} \; \mbox{\f $\dis\frac{1}{N^{d-2}}$} \; \log \IP[\cA_{\ov{f}}].
\end{equation}
\end{proposition}

\begin{proof}
We first prove (\ref{5.13}). To this end we note that for all $\tau \in \{0,\dots,\ov{K}-1\}^d$
\begin{equation}\label{5.16}
|\cF_\tau | \le (1 + |\Sigma |)^{|\cC|} \stackrel{(\ref{5.10})}{\le} (1 + |\Sigma |)^{c \,(\frac{N}{L})^d} \stackrel{(\ref{4.8})}{\le} (1 + | \Sigma |)^{c\,\frac{N^{d-2}}{\log N}}= \exp\{o(N^{d-2})\}, \;\mbox{as $N \r \infty$}.
\end{equation}
This proves (\ref{5.13}).

\medskip
We now turn to the proof of (\ref{5.14}). We note that for large $N$ on $\cA_N \backslash \cB$ the number of $(\Sigma,\kappa,\mu)$-bad boxes $B$ in $\cC$ is at most the number of $(\Sigma,\kappa,\mu)$-bad boxes meeting $[-N_L,N_L]^d$ and hence at most $\rho_L (\frac{N_L}{L})^d$. The number of sites belonging to such bad boxes is at most $\rho_L \,N_L^d = o(N^d)$, by (\ref{3.7}), (\ref{4.9}). For a $(\Sigma,\kappa,\mu)$-good box $B$ in $\cC$, we denote by $f^2_B$ (with $f_B \ge 0$) the largest element in $\{0\} \cup \Sigma$ smaller or equal to $N_u(B) / {\rm cap}(B)$.

\medskip
We observe that either $f^2_B = \max \Sigma$, and by (\ref{5.5}), $\theta_\infty \le \theta(f^2_B) + \frac{\ve}{10}$, so that (recall that $\theta_\infty = \|F\|_\infty$, see (\ref{4.4})):
\begin{equation}\label{5.17}
\sum\limits_{x \in B} \; F\big((L^u_{x + \point})\big) \le \Big(\theta(f^2_B) + \mbox{\f $\dis\frac{\ve}{10}$}\Big) \;|B|,
\end{equation}
or $f^2_B < \max \Sigma$, and denoting by $\alpha$ the smallest element in $\Sigma$ bigger than $f^2_B \in \{0\} \cup \Sigma$, we have $N_u(B) / {\rm cap}(B) < \alpha$. Since $B$ is a $(\Sigma,\kappa,\mu)$-good box, it follows by (\ref{2.77}) that $\Sigma_{x \in B_R} F((L^u_{x + \point})) \le (\theta((1 + \kappa)\alpha) + \mu) \,|B|$, and taking (\ref{5.5})~--~(\ref{5.6}) into account, as well as $\mu = 0$ by (\ref{5.7}), we find that
\begin{equation}\label{5.18}
\sum\limits_{x \in B} F\big((L^u_{x + \point})\big) \le \Big(\theta(f^2_B) + \mbox{\f $\dis\frac{2}{10}$} \;\ve\Big) \,|B| + \| F \|_\infty \, |B \backslash B_R|.
\end{equation}
Using once more that $B$ is a $(\Sigma,\kappa,\mu)$-good box for which $N_u(B) / {\rm cap}(B) \ge f^2_B$, we also find from (\ref{2.77}) that

\vspace{-5ex}
\begin{equation}\label{5.19}
\langle \ov{e}_B, L^u\rangle \ge f^2_B (1- \kappa) \stackrel{(\ref{5.6})}{\ge} f^2_B (1 - \ve).
\end{equation}
When $B$ is a $(\Sigma,\kappa,\mu)$-bad box, we simply set $f_B = 0$.

\medskip
When $N$ is large $\|F\|_\infty \, |B \backslash B_R | \le \frac{\ve}{10} \;|B|$, and we see that for large $N$, on $\cA_N \backslash \cB$,
\begin{equation}\label{5.20}
\begin{split}
\nu\,|D_N| & \stackrel{(\ref{4.6})}{\le} \sum\limits_{x \in D_N}F\big((L^u_{x + \point})\big) \stackrel{(\ref{5.10})}{\le} \dsl_{B \in \cC} \; \dsl_{x \in B} F \big((L^u_{x + \point})\big)
\\
&\; \;\le \dsl_{{\rm good} \,B\,{\rm in} \, \cC} \Big(\theta(f^2_B) + \mbox{\f $\dis\frac{3}{10}$} \;\ve\Big) \,|B| + \theta_\infty \,\rho_L \,N_L^d \le \dsl_{B \in \cC} \theta(f^2_B) \,|B| + \ve \,|D_N|.
\end{split}
\end{equation}
In addition, we have
\begin{equation}\label{5.21}
\mbox{for all} \;B \in \cC, \,\langle \ov{e}_B, L^u \rangle \ge f^2_B (1-\ve).
\end{equation}
Taking the definition (\ref{5.11}) of $\cF_\tau$ into account, we see by (\ref{5.20}), (\ref{5.21}) that for large $N$
\begin{equation}\label{5.22}
\cA_N \backslash \cB \subseteq \bigcup\limits_{\tau \in \{0,\dots,\ov{K}-1\}^d} \; \bigcup\limits_{f \in \cF_\tau} \;\bigcap\limits_{B \in \cC_\tau} \{ \langle \ov{e}_B,L^u\rangle \ge f^2_B (1- \ve)\},
\end{equation}
and this proves (\ref{5.14}). The last claim (\ref{5.15}) is now an immediate consequence of the super-exponential decay in Proposition \ref{prop3.1}, a union bound, and (\ref{5.13}). This concludes the proof of Proposition \ref{prop5.3}.
\end{proof}

{\begin{samepage} Our next task is to develop efficient exponential Chebyshev bounds on the quantities $\IP[\cA_{\ov{f}}]$, where $\ov{f}$ runs over $\bigcup_\tau \cF_\tau$, as in the right member of (\ref{5.15}). To this end, for $\tau \in \{0,\dots,\ov{K} -1\}^d$ and $\ov{f} \in \cF_\tau$, we set
\begin{equation}\label{5.23}
\begin{array}{ll}
f = \sqrt{u} + \wh{f}, &\mbox{where $\wh{f}$ is the smallest non-negative superharmonic function}
\\
&\mbox{on $\IZ^d$ such that $\wh{f} \ge (f_B - \sqrt{u})_+$ on each $B \in \cC_\tau$}
\end{array}
\end{equation}
(see for instance Proposition 6-2-8, p.~130 of \cite{Neve75}).
\end{samepage}}

\medskip
We note that $\wh{f}$ tends to $0$ at infinity (actually, with $G$ as in (\ref{1.3}), one has $\wh{f} = G(-L \wh{f}) \le (\max_{B \in \cC_\tau} f_B) \,G(e_{\bigcup_{\cC_\tau} B}))$, that $\wh{f}$ is either positive or identically equal to $0$ (as a non-negative superharmonic function), and in addition (see \cite{Neve75}, p.~130),
\begin{equation}\label{5.24}
\left\{ \begin{array}{rl}
{\rm i)} & -L f = - L \wh{f} \ge 0 \;\; \mbox{on} \;\; \IZ^d
\\[1ex]
{\rm ii)} & \{- L  f > 0\} \subseteq \bigcup\limits_{B \in \cC_\tau, f_B > \sqrt{u}} \;\{f = f_B\} \cap B.
\end{array}\right.
\end{equation}
To produce for each $\ov{f}$ the desired exponential Chebyshev bound on $\IP[\cA_{\ov{f}}]$, we will not directly work with the potential $V = - \frac{Lf}{f}$ because the (random) quantity $\inf_{\cA_{\ov{f}}} \langle L^u,V \rangle$ is delicate to control. In place, we will use the coarse-grained potential (with $\ov{f} \in \cF_\tau$):
\begin{equation}\label{5.25}
\begin{array}{l}
W(\cdot) = \sum\limits_{B \in \cC_\tau} \; \mbox{\f $\dis\frac{-L \wh{f} (B)}{f(y_B)}$} \;\ov{e}_B(\cdot) (\ge 0), \;\mbox{where $L\,\wh{f}(B) = \sum_{x \in B} L \wh{f}(x)$,}
\\[2ex]
\mbox{and we have set}
\\[1ex]
\mbox{$y_B = x_B (= z)$,  if $\{f = f_B\} \cap B =\phi$ (with $B = B_z)$},
\\ 
\quad  \;\,\mbox{$=$ some chosen point in $\{f = f_B\} \cap B$ otherwise}.
\end{array}
\end{equation}
The above definition of $W$ shares some similar flavor to (\ref{3.4}) of \cite{Szni17}, see also Proposition 4.6 and (4.63) of \cite{ChiaNitz19}. But the situation here is more intricate. To illustrate why the choice (\ref{5.25}) is pertinent, we note that for $\tau \in \{0,\dots,\ov{K}-1\}^d$, $\ov{f} \in \cF_\tau$, and $W$ as above, we have on $\cA_{\ov{f}}$\,:

\vspace{-5ex}
\begin{equation}\label{5.26}
\begin{split}
\langle L^u, W\rangle & = \dsl_{B \in \cC_\tau} \; \mbox{\f $\dis\frac{-L\,\wh{f}(B)}{f(y_B)}$} \;\langle L^u,\ov{e}_B\rangle   \stackrel{(\ref{5.12})}{\ge} \dsl_{B \in \cC_\tau}\; \mbox{\f $\dis\frac{-L\,\wh{f}(B)}{f(y_B)}$} \;f^2_B (1- \ve)
\\[1ex]
&\hspace{-4ex}\! \stackrel{(\ref{5.24}) \,{\rm ii)}, (\ref{5.25})}{=} (1-\ve)  \dsl_{B \in \cC_\tau} - L\,\wh{f}(B)\,f_B \stackrel{(\ref{5.24})\, {\rm ii)}}{=} (1-\ve) \,\dsl_{x \in \IZ^d} - L\,\wh{f}(x) \,f(x).
\end{split}
\end{equation}
We now wish to obtain an upper bound on the exponential moment $\IE[e^{a \langle L^u,W\rangle}]$ when the spacing between the $B$-boxes as measured by $K$ is large (and $a \in (0,1)$ will eventually tend to $1$). The next proposition contains the main controls to this effect.

\begin{proposition}\label{prop5.4}
Consider $a \in (0,1)$. Then, for $K \ge c_8(a) (\ge 100)$, $N \ge 1$, $\tau \in \{0,\dots,\ov{K} - 1\}^d$ and $\ov{f} \in \cF_\tau$, with $f$ as in (\ref{5.23})
\begin{equation}\label{5.27}
\mbox{for each $B \in \cC_\tau$, $f \le f(y_B) \Big(1 + \mbox{\f $\dis\frac{c}{K}$}\Big)$ on $B$},
\end{equation}
and
\begin{equation}\label{5.28}
\sqrt{u} + G(a\,W f) \le f.
\end{equation}
\end{proposition}
\begin{proof}
{ \begin{samepage} We begin by the proof of (\ref{5.27}). It will be the immediate consequence of a similar inequality with $\wh{f}$ in place of $f$ (recall that $f = \sqrt{u} + \wh{f})$, see (\ref{5.38}), (\ref{5.39}) below. With this goal in mind, we consider $N \ge 1$, and $\tau \in \{0,\dots,\ov{K} - 1\}^d$. We can assume that $\cC_\tau$ is not empty, otherwise there is nothing to prove (and $W = 0$, $f = \sqrt{u}$ so that incidentally (\ref{5.28}) holds as well). We thus consider an $\ov{f} \in \cF_\tau$ with associated $\wh{f}$, and some given box $B^0 \in \cC_\tau$. We introduce the following decomposition of $\wh{f}$ in (\ref{5.23}):
\begin{equation}\label{5.29}
\wh{f} = \wh{f}_1 + \wh{f}_0 \; \; \mbox{where} \;\; \wh{f}_1 = \dsl_{B \not= B^0} G\big(( -L\,\wh{f}) 1_B\big), \;\mbox{and} \;\; \wh{f}_0 = G\big(( -L\,\wh{f}) 1_{B^0}\big)
\end{equation}
(in the first sum $B$ runs over all $B$ in $\cC_\tau$ different from $B^0$).
\end{samepage}}

\medskip
Thus, $\wh{f}_1$ and $\wh{f}_0$ are non-negative superharmonic functions on $\IZ^d$ and
\begin{equation}\label{5.30}
\mbox{$\wh{f}_1$ is non-negative harmonic in $\IZ^d\, \backslash \, \big(\bigcup\limits_{B \not= B^0} B\big)$}
\end{equation}
(and the notation is similar to what is explained below (\ref{5.29})).

\medskip
This last set contains a $(2 K + 2) \,L$ neighborhood in sup-distance of $B^0$. Hence, when $K \ge c$, by a gradient estimate, see Theorem 1.7.1, p.~42 of \cite{Lawl91}, and Harnack inequality, see Theorem 1.7.2, p.~42 of the same reference,
\begin{equation}\label{5.31}
|\wh{f}_1(x) - \wh{f}_1(y_{B^0})| \le \mbox{\f $\dis\frac{c}{KL}$} \;|x - y_{B^0}|_\infty \;\wh{f}_1(y_{B^0}), \; \mbox{for all $x \in B^0$}.
\end{equation}

\n
We will now bound $\wh{f}_0$ from above, see (\ref{5.32}), and record a lower bound on $\wh{f}_1$, see (\ref{5.33}). To this end, we first observe that either $\wh{f}_0 = 0$, or $\wh{f}_0 = G((-L \wh{f})1_{B^0}) > 0$, and by \hbox{(\ref{5.24})\, ii)} $\{(-L \wh{f})1_{B^0} > 0\} \subseteq \{\wh{f} = (f_{B^0} - \sqrt{u})_+\} \cap \{\wh{f} > 0\} \subseteq \{\wh{f} \le f_{B^0} - \sqrt{u}\}$. Note that $\wh{f}_0$ is the potential of the finitely supported function $(- L \wh{f})1_{B^0}$, and its maximal value, if it does not identically vanish, occurs on the set $\{(- L \wh{f})1_{B^0} > 0\}$, where it is at most $f_{B^0} - \sqrt{u}$, so that taking into account the definition of $y_{B^0}$ in (\ref{5.25}), we find
\begin{equation}\label{5.32}
0 \le \wh{f}_0 \le \wh{f}(y_{B^0}) \; \; \mbox{on} \;\; B^0.
\end{equation}
In addition, by (\ref{5.31}), we see that
\begin{equation}\label{5.33}
\wh{f}_1 \ge  \wh{f}_1(y_{B^0}) \; \Big(1 - \mbox{\f $\dis\frac{c}{K}$}\Big) \; \; \mbox{on} \;\; B^0.
\end{equation}
(and we assume $K> c$ from now on, so the last factor is positive).

\medskip
We then introduce the two (finite) subsets of $\IZ^d$
\begin{equation}\label{5.34}
C = \bigcup\limits_{B \in \cC_\tau} B \;\; \mbox{and} \;\; C^0 = \bigcup\limits_{B \not= B^0, B \in \cC_\tau} B.
\end{equation}

\n
Since $\wh{f}_0$ is non-negative harmonic outside $B^0$, tends to $0$ at infinity, and is bounded by $\wh{f}(y_{B^0})$ on $B^0$ by (\ref{5.32}), one has
\begin{equation}\label{5.35}
\wh{f}_0 \le  \mbox{\f $\dis\frac{c'}{K^{d-2}}$}\;\wh{f}(y_{B^0}) \;\;\mbox{on} \;\;C^0.
\end{equation}

\n
We will now bound $\wh{f}$ from above (and this will lead to the desired (\ref{5.38}), (\ref{5.39})). We let $h_{B^0}$ and $h_{C^0}$ denote the respective equilibrium potentials of $B^0$ and $C^0$ (i.e.~$h_{B^0} = G \,e_{B^0}$ and $h_{C^0} = G \,e_{C^0}$). We introduce the function
\begin{equation}\label{5.36}
\begin{array}{l}
\wh{f}_2 = \wh{f}_1   + \mbox{\f $\dis\frac{c'}{K^{d-2}}$}\;\wh{f}(y_{B^0}) \,h_{C^0} + \Big(\wh{f}(y_{B^0}) -  \wh{f}_1(y_{B^0}) \Big(1 - \mbox{\f $\dis\frac{c}{K}$}\Big)\Big) \;h_{B^0} 
\\
\mbox{(with $c,c'$ as in (\ref{5.33}), (\ref{5.35}))}.
\end{array}
\end{equation}

\n
Thus, $\wh{f}_2$ is a non-negative superharmonic function. On $C^0$, by (\ref{5.35}), it is bigger or equal to $\wh{f}_1 + \wh{f}_0 = \wh{f}$, and on $B^0$ it is bigger or equal to $\wh{f}_1 + \wh{f}(y_{B^0}) - \wh{f}_1(y_{B^0})  (1 -  \frac{c}{K})   \stackrel{(\ref{5.33})}{\ge} \wh{f}(y_{B^0}) \ge (f_{B^0} - \sqrt{u})_+$. By the minimality of $\wh{f}$ (see (\ref{5.23})), we thus find that
\begin{equation}\label{5.37}
\wh{f}  \le \wh{f}_2 \;\;\mbox{on} \;\; \IZ^d.
\end{equation}
As a result we find that on $B^0$
\begin{equation}\label{5.38}
\begin{split}
\wh{f}  \le \wh{f}_2 &\le \wh{f}_1 + \mbox{\f $\dis\frac{c'}{K^{d-2}}$}\; \wh{f}(y_{B^0})  + \wh{f}(y_{B^0}) - \wh{f}_1 (y_{B^0}) \, \Big(1 - \mbox{\f $\dis\frac{c}{K}$}\Big)
\\
&\!\!\!\! \stackrel{(\ref{5.31})}{\le} \wh{f}_1 (y_{B^0}) \Big(1 + \mbox{\f $\dis\frac{c}{K}$}\Big) +  \mbox{\f $\dis\frac{c'}{K^{d-2}}$}\;\wh{f}(y_{B^0}) + \wh{f}(y_{B^0}) - \wh{f}_1(y_{B^0}) + \mbox{\f $\dis\frac{c}{K}$} \;\wh{f}_1(y_{B^0})
\\
&\le \wh{f}(y_{B^0}) \Big(1 + \mbox{\f $\dis\frac{c}{K}$} \Big).
\end{split}
\end{equation}
This now implies that for any $B$ in $\cC_\tau$ (playing the role of $B^0$)
\begin{equation}\label{5.39}
f = \sqrt{u} + \wh{f} \le \sqrt{u} + \wh{f} (y_B) \Big(1 + \mbox{\f $\dis\frac{c}{K}$} \Big) \le f(y_B) \Big(1 + \mbox{\f $\dis\frac{c}{K}$} \Big).
\end{equation}
This proves (\ref{5.27}). We now proceed with the proof of (\ref{5.28}), and note that
\begin{equation}\label{5.40}
\begin{split}
G(a\,W\,f) & \stackrel{(\ref{5.25})}{=} \dsl_{B \in \cC_\tau}\;  \mbox{\f $\dis\frac{-L\,\wh{f}(B)}{f(y_B)}$} \; a\,G(f \,\ov{e}_B)
\\
 & \stackrel{(\ref{5.39})}{\le}  \dsl_{B \in \cC_\tau}  - L\,\wh{f} (B) \, a\Big(1 + \mbox{\f $\dis\frac{c}{K}$} \Big) \,G(\ov{e}_B) = \dsl_{B \in \cC_\tau} \;  \mbox{\f $\dis\frac{-L\,\wh{f}(B)}{{\rm cap} (B)}$}\; a  \Big(1 + \mbox{\f $\dis\frac{c}{K}$} \Big) \,h_B
\end{split}
\end{equation}
(with $h_B = G \, e_B$ the equilibrium potential of $B$).

\medskip
We then consider some $B^0 \in \cC_\tau$ and derive upper bounds on $-L\,\wh{f}(B^0)$. For this purpose we use the fact that $\wh{f}$ and $\wh{f}_1$ do not move too much on $B^0$ (in the notation of (\ref{5.29})). Namely, either $\wh{f}_0 = 0$ identically, in which case $-L\,\wh{f}(B^0) = 0$, or $\wh{f}_0 \not= 0$ and $\wh{f}$ takes its minimal value on $B^0$ at $y_{B^0}$ (see (\ref{5.25}), (\ref{5.24})), so that on $B^0$
\begin{equation}\label{5.41}
\begin{array}{l}
\wh{f} (y_{B^0}) \le \wh{f} \stackrel{(\ref{5.38})}{\le} \wh{f} (y_{B^0})\Big(1 + \mbox{\f $\dis\frac{c}{K}$} \Big), \;\mbox{and by (\ref{5.31})}
\\[1ex]
\wh{f}_1 (y_{B^0})\Big(1 - \mbox{\f $\dis\frac{c}{K}$} \Big) \le \wh{f}_1 \le \wh{f}_1(y_{B^0}) \Big(1 + \mbox{\f $\dis\frac{c}{K}$} \Big).
\end{array}
\end{equation}
It now follows that on $B^0$ (actually, we will only use the second inequality):
\begin{equation}\label{5.42}
\wh{f}_0(y_{B^0}) - \mbox{\f $\dis\frac{c}{K}$} \;\wh{f}_1 (y_{B^0}) \le \wh{f}_0 = \wh{f} - \wh{f}_1 \le \wh{f}_0(y_{B^0}) + \mbox{\f $\dis\frac{c}{K}$} \;\wh{f} (y_{B^0}).
\end{equation}

\n
Recall that $\wh{f}_0 = G((-L\,\wh{f}) 1_{B^0})$, and note that $-L\,\wh{f} (B^0) =  \langle h_{B^0}, (-L\,\wh{f}) 1_{B^0}\rangle = \langle G\,e_{B^0}$, $(-L\,\wh{f})1_{B^0} \rangle = \langle e_{B^0},\wh{f}_0\rangle$, so that by (\ref{5.42})
\begin{equation}\label{5.43}
\left\{ \begin{array}{rl}
{\rm i)} & -L \wh{f} (B^0)= \langle e_{B^0}, \wh{f}_0\rangle \le \big(\wh{f}_0(y_{B^0}) +  \mbox{\f $\dis\frac{c}{K}$} \; \wh{f}(y_{B^0})\big) \;{\rm cap}(B^0),
\\[1ex]
{\rm ii)} & -L \wh{f} (B^0)= \langle e_{B^0}, \wh{f}_0\rangle \ge \big(\wh{f}_0(y_{B^0}) -  \mbox{\f $\dis\frac{c}{K}$} \; \wh{f}_1(y_{B^0})\big) \;{\rm cap}(B^0)
\end{array}\right.
\end{equation}
(we will only need i) in what follows).

\medskip
Note that for $B \not= B^0$ we have $h_B/ {\rm cap}(B) = G(\ov{e}_B)$. For $x,x'$ in $B$ and $x_0$ in $B^0$ we have $g(x,x_0) \le (1 + \frac{c}{K}) \;g(x',x_0)$ by a similar argument as in (\ref{5.31}) (with now $B$ in place of $B^0$ and $x'$ in place of $y_{B^0}$). We now find that on $B^0$ (keeping in mind (\ref{5.40}))
\begin{equation}\label{5.44}
\dsl_{B \not= B^0} \; \mbox{\f $\dis\frac{L\,\wh{f}(B)}{{\rm cap}(B)}$}\; a\Big(1 + \mbox{\f $\dis\frac{c}{K}$} \Big) \; h_B \le \dsl_{B \not= B^0}\; a\Big(1 + \mbox{\f $\dis\frac{c'}{K}$} \Big)\;G\big((-L\wh{f}) 1_{B}\big) = a\Big(1 + \mbox{\f $\dis\frac{c'}{K}$} \Big)\;\wh{f}_1
\end{equation}
(and the summation refers to $B \not= B^0$ in $\cC_\tau$).

\medskip
Thus, on $B^0$, if $-L\,\wh{f}(B^0) \not= 0$, we have by (\ref{5.40}), (\ref{5.44}), and (\ref{5.43}) i)
\begin{equation}\label{5.45}
\begin{split}
G(a\,W\,f) & \le a\Big(1 + \mbox{\f $\dis\frac{c'}{K}$} \Big)\wh{f}_1 + a \Big(1 + \mbox{\f $\dis\frac{c}{K}$} \Big)  \;\Big(\wh{f}_0 (y_{B^0}) + \mbox{\f $\dis\frac{c}{K}$}  \; \wh{f} (y_{B^0})\Big)
\\[1ex]
&\!\!\!\!\stackrel{(\ref{5.31})}{\le} a\Big(1 + \mbox{\f $\dis\frac{c{''}}{K}$} \Big)\;\wh{f}_1(y_{B^0}) + a\big(\wh{f}_0(y_{B^0}) + \mbox{\f $\dis\frac{c{'''}}{K}$} \;\wh{f} (y_{B^0})\big) \le a \Big(1 + \mbox{\f $\dis\frac{c}{K}$}\Big)  \; \wh{f} (y_{B^0})
\\[1ex]
& \le a\Big(1 + \mbox{\f $\dis\frac{c}{K}$} \Big)\; \wh{f} 
\end{split}
\end{equation}
(recall that when $-L \wh{f}(B^0) \not= 0$, $\wh{f}(y_{B^0}) = (f_{B^0} - \sqrt{u})_+ \le \wh{f}$ on $B^0$, see (\ref{5.25}), (\ref{5.24})). On the other hand, if $-L \wh{f}(B^0) = 0$, then on $B^0$
\begin{equation}\label{5.46}
G(a\, W\, f) \stackrel{(\ref{5.40}),(\ref{5.44})}{\le} a\Big(1 + \mbox{\f $\dis\frac{c'}{K}$} \Big)\;\wh{f}_1 = a\Big(1 + \mbox{\f $\dis\frac{c'}{K}$} \Big)\;\wh{f} .
\end{equation}
Note that $B^0 \in \cC_\tau$ is arbitrary. Thus, taking (\ref{5.45}) and (\ref{5.46}) into account, we see that
\begin{equation}\label{5.47}
\mbox{when $a(1 + \frac{\wt{c}}{K}) \le 1$, then $\sqrt{u} + G(a W f) \le \sqrt{u} + \wh{f} = f$}.
\end{equation}
This proves (\ref{5.28}) and completes the proof of Proposition \ref{prop5.4}.
\end{proof}

The above proposition contains the crucial control to bound the exponential moment of $a \langle L^u,W\rangle$ in Proposition \ref{prop5.6} below. Combined with the lower bound of $\langle L^u,W\rangle$ on  the event $\cA_{\ov{f}}$ derived in (\ref{5.26}), it will produce the efficient upper bounds on $\IP[\cA_{\ov{f}}]$ that we are looking for, see (\ref{5.54}).

\begin{remark}\label{rem5.5} \rm
If instead of the excess deviation event $\cA_N$ from (\ref{4.6}), one considers the deficit deviation event $\cA'_N = \{\sum_{x \in D_N} F((L^u_{x + \point})) < \nu \,|D_N|\}$, where $0 < \nu < \theta(u)$, as mentioned in Remark 4.6, the derivation of an asymptotic lower bound $\liminf_N \frac{1}{N^{d-2}} \;\log \IP[\cA'_N]$ can be achieved with similar methods as in Section 4. However, the derivation of an asymptotic upper bound runs into trouble at the stage we are now. In essence, without entering into details, it requires efficient controls on the negative exponential moments of $\langle L^u,W'\rangle$, where $W'$ is a coarse-grained potential attached to an event of the form $\cA'_{\underline{f}} = \bigcap_{B\in \cC_\tau} \{\langle \ov{e}_B, L^u \rangle \le (1 + \ve) \,f^2_B\}$ (in place of (\ref{5.12})), and $W'$ has a similar form as in (\ref{5.25}), except that now $f = \sqrt{u} - \wh{f}$, and $\wh{f}$ is the smallest non-negative superharmonic function on $\IZ^d$ such that $\wh{f} \ge (\sqrt{u} - f_B)_+$ for each $B$ in $\cC_\tau$ (and the $f_B$ remain uniformly positive). The steps corresponding to (\ref{5.51}), (\ref{5.52}) below are lacking. \hfill $\square$
\end{remark}

For $\ve$ as in (\ref{5.4}) and $K \ge 100$, we define (see (\ref{5.11}), (\ref{5.23}) for notation)
\begin{equation}\label{5.48}
I_{\ve,K} = \liminf\limits_N \; \inf\limits_{\tau \in \{0,\dots,\ov{K} - 1\}^d} \;\;\inf\limits_{\ov{f} \in \cF_\tau} \;\;\mbox{\f $\dis\frac{1}{N^{d-2}}$} \;\cE(f- \sqrt{u}, \;f - \sqrt{u})
\end{equation}
(where $\cE(\cdot,\cdot)$ stands for the Dirichlet form, see (\ref{1.2})).

\begin{proposition}\label{prop5.6} (the exponential Chebyshev bound)

\smallskip\n
For $\ve$ as in (\ref{5.4}), $0 < a  < 1$ and $K \ge c_9(a,F,\ve) \ge c_7 (F,\ve) \vee c_8(a)$ (see (\ref{5.8}) and Proposition \ref{prop5.4}), one has
\begin{equation}\label{5.49}
\limsup\limits_N \; \mbox{\f $\dis\frac{1}{N^{d-2}}$} \;\log \IP[\cA_N] \le -a \big(1 - \ve (1 + \sqrt{u})\big) \,I_{\ve,K} + c \,\sqrt{u} \ve.
\end{equation}
\end{proposition}

\begin{proof}
By Proposition \ref{prop5.3} it suffices to bound the right member of (\ref{5.15}). For any $\tau \in \{0,\dots,\ov{K} - 1\}$, $\ov{f} \in \cF_\tau$, $a \in (0,1)$, and $W$ as in (\ref{5.25}), the lower bound of $\langle L^u,W \rangle$ on $\cA_{\ov{f}}$ in (\ref{5.26}) and the exponential Chebyshev bound yield
\begin{equation}\label{5.50}
 \IP[\cA_{\ov{f}}]  \le \exp\{-a(1-\ve) \,\langle - L \wh{f},f \rangle\} \;\IE[\exp\{a \langle L^u, W\rangle\}].
 \end{equation}
On the other hand by (\ref{5.28}) of Proposition \ref{prop5.4} and induction, we find that
\begin{equation}\label{5.51}
f \ge \sqrt{u} + G(a W f) \ge \sqrt{u} + G(a  W) \,\sqrt{u} + (G a W)^2 f \ge \dots \ge \dsl^\infty_{k=0} (G a W)^k \sqrt{u}.
\end{equation}
It now follows from the expression for the exponential moment of $a \langle L^u,W\rangle$ in (\ref{1.13}) that
\begin{equation}\label{5.52}
\begin{split}
\log \IE[\exp\{a \langle L^u, W\rangle\}] & = u \big\langle a\,W, \dsl^\infty_{k=0} (G a  W)^k 1\big\rangle \stackrel{(\ref{5.51})}{\le} a\,\sqrt{u} \langle W,f\rangle
\\
&\hspace{-4ex} \stackrel{(\ref{5.25}),(\ref{5.27})}{\le}   a\,\sqrt{u}  \dsl_{B \in \cC_\tau} - L\,\wh{f} (B) \Big(1 + \mbox{\f $\dis\frac{c}{K}$}\Big) = a \,\sqrt{u} \Big(1 + \mbox{\f $\dis\frac{c}{K}$}\Big) \;\langle - L\,\wh{f},1\rangle.
\end{split}
\end{equation}

\n
Now with $C = \bigcup_{B \in \cC_\tau} B$, as in (\ref{5.34}), and $h_C$ its equilibrium potential, the Cauchy-Schwarz inequality yields that
\begin{equation}\label{5.53}
\langle -L \wh{f}, 1 \rangle = \langle - L\wh{f}, h_C\rangle =  \cE(\wh{f}, h_C) \le \cE(\wh{f},\wh{f})^{\frac{1}{2}} {\rm cap} (C)^{\frac{1}{2}} \le \fr \;\cE (\wh{f},\wh{f}) + \fr \;{\rm cap}(C).
\end{equation}
Coming back to (\ref{5.50}), we see that for large $N$ and all $\tau \in \{0,\dots,\ov{K} -1\}^d, f \in \cF_\tau$, one has
\begin{equation}\label{5.54}
\begin{split}
\log \IP [\cA_{\ov{f}}] & \le -a(1- \ve) \langle -L \wh{f},f \rangle + a \,\sqrt{u} \Big(1 + \mbox{\f $\dis\frac{c}{K}$}\Big) \langle - L \wh{f},1\rangle
\\
&\!\!\!\! \stackrel{(\ref{5.23})}{=} -a(1- \ve) \langle -L \wh{f},\wh{f} \rangle + a \,\sqrt{u} \Big(\ve + \mbox{\f $\dis\frac{c}{K}$}\Big) \langle - L \wh{f},1\rangle
\\
&\!\!\!\!  \stackrel{(\ref{5.53})}{\le} -a(1- \ve)\cE(\wh{f},\wh{f})  +  \mbox{\f $\dis\frac{a}{2}$}\;\sqrt{u}   \Big(\ve + \mbox{\f $\dis\frac{c}{K}$}\Big)\big(\cE(\wh{f},\wh{f}) + {\rm cap}(C)\big)
\\
&\!\!\!\stackrel{\ve \,K \ge c}{\le} -a \big(1-\ve(1 + \sqrt{u})\big) \, \cE(\wh{f},\wh{f}) + a \, \sqrt{u} \ve  \;{\rm cap}(C).
\end{split}
\end{equation}
For large $N$ and all $\tau \in \{0,\dots, \ov{K} - 1\}^d$, $C \subseteq B(0,2N)$ so that ${\rm cap}(C) \le c \,N^{d-2}$ by (\ref{1.25}). With the notation (\ref{5.48}), we thus find that (recall that $a<1$):
\begin{equation}\label{5.55}
\limsup\limits_N \; \mbox{\f $\dis\frac{1}{N^{d-2}}$} \;\sup\limits_\tau \;\;\sup\limits_{\ov{f} \in \cF_\tau} \; \log \IP[\cA_{\ov{f}}] \le - a\big(1-\ve(1 + \sqrt{u})\big) \,I_{\ve,K} + c\, \sqrt{u} \ve.
\end{equation}
As mentioned above, in view of Proposition \ref{prop5.3}, the claim (\ref{5.49}) follows and Proposition \ref{prop5.6} is proved.
\end{proof}

We will now look for a meaningful lower bound on $I_{\ve,K}$ that feels the ``legacy'' of the constraint involving $\theta$ that appears in the second line of (\ref{5.11}). We set for $b \ge 0$ and $r \ge 1$
\begin{equation}\label{5.56}
\begin{split}
J^D_{b,r} =  \inf\Big\{&\; \fd \;\dis\int_{\IR^d} |\nabla \varphi |^2 dz; \varphi \ge 0 \;\mbox{supported in}\;  B_{\IR^d}(0,400r), 
\\ 
&\;\varphi \in H^1 (\IR^d), \strokedint_D \theta\big((\sqrt{u} + \varphi)^2\big)\,dz \ge b\Big\}.
\end{split}
\end{equation}

\begin{proposition}\label{prop5.7}
For $\ve$ as in (\ref{5.4}) and $K \ge 100$, one has for each integer $r \ge 10$
\begin{equation}\label{5.57}
\Big(1 + \mbox{\f $\dis\frac{c_{10}}{r^{d-2}}$}\Big)\; I_{\ve,K} \ge J^D_{\nu - \ve,r}.
\end{equation}
\end{proposition}

\begin{proof}
We first need some notation. For each $N \ge 1$, we denote by $\IL_N$ the scaled lattice $\frac{1}{N} \;\IZ^d$, and for $h$ a function on $\IL_N$ we consider the Dirichlet form
\begin{equation}\label{5.58}
\cE_N(h,h) = \mbox{\f $\dis\frac{1}{2N^{d-2}}$} \;\dsl_{y \underset{N}{\sim} y'} \;\fr \;\big(h(y') - h(y)\big)^2 \;(\le \infty),
\end{equation}
where $y \underset{N}{\sim} y'$ means that $y$ and $y'$ are neighbors in $\IL_N$. In particular, when $N =1$ and $h$ is a function on $\IZ^d = \IL_{N=1}$, we find that $\cE_{N=1} (h,h) = d \cE(h,h)$ in the notation of (\ref{1.2}). When $h_1,h_2$ are functions on $\IL_N$ with finite Dirichlet form, as usual, $\cE_N(h_1,h_2)$ is defined by polarization.

\medskip
Let us now prove (\ref{5.57}). Without loss of generality, we can assume that the left member of (\ref{5.57}) is finite. For any $K ( \ge 100)$, for which $I_{\ve,K} < \infty$, we can find by (\ref{5.48}) an increasing sequence $N_\ell, \ell \ge 1$, as well as $\tau_0$ in $\{0,\dots,\ov{K}-1\}^d$ and $\ov{f}_\ell$ in $\cF_{\tau_0}$ such that setting
\begin{equation}\label{5.59}
\mbox{$\varphi_\ell(z) = \wh{f}_\ell (N_\ell z)$, for $z \in \IL_{N_\ell}$, with $\wh{f}_\ell \ge 0$ attached to $\ov{f}_\ell$ as in (\ref{5.23})},
\end{equation}
one has
\begin{equation}\label{5.60}
I_{\ve,K} = \lim\limits_\ell \; \mbox{\f $\dis\frac{1}{d}$}  \;\cE_{N_\ell}(\varphi_\ell,\varphi_\ell) < \infty.
\end{equation}
For large $\ell$, the non-negative functions $\varphi_\ell$ on $\IL_{N_\ell}$ are (for the simple random walk on $\IL_{N_\ell}$) harmonic outside $[-2,2]^d \cap \IL_{N_\ell}$, and tend to zero at infinity. By the ``cut-off lemma'', i.e.~Lemma 5.7 of \cite{LiSzni15}, we can find for each integer $r \ge 10$ and each $\ell \ge 1$, $\ov{\varphi}_\ell$ on $\IL_{N_\ell}$ such that
\begin{equation}\label{5.61}
\left\{ \begin{array}{rl}
{\rm i)} &\mbox{$\ov{\varphi}_\ell = \varphi_\ell$ on $[-r,r]^d \cap \IL_{\IN_\ell}$},
\\[1ex]
{\rm ii)} & \mbox{$\ov{\varphi}_\ell = 0$ outside $C_r \cap \IL_{N_\ell}$, where $C_r= B_{\IR^d}(0,400 r)$},
\\[1ex]
{\rm iii)} & \cE_{N_\ell} (\ov{\varphi}_\ell, \ov{\varphi}_\ell) \le \cE_{N_\ell} (\varphi_\ell,\varphi_\ell) \;\Big(1 + \mbox{\f $\dis\frac{c_{10}}{r^{d-2}}$}\Big).
\end{array}\right.
\end{equation}
We then introduce the sequence of functions $\Phi_\ell$ on $\IR^d$ such that $\Phi_\ell$ is constant on each box $y + \frac{1}{N_\ell} \;[0,1)^d$, $y \in \IL_{N_\ell}$, where it equals $\ov{\varphi}_\ell(y)$, namely:
\begin{equation}\label{5.62}
\Phi_\ell(z) = \dsl_{y \in \IL_{N_\ell}} \ov{\varphi}_\ell(y) \;1\Big\{z \in y + \mbox{\f $\dis\frac{1}{N_\ell}$} \;[0,1)^d\Big\}, \;z \in \IR^d, \ell \ge 1.
\end{equation}

\n
By (\ref{5.37}) of \cite{LiSzni15}, we can extract a subsequence (still denoted by $\Phi_\ell$ for notational convenience) such that
\begin{equation}\label{5.63}
\mbox{$0 \le \Phi_\ell \;\underset{\ell}{\longrightarrow} \Phi$ in $L^2(\IR^d)$ and a.e., with $\Phi = 0$ outside $C_r$,} 
\end{equation}
and by (\ref{5.42}) of the same reference $\Phi$ belongs to $H^1(\IR^d)$, and
\begin{equation}\label{5.64}
\fd \;\dis\int_{\IR^d} | \nabla \Phi |^2 \,dz \le I_{\ve,K} \Big(1 + \mbox{\f $\dis\frac{c_{10}}{r^{d-2}}$}\Big).
\end{equation}

\n
At this stage, our aim is to see that $\Phi$ still feels the constraint on $\theta$ present in the definition of $\cF_{\tau_0}$ in (\ref{5.11}), see (\ref{5.71}) below. With this goal in mind, for $\tau \in \{0,\dots,\ov{K}-1\}^d$ and $\ell \ge 1$, we consider $C_{\tau,\ell}$ the ``scaled and solid'' union of the boxes $B$ in $\cC_\tau$, cf.~(\ref{5.10}), namely:
\begin{equation}\label{5.65}
C_{\tau,\ell} = \mbox{\f $\dis\frac{1}{N_\ell}$} \;\big(\bigcup\limits_{B \in \cC_\tau} \; \bigcup\limits_{x \in B} \;x + [0,1)^d\big) \subseteq \IR^d.
\end{equation}

\n
Note that for each $\ell \ge 1$ the sets $C_{\tau,\ell}$ are pairwise disjoint as $\tau$ varies over $\{0,\dots,\ov{K}-1\}^d$ (since the collections $\cC_\tau$ are pairwise disjoint), and by (\ref{5.10}), and the fact that $D$ is by (\ref{0.6}) either a closed box, or the closure of a smooth bounded domain in $\IR^d$, their union is for large $\ell$ close to $D$ in the sense that (with $\Delta$ denoting the symmetric difference)
\begin{equation}\label{5.66}
\big| \big(\bigcup\limits_\tau \;C_{\tau,\ell}\big) \,\Delta D \big| \underset{\ell}{\longrightarrow} 0 \;\;\mbox{(where $\tau$ ranges over $\{0,\dots,\ov{K} - 1\}^d$ in the union)}.
\end{equation}
Moreover, the sets $C_{\tau,\ell}$ are nearly translates of each other as $\tau$ varies over $\{0,\dots,\ov{K} - 1\}^d$. Namely, with $L_\ell$ as in (\ref{4.8}), where $N = N_\ell$, one has
\begin{equation}\label{5.67}
\Big| C_{\tau,\ell} \; \Delta \Big(\mbox{\f $\dis\frac{L_\ell}{N_\ell}$} \; \tau + C_{\tau = 0,\ell}\Big) \Big| \underset{\ell}{\longrightarrow}  0, \; \mbox{for each $\tau \in \{0,\dots,\ov{K} -1\}^d$}.
\end{equation}

\n
Recall the construction of $\varphi_\ell$ in (\ref{5.59}) and note that for each $B \in \cC_{\tau_0}$ one has $\wh{f}_\ell + \sqrt{u} \ge f_{\ell,B}$, if $\ov{f}_\ell = (f_{\ell,B})_{B \in \cC_{\tau_0}}$, see (\ref{5.23}). Hence, for large $\ell$, it holds that
\begin{equation}\label{5.68}
\begin{array}{l}
\dis\int_{C_{\tau_0,\ell}} \theta\big((\sqrt{u} + \Phi_\ell(z))^2\big) \,dz \; \underset{(\ref{5.65})}{\stackrel{(\ref{5.62}),(\ref{5.61})}{=}} \;\dsl_{x \in \bigcup_{\cC_{\tau_0}}B} \theta\big(\big(\sqrt{u} + \wh{f}_\ell (x)\big)^2\big)\,N_\ell^{-d} \stackrel{(\ref{5.23})}\ge
\\
\mbox{\f $\dis\frac{1}{N_\ell^d}$} \;\dsl_{B \in \cC_{\tau_0}} \theta(f^2_{\ell,B}) \,|B| \stackrel{(\ref{5.11})}{\ge} \mbox{\f $\dis\frac{1}{N_\ell^d}$}  \; \mbox{\f $\dis\frac{(\nu - \ve)}{\ov{K}\,^d}$} \;|D_{N_\ell}| \; \underset{\ell \r \infty}{\longrightarrow} \mbox{\f $\dis\frac{\nu - \ve}{\ov{K}\,^d}$} \;|D| 
\end{array}
\end{equation}
(where $|D_{N_\ell}|$ stands for the number of points in $D_{N_\ell}$ and $|D|$ for the Lebesgue measure of $D$).

\medskip
In addition, since $\theta((\sqrt{u} + \point)^2)$ is bounded continuous, and $C_{\tau_0,\ell}$ is contained in $[-2,2]^d$ for large $\ell$, it follows from (\ref{5.63}) by dominated convergence that
\begin{equation}\label{5.69}
\lim\limits_\ell \;\Big| \dis\int_{C_{\tau_0,\ell}} \theta\big(\big(\sqrt{u} + \Phi_\ell(z)\big)^2\big)\,dz - \dis\int_{C_{\tau_0,\ell}} \theta\big(\big(\sqrt{u} + \Phi (z)\big)^2\big)\,dz\Big| = 0.
\end{equation}
Moreover, the map $z \in \IR^d \r \Phi(z + \cdot) \in L^2(\IR^d)$ is continuous. This readily implies the continuity of the map $z \in \IR^d \r \theta((\sqrt{u} + \Phi(z+\cdot))^2) \in L^1_{{\rm loc}} (\IR^d)$. By (\ref{5.67}) and the fact that for large $\ell$ all $C_{\tau,\ell}$, $\tau \in \{0,\dots,\ov{K} -1\}^d$ are contained in $[-2,2]^d$, we then find that for all $\tau \in \{0,\dots, \ov{K} - 1\}^d$,
\begin{equation}\label{5.70}
\lim\limits_\ell \Big| \dis\int_{C_{\tau,\ell}}  \theta\big(\big(\sqrt{u} + \Phi(z)\big)^2\big)\,dz - \dis\int_{C_{\tau_0,\ell}} \theta\big(\big(\sqrt{u} + \Phi(z)\big)^2\big)\,dz\Big| = 0.
\end{equation}
Taking into account that the $C_{\tau,\ell}$ are pairwise disjoint as $\tau$ varies over $\{0,\dots,\ov{K} - 1\}^d$ and that (\ref{5.66}) holds, we find that
\begin{equation}\label{5.71}
\begin{split}
\dis\int_D  \theta\big(\big(\sqrt{u} + \Phi(z)\big)^2\big)\,dz & = \lim\limits_\ell \; \dsl_{\tau} \;\dis\int_{C_{\tau,\ell}}  \theta\big((\sqrt{u} + \Phi)^2\big)\,dz
\\
&\!\!\!\!\! \stackrel{(\ref{5.70})}{=} \ov{K}\,^{\!d}\; \lim\limits_\ell \; \dis\int_{C_{\tau_0,\ell}} \; \theta\big((\sqrt{u} + \Phi)^2\big)\,dz 
\\
&\!\!\!\!\! \stackrel{(\ref{5.69})}{=}  \ov{K}\,^{\!d}\; \lim\limits_\ell \;\dis\int_{C_{\tau_0,\ell}} \theta\big((\sqrt{u} + \Phi_\ell)^2\big)\,dz
\\[-1ex]
&\!\!\!\!\! \stackrel{(\ref{5.68})}{\ge} (\nu - \ve) \,|D|.
\end{split}
\end{equation}
Combining (\ref{5.63}), (\ref{5.64}), and (\ref{5.71}), we see that for $r \ge 10$, the left member of (\ref{5.57}) is bigger or equal to the right member of (\ref{5.57}). This proves Proposition \ref{prop5.7}.
\end{proof}

We are now ready to complete the proof of Theorem \ref{theo5.1}.

\bigskip\n
{\it Proof of Theorem \ref{theo5.1}:} Consider $\ve$ as in (\ref{5.4}) and $r \ge 10$ integer. Letting successively $K$ tend to infinity and $a$ tend to $1$, Propositions \ref{prop5.6} and \ref{prop5.7} show that
\begin{equation}\label{5.72}
\limsup\limits_N \; \mbox{\f $\dis\frac{1}{N^{d-2}}$} \; \log \IP[\cA_N] \le - \big(1 - \ve(1 + \sqrt{u})\big) \,J^D_{\nu - \ve,r} \Big(1 + \mbox{\f $\dis\frac{c_{10}}{r^{d-2}}$}\Big)^{-1} + c \,\sqrt{u} \;\ve.
\end{equation}
We can now let $\ve$ to $0$ and find (note that $J^D_{b,r}$ is non-decreasing in $b$)

\vspace{-3ex}
\begin{equation}\label{5.73}
\limsup\limits_N \; \mbox{\f $\dis\frac{1}{N^{d-2}}$} \; \log \IP[\cA_N] \le -   \Big(1 + \mbox{\f $\dis\frac{c_{10}}{r^{d-2}}$}\Big)^{-1} \;\lim\limits_{\ve \r 0} \;J^D_{\nu - \ve,r}  .
\end{equation}
We can assume that the right member of (\ref{5.73}) is finite (otherwise Theorem \ref{theo5.1} is proved). If we now consider $\ve_n \r 0$, and $\varphi_n \ge 0$ in $H^1(\IR^d)$ supported in $C_r$ such that $\strokedint_D \theta((\sqrt{u} + \varphi_n)^2)\,dz \ge \nu - \ve_n$ with $\frac{1}{2d} \,\int_{\IR^d} |\nabla \varphi_n|^2 \,dz \underset{n}{\longrightarrow} \lim_{\ve \r 0} \,J^D_{\nu - \ve,r} < \infty$, we can, after extraction of a subsequence, assume that $\varphi_n$ converges in $L^2(C_r)$ and a.e.~to $\varphi \ge 0$ in $H^1(\IR^d)$ supported in $C_r$ with $\frac{1}{2d} \,\int_{\IR^d} |\nabla \varphi|^2 \,dz \le \lim_{\ve \r 0} J^D_{\nu - \ve, r}$ (see \cite{LiebLoss01}, p.~208 and 212). Hence, by dominated convergence, we have $\int_D \theta((\sqrt{u} + \varphi)^2)\,dz \ge \nu$, and this proves that for any $r \ge 10$ integer, in the notation of (\ref{5.2}):
\begin{equation}\label{5.74}
\limsup\limits_N \; \mbox{\f $\dis\frac{1}{N^{d-2}}$} \; \log \IP[\cA_N] \le -  \Big(1 + \mbox{\f $\dis\frac{c_{10}}{r^{d-2}}$}\Big)^{-1}  \; J^D_{\nu,r} \ge -  \Big(1 + \mbox{\f $\dis\frac{c_{10}}{r^{d-2}}$}\Big)^{-1}  \; J^D_\nu .
\end{equation}
Letting $r$ tend to infinity proves (\ref{5.1}) and concludes the proof of Theorem \ref{theo5.1}. \hfill $\square$

\begin{remark}\label{rem5.8} \rm
The argument below (\ref{5.73}) shows that when $r \ge \sup\{|z|_\infty$; $z \in D\}$, 
\begin{equation}\label{5.75}
\mbox{the non-decreasing function $b \in [0,\theta_\infty) \r J^D_{b,r} \in \IR_+$ is left-continuous}.
\end{equation}

\vspace{-3ex}
\hfill $\square$
\end{remark}

As we will now see, the respective lower and upper bounds on the principal exponential decay of $\IP[\cA_N]$ are matching. In what follows, $D^1(\IR^d)$ stands for the space of locally integrable functions $\varphi$ on $\IR^d$ that vanish at infinity, i.e.~that are such that $|\{|\varphi | > a\}| < \infty$ for all $a > 0$, and have a finite Dirichlet integral $\int_{\IR^d} |\nabla \varphi|^2 \,dz < \infty$, see Chap. 8~\S 2 in \cite{LiebLoss01}.

\begin{corollary}\label{cor5.9}
Assume that the local function $F$ satisfies (\ref{2.1}), (\ref{4.1}), and that $D$ is of the form (\ref{0.6}). Then, for any $u > 0$ and $\nu$ in $(\theta(u),\theta_\infty)$, one has in the notation of (\ref{4.6}), (\ref{4.7}) and (\ref{5.2}),

\vspace{-3ex}
\begin{equation}\label{5.76}
\begin{array}{l}
\lim\limits_N \; \mbox{\f $\dis\frac{1}{N^{d-2}}$} \; \log \IP[\cA_N]  = -I^D_\nu = - J^D_\nu
\\[1ex]
= -\min \Big\{\fd \;\dis\int_{\IR^d} |\nabla \varphi |^2\,dz; \varphi \ge 0, \varphi \in D^1(\IR^d),  \strokedint_D \theta\big((\sqrt{u} + \varphi)^2\big) \,dz = \nu\Big\}
\end{array}
\end{equation}

\medskip\n
(in particular there is a minimizer in the variational problem on the last line of (\ref{5.76}), moreover, one can drop the condition $\varphi \ge 0$ in the variational problem).
\end{corollary}

\begin{proof}
We begin with the proof of the first two equalities. By Theorems \ref{theo4.2} and \ref{theo5.1}, it suffices to show the second equality. By direct inspection $I^D_\nu \ge J^D_\nu$, and it thus remains to show that
\begin{equation}\label{5.77}
I^D_\nu \le J^D_\nu .
\end{equation}

\n
For this purpose, note that for any non-negative compactly supported $\varphi$ in $H^1(\IR^d)$ that satisfies $\strokedint_D \theta((\sqrt{u} + \varphi)^2) \,dz \ge \nu$, one can add a non-negative smooth compactly supported function $\gamma$, which is positive on the compact set $D$, with small Dirichlet energy. Since $\theta$ is strictly increasing, $\strokedint_D \theta((\sqrt{u} + \varphi + \gamma)^2) \,dz >\nu$. Using regularization by convolution of $\varphi + \gamma$, one sees that $\int |\nabla \varphi |^2\,dz$ can be approximated by Dirichlet integrals of non-negative functions $\wt{\varphi} \in C^\infty_0(\IR^d)$, such that $\strokedint_D \theta((\sqrt{u} + \wt{\varphi})^2) \,dz > \nu$. This proves (\ref{5.77}) and the two equalities on the first line of (\ref{5.76}) follow.

\medskip
Let us now prove the last equality. We denote by $\wt{J}^D_\nu$ the quantity after the minus sign on the last line of (\ref{5.76}) with $\min$ replaced by $\inf$. Our next step is to establish that
\begin{equation}\label{5.78}
J^D_\nu = \wt{J}^D_\nu .
\end{equation}

\n
To this end, note that whenever $\varphi \ge 0$ with compact support is such that $\strokedint_D \theta((\sqrt{u} + \varphi)^2) \,dz \ge \nu$, then, by continuity, for some $0 \le a \le 1$, $\varphi_a = a \varphi$ satisfies $\strokedint_D \theta((\sqrt{u} + \varphi_a)^2) \,dz =\nu$. As a result, $J^D_\nu \ge \wt{J}^D_\nu$. To see the reverse inequality, note that any $\varphi \ge 0$, in $D^1(\IR^d)$, such that $\strokedint_D \theta((\sqrt{u} + \varphi)^2) \,dz = \nu$ ($> \theta(u)$), is not a.e.~equal to $0$ on $D$. Since $\theta$ is strictly increasing, $\varphi_a = a \varphi$, with $a > 1$ (close to $1$) satisfies $\strokedint_D \theta((\sqrt{u} + \varphi_a)^2) \,dz > \nu$. Then, see \cite{LiebLoss01}, p.~204, setting for $\delta > 0$ (small) $\varphi_{a,\delta} = (\varphi_a - \delta)_+ \wedge \delta^{-1}$, one has a bounded function with support of finite Lebesgue measure, smaller or equal Dirichlet integral, and such that $\strokedint_D \theta((\sqrt{u} + \varphi_{a,\delta})^2) \,dz > \nu$. Multiplying by a cut-off function, one can approximate $\varphi_{a,\delta}$ by a compactly supported $\wt{\varphi} \ge 0$ in $H^1(\IR^d)$ with Dirichlet integral close to that of $\varphi_{a,\delta}$ and $\strokedint_D \theta((\sqrt{u} + \wt{\varphi})^2) \,dz > \nu$, This shows that $\wt{J}^D_\nu \ge J^D_\nu$ and finishes the proof of (\ref{5.78}).

\medskip
To complete the proof of the second equality in (\ref{5.76}), there remains to show the existence of a minimizer for $\wt{J}^D_\nu$. To this end, consider a minimizing sequence $\varphi_n$. By Theorem 8.6, p.~208 and Corollary 8.7, p.~212 of \cite{LiebLoss01}, we can extract a subsequence still denoted by $\varphi_n$ that converges a.e.~and in $L^2_{\rm loc}(\IR^d)$ to $\varphi \ge 0$ in $D^1(\IR^d)$ such that $\int_{\IR^d} |\nabla \varphi |^2\,dz \le \liminf_n \int_{\IR^d} |\nabla \varphi_n|^2 \,dz$. By dominated convergence, we thus find that $\strokedint_D \theta((\sqrt{u} + \varphi)^2) \,dz = \nu$, and this shows that $\varphi$ is a minimizer. Finally, note that when $\varphi \in D^1(\IR^d)$, $|\varphi | \in D^1(\IR^d)$ and has smaller or equal Dirichlet integral, see p.~204 of \cite{LiebLoss01}. One can thus drop the condition $\varphi \ge 0$ in the last line of (\ref{5.76}) without changing the value of the minimum. This concludes the proof of Corollary \ref{cor5.9}.
\end{proof}

\begin{remark}\label{rem5.10} \rm 1) Under the assumptions of Corollary \ref{cor5.9}, note that when $\varphi$ is a minimizer of the variational problem on the second line of (\ref{5.76}), it follows using smooth perturbations, which are compactly supported in $\IR^d \backslash D$, that $\Delta \varphi = 0$ in $\IR^d \backslash D$ in the distribution sense, so that $\varphi$ is a non-negative smooth harmonic function in $\IR^d \backslash D$, see p.~127 of \cite{Dono69}. Then, from Harnack inequality, and the fact that $\varphi$ vanishes at infinity (in the sense stated above Corollary \ref{cor5.9}), one sees that $\varphi$ tends to $0$ at infinity. Hence, by routine direct comparison, setting $r_D = \sup\{|z|; z \in D\}$, $\varphi$ satisfies the bound
\begin{equation}\label{5.79}
0 \le \varphi(z) \le \Big(\mbox{\f $\dis\frac{2r_D}{|z|}$}\Big)^{d-2} \;\sup\limits_{|z'| = 2r_D} \varphi(z'), \;\mbox{when} \; |z| \ge 2r_D.
\end{equation}
In particular, when $d \ge 5$, $\varphi$ belongs to $L^2(\IR^d)$ and this shows that under the assumptions of Corollary \ref{cor5.9},
\begin{equation}\label{5.80}
\begin{array}{l}
\mbox{when $d \ge 5$, any minimizer in the second line of (\ref{5.76}) belongs to}
\\
H^1(\IR^d) \subseteq D^1(\IR^d)
\end{array}
\end{equation}
(and one can replace $D^1(\IR^d)$ by $H^1(\IR^d)$ in the second line of (\ref{5.76})).

\bigskip\n
2) when $D$ is a closed ball of positive radius centered at the origin, given a minimizer $\varphi$ of the variational problem in the second line of (\ref{5.76}), one can consider its symmetric-decreasing rearrangement $\varphi^*$, see Chapter 3~\S 3 of \cite{LiebLoss01}, and one knows, see p.~188-189 of the same reference that $\varphi^* \in D^1(\IR^d)$ and $\int_{\IR^d} |\nabla \varphi^*|^2\,dz \le \int_{\IR^d} |\nabla \varphi |^2\,dz$. In addition, if $\mu_D$ stands for the normalized Lebesgue measure on $D$, one has $\mu_D(\varphi > s) \le \mu_D (\varphi^* > s)$ for  each $s \in \IR$, so that
\begin{equation}\label{5.81}
\begin{split}
\nu = \strokedint_D \theta\big((\sqrt{u} + \varphi)^2\big)\,dz & = \dis\int^\infty_0 \mu_D \big(\theta\big((\sqrt{u} + \varphi)^2\big) > t \big) \,dt
\\[1ex]
& = \dis\int^\infty_0 \mu_D \big(\varphi > \sqrt{\theta^{-1}(t)} - \sqrt{u}\big) \, dt
\\[1ex]
&\le \dis\int^\infty_0 \mu_D \big(\varphi^* > \sqrt{\theta^{-1}(t)} - \sqrt{u}\big)\,dt = \strokedint_D \theta\big((\sqrt{u} + \varphi^*)^2\big)\,dz.
\end{split}
\end{equation}

\medskip\n
Using continuity, this implies that for some $0 < a^* \le 1$, $a^* \,\varphi^*$ is a minimizer of the variational problem in the second line of (\ref{5.76}). But $a^* < 1$ is impossible (it would contradict the minimality of $\int |\nabla \varphi |^2\,dz$, which is positive since $\nu > \theta(u)$). So, the inequality in (\ref{5.81}) is an equality, and in particular, when $D$ is a closed ball of positive radius centered at the origin, under the assumption of Corollary \ref{cor5.9},
\begin{equation}\label{5.82}
\begin{array}{l}
\mbox{$\varphi^*$ is a minimizer of the variational problem in the second line of (\ref{5.76})}
\\
\mbox{whenever $\varphi$ is a minimizer of that variational problem}
\end{array}
\end{equation}

\n
(note that by 1) above $\varphi^* \ge 0$ is harmonic in $\IR^d \backslash D$). 

\bigskip\n
3) When the function $\eta(a) = \theta (a^2)$, $a \in \IR$, belongs to $C^2_b(\IR)$, the functional $A(\varphi) = \strokedint_D \theta((\sqrt{u} + \varphi)^2)\,dz$, $\varphi \in D^1(\IR^d)$ is $C^1$. This follows for instance from the identity (for $\varphi, \psi$ in $D^1(\IR^d)$):
\begin{equation*}
A(\varphi + \psi) - A(\varphi) - \strokedint_D \eta' (\sqrt{u} + \varphi) \,\psi\,dz = \dis\int^1_0 ds \,\dis\int^s_0 dt \, \strokedint_D \eta'' (\sqrt{u} + \varphi + t \psi)\, \psi^2 dz,
\end{equation*}

\n
and the control of the $L^{\frac{2d}{d-2}}(\IR^d)$-norm by the norm on $D^1(\IR^d)$, see p.~202 of \cite{LiebLoss01}.

\medskip
In addition, if $\theta '$ remains positive on $(0,\infty)$, then for any minimizer $\varphi$ of the variational problem on the second line of (\ref{5.76}) the differential of $A$ at $\varphi$ does not vanish, and by usual variational methods (see also p.~27 of \cite{Lang72}) the non-negative function $\varphi$, which vanishes at infinity, satisfies in the weak sense the semilinear equation
\begin{equation}\label{5.83a}
- \mbox{\footnotesize $\dis\frac{1}{2d}$} \;\Delta \varphi = \lambda(\sqrt{u} + \varphi) \,\theta ' \big((\sqrt{u} + \varphi)^2\big)\,1_D,
\end{equation}
for a suitable constant $\lambda$ (Lagrange multiplier), and by (\ref{5.79}), denoting by $\cG$ the convolution with the Brownian Green function (i.e.~$\frac{1}{2 \pi^{d/2}}\, \Gamma(\frac{d}{2} -1) \,| \cdot |_2^{-(d-2)}$), actually one has
\begin{equation}\label{5.84a}
\varphi = d \lambda \,\cG \big((\sqrt{u} + \varphi) \,\theta' \big((\sqrt{u} + \varphi)^2\big)\,1_D\big), \;\mbox{with $\lambda > 0$}
\end{equation}
(note that $\lambda \le 0$ is impossible).

\bigskip\n
4) As we now explain, under the assumptions of Corollary \ref{cor5.9}, for $u>0$,
\begin{equation}\label{5.84b}
\mbox{the map}\, \, \nu \in [\theta(u),\theta_\infty) \longrightarrow I^D_\nu \in [0,\infty) \,\,  \mbox{is an increasing homeomorphism}.
\end{equation}

\medskip
Indeed, by direct inspection, see (\ref{4.7}), the map is non-decreasing, right-continuous. It is also straightforward to argue that it vanishes for $\nu = \theta(u)$. By the equality with the minimum $\wt{J}^D_\nu$ on the last line of (\ref{5.76}) the map is also left-continuous: this follows by considering an increasing sequence $\nu_n$ tending to $\nu \in (\theta(u),\theta_\infty)$ and a corresponding sequence $\varphi_n$ of minimizers. As in the last paragraph of the proof of Corollary \ref{cor5.9}, after extraction of a subsequence still denoted by $\varphi_n$ that converges a.e. to $\varphi$, one finds that $\wt{J}^D_\nu$ is at most (and hence equal to) $\lim_n \wt{J}^D_{\nu_n}$. The map in (\ref{5.84b}) is thus continuous. 

\n
It is also strictly increasing. Indeed, if $\nu < \nu'$ and $\varphi'$ is a corresponding minimizer for $\wt{J}^D_{\nu'}$ then $\varphi'$ is not the null function, and for some $a$ in $[0,1)$, the function $a\varphi'$ belongs to the set that appears in the minimization problem defining $\wt{J}^D_\nu$, so that $\wt{J}^D_\nu <  \wt{J}^D_{\nu'}$. 

\medskip\n
There only remains to see that the map in (\ref{5.84b}) tends to infinity as $\nu$ tends to $\theta_\infty$. But otherwise, by a similar argument as in the above proof of the left-continuity, we would find $\varphi \ge 0$ in $D^1(\IR^d)$ such that $\strokedint_D \theta((\sqrt{u} + \varphi)^2) \,dz = \theta_\infty$. This is impossible since $\theta$ is a strictly increasing function. Thus, (\ref{5.84b}) holds true. 

\medskip
Incidentally, we provide in Remark \ref{rem6.6} 1) an example that shows that the increasing homeomorphism in (\ref{5.84b}) above need not be a convex function. \hfill $\square$
\end{remark}

We close this section with a consequence for the simple random walk of the upper bound that we have obtained in (\ref{5.74}). To this end, we denote by $L_x = \int^\infty_0 1\{X_s = x\}\,ds$, for $x \in \IZ^d$, the total time spent at $x$ by the continuous-time simple random walk $X_\point$, and for a local function $F$ as in (\ref{2.1}), $D$ as in (\ref{0.6}), and $D_N = (ND) \cap \IZ^d$ for $N \ge 1$, we consider the event
\begin{equation}\label{5.83}
\cA^0_N = \big\{\dsl_{x \in D_N} F\big((L_{x + \point})\big) > \nu \,|D_N|\big\}, \;\mbox{where $0 < \nu < \theta_\infty$}.
\end{equation}
We have the following
\begin{corollary}\label{cor5.11}
({\it large deviation upper bound for the simple random walk}) 

\smallskip\n
Assume that the local function $F$ satisfies (\ref{2.1}), (\ref{4.1}), and $D$ is of the form (\ref{0.6}). Then, for any $\nu$ such that $0 < \nu < \theta_\infty$ and $y \in \IZ^d$, one has
\begin{equation}\label{5.84}
\begin{array}{l}
\limsup\limits_N \; \mbox{\f $\dis\frac{1}{N^{d-2}}$} \;\log P_y [\cA^0_N] \le 
\\
- \inf \Big\{\fd \;\dis\int_{\IR^d} |\nabla \varphi |^2\,dz;  \varphi \ge 0, \; \mbox{compactly supported in} \; H^1(\IR^d), \strokedint_D \theta(\varphi^2) \ge \nu\Big\}.
\end{array}
\end{equation}
\end{corollary}

\begin{proof}
We consider $u > 0$ such that $\theta(u) < \nu$ and $y \in \IZ^d$. As explained below (\ref{5.3}), we can also assume that $D \subseteq [-1,1]^d$. As in the proof of Corollary 7.3 of \cite{Szni15}, we can find a coupling $\ov{P}$ of the field of occupation times $(L^u_\point)$ of random interlacements under $\IP [\cdot \,|y \in \cI^u]$, with $(L_\point)$ under $P_y$, such that $\ov{P}$-a.s., $L_x \le L^u_x$ for all $x \in \IZ^d$. Then, we see that for $N \ge 1$, with $\cA_N$ as in (\ref{4.6}),

\vspace{-4ex}
\begin{equation}\label{5.85}
P_y [\cA^0_N] \le \IP [\cA_N \, | \,y \in \cI^u],
\end{equation}

\n
and since $\IP[y \in \cI^u] = 1 - e^{-u/g(0,0)}$, we find that by (\ref{5.74}), that for any $r \ge 10$, and $u > 0$ with $\theta(u) < \nu$ (and $C_r = B_{\IR^d}(0,400 r)$),

\vspace{-3ex}
\begin{equation}\label{5.86}
\begin{array}{l}
\limsup\limits_N \; \mbox{\f $\dis\frac{1}{N^{d-2}}$} \;\log P_y [\cA^0_N]  \le  \limsup\limits_N \;\mbox{\f $\dis\frac{1}{N^{d-2}}$} \;\log \IP[\cA_N]  \stackrel{(\ref{5.74})}{\le} 
\\[1ex]
- \Big(1 + \mbox{\f $\dis\frac{c_{10}}{r^{d-2}}$}\Big)^{-1}  \inf\Big\{ \fd \;\dis\int_{\IR^d} |\nabla \varphi |^2 \,dz;\;\varphi \ge 0 \;\mbox{supported in $C_r$,} 
\\[1ex]
\hspace{3.2cm}\mbox{in $H^1(\IR^d), \strokedint_D \theta\big((\sqrt{u} + \varphi)^2\big)\,dz \ge \nu\Big\}$}.
\end{array}
\end{equation}
As $u$ decreases to $0$, the infimum in the right member of (\ref{5.86}) increases to a limit $I \le I^0 = \inf\{\frac{1}{2d} \,\int |\nabla \varphi |^2 \,dz; \varphi \ge 0$, supported in $C_r$, in $H^1(\IR^d)$, $\strokedint_D \theta(\varphi^2)\,dz \ge \nu\}$. As we now explain, actually $I = I^0$. The argument is similar to the proof below (\ref{5.73}). One considers $u_n \downarrow 0$ and a corresponding sequence $\varphi_n \ge 0$, compactly supported in $C_r$, in $H^1(\IR^d)$ with $\strokedint_D \theta((\sqrt{u}_n + \varphi_n)^2)\,dz \ge \nu$ such that $\frac{1}{2d} \,\int |\nabla \varphi_n|^2 \,dz \r I$. After the possible extraction of a subsequence, one can assume that $\varphi_n$ converges a.e.~to $\varphi \ge 0$ compactly supported in $C_r$, with $\frac{1}{2d} \;\int |\nabla \varphi |^2 \, dz \le I$ and $\strokedint_D \theta(\varphi^2) \,dz \ge \nu$. This shows that $I^0 \le I$, and hence $I = I^0$. Thus, letting successively $u$ tend to $0$ and then $r$ to infinity proves (\ref{5.84}).
\end{proof}

\begin{remark}\label{rem5.12} \rm 1) One can naturally wonder whether (in the spirit of Theorem \ref{theo4.2})
\begin{equation}\label{5.87}
\begin{array}{l}
\liminf_N \; \mbox{\f $\dis\frac{1}{N^{d-2}}$} \; \log P_y [\cA^0_N] \ge
\\[1ex]
 -\inf\Big\{ \fd \;\dis\int_{\IR^d} |\nabla \varphi |^2 \,dz; \varphi \in C^\infty_0 (\IR^d) \;\mbox{and} \; \strokedint_D \theta(\varphi^2)\, dz > \nu\Big\}
 \end{array}
\end{equation}
that is, whether a lower bound can also be derived. The consideration of so-called {\it tilted walks}, as was performed in \cite{Li17}, in the context of a problem of disconnection by the simple random walk, should be helpful on this matter. 

\bigskip\n
2) In Corollary \ref{cor5.11}, the same arguments as below (\ref{5.78}) show that the infimum in the right member of (\ref{5.84}) actually coincides with
\begin{equation}\label{5.88}
\min \Big\{ \fd \;\dis\int_{\IR^d} |\nabla \varphi|^2 \,dz; \;\varphi \ge 0, \varphi \in D^1 (\IR^d), \;\strokedint_D \theta(\varphi^2) \,dz = \nu\Big\}
\end{equation}
(and one can drop the condition $\varphi \ge 0$ without changing the above value). \hfill $\square$
\end{remark}

\section{Some applications}
\setcounter{equation}{0}

In this short section we will discuss some consequences of the results in the previous sections for the interlacement sausage, see Theorem \ref{theo6.1}, and for the finite pockets in the vacant set, see Theorem \ref{theo6.3} and Proposition \ref{prop6.5}. We also explain in Remark \ref{rem6.6} how these last results relate to the study of macroscopic holes in connected components of the vacant set of random interlacements in the strongly percolative regime, from \cite{Szni}.

\bigskip\n
{\bf Interlacement sausage}

\smallskip\n
We consider an integer $R \ge 0$ and define the interlacement sausage of radius $R$ at level $u \ge 0$, as
\begin{equation}\label{6.1}
\cI^{u,R} = \bigcup\limits_{x \in \cI^u} B(x,R) \subseteq \IZ^d
\end{equation}
(in particular $\cI^{u,R = 0} = \cI^u$ the interlacement at level $u$).

\medskip
We can now apply Corollary \ref{cor5.9} to the local function $F(\ell) = 1\{\sum_{x \in B(0,R)} \ell_x > 0\}$ from (\ref{2.5}). Note that when $R = 0$, then ${\rm cap}(B(0,R)) = g(0,0)^{-1}$, with $g(\cdot,\cdot)$ the Green function of the simple random walk. We recall that $D$ is as in (\ref{0.6}) (for instance, a closed ball of positive radius centered at the origin for the supremum or the Euclidean distance on $\IR^d$), and $D_N = (ND) \cap \IZ^d$.

\begin{theorem}\label{theo6.1}
For $R \ge 0$ integer, $u > 0$, and $1 - e^{-u \,{\rm cap}(B(0,R))} < \nu < 1$, one has
\begin{equation}\label{6.2}
\begin{array}{l}
\lim\limits_N \; \mbox{\f $\dis\frac{1}{N^{d-2}}$} \;\log \IP\big[ | \cI^{u,R} \cap D_N| > \nu \,|D_N|\big]
\\[1ex]
= - \inf\Big\{ \fd \;\dis\int_{\IR^d} |\nabla \varphi |^2 \,dz; \;\varphi \ge 0, \varphi \in C^\infty_0(\IR^d), \;\strokedint_D 1 - e^{-(\sqrt{u} + \varphi)^2 {\rm cap}(B(0,R))} dz > \nu\Big\}
\\[1ex]
= - \min \Big\{ \fd \;\dis\int_{\IR^d} |\nabla \varphi |^2 \,dz; \;\varphi \ge 0, \varphi \in D^1(\IR^d), \;\strokedint_D 1 - e^{-(\sqrt{u} + \varphi)^2 {\rm cap}(B(0,R))} dz = \nu\Big\}
\end{array}
\end{equation}
(one can drop the requirement $\varphi \ge 0$ in the above minimum, and when $d \ge 5$ replace $D^1(\IR^d)$ by the Sobolev space $H^1(\IR^d)$).
\end{theorem}

\begin{remark}\label{rem6.2} \rm 1) For instance when $D_N = B(0,N)$, the above result pertains to the probability of an {\it excess} of the interlacement sausage in a large box centered at the origin. The result has a similar flavor to Theorem 1 of \cite{BoltHollVanb01}, where moderate deviations of the Wiener sausage were analyzed, and it was shown that for $W^R(t)$ the Wiener sausage of radius $R > 0$ of Brownian motion in $\IR^d$ up to time $t$, and any $\nu > 0$
\begin{equation}\label{6.3}
\begin{array}{l}
\lim\limits_{t \r \infty} \; t^{-\frac{(d-2)}{d}} \log \IP \big[ |W^R (t)| \le \nu \,t\big]
\\
= - \inf \Big\{ \fr \;\dis\int_{\IR^d} |\nabla \varphi |^2 \,dz; \;\varphi \in H^1(\IR^d), \dis\int_{\IR^d} \varphi^2 dz = 1, \dis\int_{\IR^d} 1 - e^{-\kappa_R \varphi^2} dz \le \nu\Big\}
\end{array}
\end{equation}
(where $\kappa_R$ is the Brownian capacity of the Euclidean ball of radius $R$, and in the above infimum $\le \nu$ can be replaced by $= \nu$, see Lemma 12 of \cite{BoltHollVanb01}).

\medskip
In our context, the ``Swiss cheese picture'' advocated by \cite{BoltHollVanb01}, see also the discussions in  \cite{AsseScha17}, \cite{AsseScha18}, corresponds to the {\it tilted interlacements} entering the change of probability method in the proof of the lower bound in Section 4, see Theorem \ref{theo4.2}.

\bigskip\n
2) We refer to Remark \ref{rem5.10} for further properties of minimizers $\varphi$ of the variational formula on the last line of (\ref{6.2}). In particular, $\varphi$ is harmonic outside $D$ and tends to $0$ at infinity, see (\ref{5.79}). In addition, the considerations of Remark \ref{rem5.10} 3) apply, and the non-negative function $\varphi$ satisfies in a weak sense the semilinear equation
\begin{equation}\label{6.4}
- \mbox{\footnotesize $\dis\frac{1}{2d}$} \;\Delta \varphi = \lambda (\sqrt{u} + \varphi) \,e^{-(\sqrt{u} + \varphi)^2 {\rm cap}(B(0,R))} 1_D
\end{equation}
for a suitable constant $\lambda >0$ (Lagrange multiplier). \hfill $\square$
\end{remark}

\medskip\n
{\bf Finite pockets in the vacant set}

\smallskip\n
We now consider $0 \le r < R$, and the set of $(r,R)$-disconnected sites at level $u \ge 0$:
\begin{equation}\label{6.5}
\begin{split}
\cD^{u,r,R} & = \{x \in \IZ^d; \;\mbox{there is no path in $\cV^u$ between $B(x,r)$ and $S(x,R)\}$}
\\
& = \{B(x,r) \Vu S(x,R)\}
\end{split}
\end{equation}
(so when $r = 0$, $\cD^{u,r = 0,R}$ corresponds to the sites disconnected by $\cI^u$ within sup-distance $R$). We also define
\begin{equation}\label{6.6}
\theta_{r,R}(u) = \IP\big[B(0,r) \Vu S(0,R)\big], \;\mbox{for $u \ge 0$}.
\end{equation}
As noted in (\ref{2.8}), this is an analytic function, increasing on $\IR_+$. It is also convenient to introduce for $D$ satisfying the condition (\ref{0.6}) and $u>0$ the notation
\begin{equation}\label{6.7}
\begin{array}{l}
K_{r,R} (a) = 
\\
\inf \Big\{ \fd \;\dis\int_{\IR^d} |\nabla \varphi |^2 \,dz; \;\varphi \ge 0, \;\mbox{in $C^\infty_0(\IR^d)$, and} \; \strokedint_D \theta_{r,R}\big((\sqrt{u} + \varphi)^2\big)\,dz  > a \Big\}, \;a \ge 0.
\end{array}
\end{equation}

\medskip
The application of Corollary \ref{cor5.9} now yields (with the notation (\ref{6.5}))

\begin{theorem}\label{theo6.3}
For $0 \le r < R$ integers, $u > 0$, and $\theta_{r,R}(u) < \nu < 1$, one has
\begin{equation}\label{6.8}
\begin{array}{l}
\lim\limits_N \; \mbox{\f $\dis\frac{1}{N^{d-2}}$} \;\log \IP\big[ |\cD^{u,r,R} \cap D_N| > \nu \,|D_N|\big] = - K_{r,R}(\nu)
\\
= - \min \Big\{ \fd \;\dis\int_{\IR^d} |\nabla \varphi |^2 \,dz; \;\varphi \ge 0, \varphi \in D^1(\IR^d), \;\mbox{and} \; \strokedint_D \theta_{r,R}\big((\sqrt{u} + \varphi)^2\big) \,dz = \nu\Big\}
\end{array}
\end{equation}
(one can drop the requirement $\varphi \ge 0$ in the above minimum, and when $d \ge 5$ replace $D^1(\IR^d)$ by the Sobolev space $H^1(\IR^d)$).
\end{theorem}

\begin{remark}\label{rem6.4} \rm The application of Russo's formula for random interlacements, see (2) in Theorem 1 of \cite{DebePopo15}, shows that for all $u > 0$, $\theta'_{r,R}(u) > 0$. In addition, by the inclusion-exclusion formula (2.17) of \cite{Szni10}, the bounded function $\theta_{r,R}(u)$ is a finite linear combination of functions $e^{- \rho u}$, where $\rho \ge 0$. Thus, the function $\eta(a) = \theta_{r,R} (a^2)$ belongs to $C^2_b(\IR)$, and the considerations of Remark \ref{rem5.10} 1) and 3) apply:
any minimizer $\varphi$ in the second line of (\ref{6.8}) decays at infinity, see (\ref{5.79}), and satisfies in a weak sense the equation
\begin{equation}\label{6.9}
- \mbox{\footnotesize $\dis\frac{1}{2d}$} \;\Delta \varphi = \lambda (\sqrt{u} + \varphi) \,\theta'_{r,R} \big((\sqrt{u} + \varphi)^2\big) \,1_D
\end{equation}
for a suitable positive constant $\lambda$ (Lagrange multiplier) (or the integral equation corresponding to (\ref{5.84a})). Moreover, when $D$ is a closed ball of positive radius centered at the origin, there is a spherically symmetric function, which is non-increasing in the radius, minimizing the variational problem on the second line of (\ref{6.8}). \hfill $\square$
\end{remark}

We will now prove a result that provides a heuristic link between Theorem \ref{theo6.3} and Theorems 3.1 and 3.2 of \cite{Szni}, where upper and lower exponential bounds on the large deviation probability that the adequately thickened component of $S(0,N)$ in the vacant set $\cV^u$ leaves a macroscopic volume $\wt{\nu}\,N^d$ in its complement, see Remark \ref{rem6.6} 2) below.

\medskip
We first collect some facts concerning the functions $K_{r,R}(\cdot)$ in (\ref{6.7}). We note that the functions $\theta_{r,R}(\cdot)$ increase with $R$ and decrease with $r$, so that:
\begin{equation}\label{6.10}
\left\{ \begin{array}{l}
K_{r,R_1}(\cdot) \ge K_{r,R_2} (\cdot), \; \mbox{if $0 \le r < R_1 \le R_2$, and}
\\[1ex]
K_{r_1,R}(\cdot) \le K_{r_2,R}, \; \mbox{if $0 \le r_1 \le r_2 < R$}.
\end{array}\right.
\end{equation}
We can thus let $R$ tend to infinity and define for $r \ge 0$,
\begin{equation}\label{6.11}
K_r(a) = \lim\limits_{R \r \infty} \downarrow K_{r,R}(a), \; \mbox{for $a \ge 0$}
\end{equation}
(and the arrow $\downarrow$ reflects that this is a non-increasing limit). As a consequence of the second line of (\ref{6.10}), one sees that $K_r(\cdot)$ is non-decreasing in $r$, and we define
\begin{equation}\label{6.12}
K(a) = \lim\limits_{r \r \infty} \uparrow K_r(a), \; \mbox{for $a \ge 0$}
\end{equation}
(and the arrow $\uparrow$ reflects that this is a non-decreasing limit).

\medskip
The next proposition relates Theorem \ref{theo6.3} to the results of Section 3 of \cite{Szni}, see Remark \ref{rem6.6} below. For simplicity, we state the result in the case $D = [-1,1]^d$. We let $\omega_d$ stand for the volume of the Euclidean ball of unit radius in $\IR^d$.

\begin{proposition}\label{prop6.5}
Assume that $D = [-1,1]^d$. For $\nu \ge 0$ set $\wt{\nu} = \nu \, |D| ( = 2^d \nu)$, and denote by $B_{\wt{\nu}}$ the Euclidean ball of volume $\wt{\nu}$ centered at the origin. Then,
\begin{equation}\label{6.13}
K(\nu) = \mbox{\f $\dis\frac{1}{d}$} \;(\sqrt{u_*} - \sqrt{u})^2 \,{\rm cap}_{\IR^d}(B_{\wt{\nu}}), \; \mbox{when $\wt{\nu} < \omega_d$ and $0 < u < u_*$},
\end{equation}

\n
with $u_* \in (0,\infty)$ the critical level for the percolation of the vacant set of random interlacements, and ${\rm cap}_{\IR^d}(\cdot)$ the Brownian capacity.
\end{proposition}

\begin{proof}
We first collect some observations concerning the $[0,1)$-valued functions $\theta_{r,R}$ in (\ref{6.6}). As already pointed out, $\theta_{r,R}$ is non-decreasing in $R$ and from (\ref{6.6})
\begin{equation}\label{6.14}
\lim\limits_{R \r \infty} \uparrow \theta_{r,R}(u) = \theta_r(u) \stackrel{\rm def}{=} \IP[B(0,r) \Vu \infty], \; \mbox{for $u \ge 0$}.
\end{equation}

\begin{center}
\psfrag{0}{$0$}
\psfrag{1}{$1$}
\psfrag{u}{$u_*$}
\psfrag{t0}{$\theta_0$}
\psfrag{tr}{$\theta_r$}
\psfrag{trR}{$\theta_{r,R}$}
\includegraphics[width=10cm]{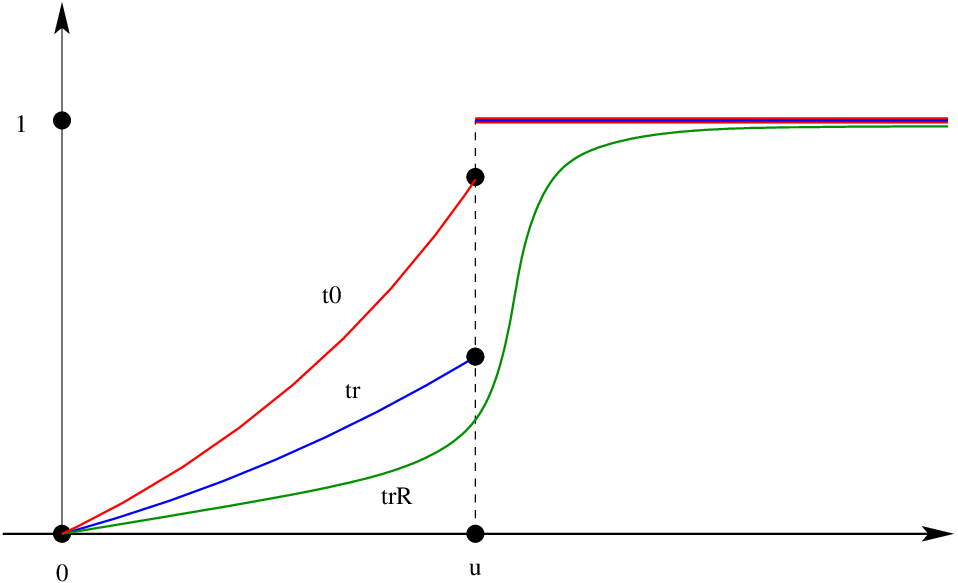}
\end{center}

\begin{center}
\begin{tabular}{ll}
Fig.~1: & A heuristic sketch of the functions $\theta_0,\theta_r, \theta_{r,R}$ (with a possible
\\
&but not expected jump at $u_*$ for $\theta_0$ and $\theta_r$).
\end{tabular}
\end{center}

The function $\theta_r$ is non-decreasing, equals $1$ for $u > u_*$. It is left-continuous, and actually continuous except maybe at $u_*$, by the same proof as in the case $r = 0$, see Corollary 1.2 in \cite{Teix09b}. Note that $\theta_r(\cdot)$ is non-increasing in $r$ and
\begin{equation}\label{6.15}
\begin{array}{l}
\lim\limits_{r \r \infty} \; \downarrow \theta_r(u) = \IP \big[\bigcap\limits_{r \ge 0} \{B(0,r)  \Vu \infty\}] = \mbox{$\IP[\cV^u$ does not percolate}]
\\
\qquad \qquad \quad  \mbox{($= 1$ for $u > u_*$ and $0$ for $u < u_*$)}.
\end{array}
\end{equation}

\n
We now consider $0 < u < u_*$ and $\nu > 0$ such that $\wt{\nu} =\nu |D| < \omega_d$. We will first bound $K(\nu)$ from above, see (\ref{6.20}) below. For this purpose we consider $\ve > 0$, and note that by (\ref{6.14}), and the fact that $\theta_r(u_* + \ve) = 1$, for $r \ge 0$,
\begin{equation}\label{6.16}
\nu < \nu' \stackrel{\rm def}{=} \nu / \theta_{r,R}(u_* + \ve) < \o_d / |D|, \;\mbox{when $r \ge 0$ and $R \ge c(\nu,\ve,r)$}.
\end{equation}
We then consider $\varphi \in C^\infty_0 (\IR^d)$, non-negative, such that $\varphi > \sqrt{u_* + \ve} - \sqrt{u}$ on $B_{\nu' |D|} ( \subseteq D)$. Since $\theta_{r,R}(\cdot)$ is increasing, we have $\strokedint_D \theta_{r,R} ((\sqrt{u} + \varphi)^2) \,dz  >  \nu' \theta_{r,R}(u_* + \ve) = \nu$, and 
\begin{equation}\label{6.17}
K_{r,R} (\nu) \le \fd \;\dis\int_{\IR^d} |\nabla \varphi |^2 \,dz.
\end{equation}
Optimizing over $\varphi$, we find that
\begin{equation}\label{6.18}
K_{r,R}(\nu) \le \mbox{\f $\dis\frac{1}{d}$} \;(\sqrt{u_* + \ve} - \sqrt{u})^2 \;{\rm cap}_{\IR^d}(B_{\nu' |D|}).
\end{equation}
We can let $R$ tend to infinity and find with (\ref{6.11}) and (\ref{6.14}) (recall that $\theta_r (u_* + \ve) = 1$)  \begin{equation}\label{6.19}
K_r(\nu) \le \mbox{\f $\dis\frac{1}{d}$}\;(\sqrt{u_* + \ve} - \sqrt{u})^2 \,{\rm cap}_{\IR^d} (B_{\nu |D|}).
\end{equation}
Recall that $\wt{\nu} = \nu |D|$, so that letting $r$ tend to infinity and then $\ve$ to $0$, we find by (\ref{6.12}) the desired upper bound
\begin{equation}\label{6.20}
K(\nu) \le \mbox{\f $\dis\frac{1}{d}$} \;(\sqrt{u_*} - \sqrt{u})^2 \,{\rm cap}_{\IR^d} (B_{\wt{\nu}}), \;\mbox{for $\nu > 0$}
\end{equation}

\n
(and this bound immediately extends to $\nu = 0$).

\medskip
We now turn to the derivation of a lower bound. We consider
\begin{equation}\label{6.21}
0 < \delta < \nu \;\;\mbox{and} \;\; 0 < \ve < u_* - u .
\end{equation}
In view of (\ref{6.15}), we assume that
\begin{equation}\label{6.22}
\mbox{$r \ge r_1(\delta,\ve)$, so that $\theta_r (u_* - \ve) \le \delta$ and $\theta_r (u_* + \ve) \ge 1 - \delta$}.
\end{equation}
Then, given $r \ge r_1(\delta,\ve)$, we consider a sequence $R_n > r$ tending to infinity and a sequence $\varphi_n \in C^\infty_0 (\IR^d)$ of non-negative functions so that
\begin{equation}\label{6.24}
\strokedint_D \theta_{r,R_n} \big((\sqrt{u} + \varphi_n)^2) \,dz > \nu \;\;\mbox{and} \;\; \fd \; \dis\int_{\IR^d} |\nabla \varphi_n|^2 \,dz \underset{n}{\longrightarrow} \lim\limits_n \downarrow K_{r,R_n} (\nu) \stackrel{(\ref{6.11})}{=} K_r(\nu).
\end{equation}

\n
By Theorems 8.6 and 8.7 on p.~208 and 212 of \cite{LiebLoss01}, we can extract a subsequence, still denoted by $\varphi_n$, and find $\varphi \ge 0$ in $D^1(\IR^d)$, so that $\varphi_n \r \varphi$ a.e. and $\frac{1}{2d} \;\int_{\IR^d} |\nabla \varphi |^2 \,dz \le \liminf_n \frac{1}{2d} \;\int_{\IR^d}|\nabla \varphi_n |^2 \,dz = K_r(\nu)$. Then, letting $\ov{\theta}_r(\cdot)$ denote the right-continuous modification of $\theta_r(\cdot)$, we see with Fatou's lemma that $\strokedint_D 1 - \ov{\theta}_r ((\sqrt{u} + \varphi)^2)\,dz \le \strokedint_D \liminf_n (1 - \theta_r ((\sqrt{u} + \varphi_n)^2))\,dz \stackrel{\rm Fatou}{\le} \liminf_n \strokedint_D 1 - \theta_r ((\sqrt{u} + \varphi_n)^2)\,dz$ so that 
\begin{equation}\label{6.25}
\strokedint_D \ov{\theta}_r \big((\sqrt{u} + \varphi)^2\big) \,dz  \ge \limsup_n \strokedint_D \theta_r  \big((\sqrt{u} + \varphi_n)^2\big) \,dz \ge \nu.
\end{equation}
As a result we find that
\begin{equation}\label{6.26}
K_r(\nu) \ge \inf\Big\{ \fd \;\dis\int_{\IR^d} |\nabla \varphi |^2 \,dz, \varphi \ge 0, \,\varphi \;\mbox{in $D^1(\IR^d)$, and} \; \strokedint_D \ov{\theta}_r \big((\sqrt{u} + \varphi)^2\big) \,dz \ge \nu\Big\}.
\end{equation}
Note that for $\varphi$ in the above set, where the infimum is taken, one has by (\ref{6.22}) (and $\ov{\theta}_r(v) = \theta_r(v)$ for $v < u_*$)
\begin{equation}\label{6.27}
\nu \le \strokedint_D \ov{\theta}_r \big((\sqrt{u} + \varphi)^2\big) \,dz \le \delta \,\mu_D (\{\varphi \le \sqrt{u_* - \ve} - \sqrt{u}\}) + \mu_D (\{\varphi > \sqrt{u_* - \ve} - \sqrt{u}\})
\end{equation}
with $\mu_D$ the normalized Lebesgue measure on $D$. This implies that
\begin{equation}\label{6.28}
\mu_D (\{\varphi > \sqrt{u_* - \ve} - \sqrt{u}\}) \ge \mbox{\f $\dis\frac{1}{1 - \delta}$} \;(\nu - \delta) \stackrel{\rm def}{=} \nu_\delta.
\end{equation}

\n
If $\varphi^*$ denotes the symmetric decreasing rearrangement of $\varphi$, one knows by \cite{LiebLoss01}, p.~188-189, that $\int_{\IR^d} | \nabla \varphi^*|^2 \,dz \le \int_{\IR^d} | \nabla \varphi |^2 \,dz$. Moreover, by (\ref{6.28}) one has  $|\{\varphi^* > \sqrt{u_* - \ve} - \sqrt{u}\}| \ge \nu_\delta \,|D|$, and we find that
\begin{equation}\label{6.29}
\fr \;\dis\int_{\IR^d} | \nabla \varphi |^2 \,dz \ge \fr \; \dis\int_{\IR^d} | \nabla \varphi^* |^2 \,dz \ge (\sqrt{u_* - \ve} - \sqrt{u})^2 \,{\rm cap}_{\IR^d} (B_{\nu_\delta |D|}).
\end{equation}
This shows that for $r \ge r_1 (\delta,\ve)$
\begin{equation}\label{6.30}
K_r(\nu) \ge \mbox{\f $\dis\frac{1}{d}$} \;  (\sqrt{u_* - \ve} - \sqrt{u})^2  \,{\rm cap}_{\IR^d} (B_{\nu_\delta |D|}).
\end{equation}
Letting $r$ tend to infinity and then $\delta$ and $\ve$ to $0$, we find by (\ref{6.12}) that
\begin{equation}\label{6.31}
K(\nu) \ge \mbox{\f $\dis\frac{1}{d}$} \; (\sqrt{u_*} - \sqrt{u})^2  \,{\rm cap}_{\IR^d} (B_{\wt{\nu}}).
\end{equation}
Combined with the upper bound (\ref{6.20}), this concludes the proof of Proposition \ref{prop6.5}.
\end{proof}

\begin{remark}\label{rem6.6} \rm 1) It follows from Proposition \ref{prop6.5} and the formula for the Brownian capacity of a ball, see p.~58 of \cite{PortSton78}, that when $\nu$ tends to $0$, $K(\nu) \sim c(u) \,\nu^{\frac{d-2}{d}}$. In particular, the function $K(\nu)$, $0 \le \nu < 1$ is not convex, as follows by looking at its values at $0,  \frac{\nu_0}{2}, \nu_0$ for some small $\nu_0$ in $[0,1)$. By the definition (\ref{6.12}), (\ref{6.11}) of $K(\cdot)$, we see that for large $r$ and sufficiently large $R$ (depending on $r$), the increasing homeomorphism $\nu \in [0,1) \rightarrow K_{r,R}(\nu) \in [0,\infty)$, see (\ref{5.84b}), is not a convex function (as follows by looking at its value at $0$, $\frac{\nu_0}{2}$, and $\nu_0$).

\bigskip\n
2) The above Proposition \ref{prop6.5} offers a heuristic link to the results of \cite{Szni} where $\cC^u_N$, the connected component of $S(0,N)$ in $\cV^u \cup S(0,N)$ was considered. To streamline the statement of the results, see (\ref{0.11}), (\ref{0.12}) in the Introduction, we recall them in the case where the equalities $\ov{u} = u_* = u_{**}$ hold, with $u < \ov{u}$ the strong percolative regime, and $u > u_{**}$ the strong non-percolative regime for $\cV^u$. These equalities are plausible, but open at the moment. Then, Theorems 3.1 and 3.2 of \cite{Szni} show that there is a scale $\wt{L}_0(N) = o(N)$, such that looking at $\wt{\cC}^u_N$ the $\wt{L}_0(N)$-thickening in sup-distance of $\cC^u_N$ and the hole $B(0,N) \backslash \wt{\cC}^u_N$ left by $\wt{\cC}^u_N$ in $B(0,N)$, one has:
\begin{equation}\label{6.32}
\begin{array}{l}
\lim\limits_N \; \mbox{\f $\dis\frac{1}{N^{d-2}}$} \;\log \IP\big[|B(0,N) \backslash \wt{\cC}^u_N| \ge \nu \,|B(0,N)|\big] = - \mbox{\f $\dis\frac{1}{d}$} \;(\sqrt{u_*} - \sqrt{u})^2 \,{\rm cap}_{\IR^d} (B_{\wt{\nu}}),
\\
\mbox{when $\wt{\nu} = 2^d \,\nu < \o_d$, and $0 < u < u_*$}.
\end{array}
\end{equation}
(and the hole is nearly spherical, see (3.66) of \cite{Szni}).

\medskip
Proposition \ref{prop6.5} sheds some light at a heuristic level  on the role of the thickening of $\cC^u_N$, as already pointed out below (\ref{0.13}) in the Introduction. This naturally leads to the question of understanding what happens, in the absence of thickening, when one directly considers $|B(0,N) \backslash \cC^u_N|$. As we now explain, the results obtained in the present article imply that for $0 < u < u_*$ and $\IP[0\Vu \infty] < \nu < 1$, 
\begin{equation}\label{6.33}
\begin{array}{l}
\liminf\limits_N \; \mbox{\f $\dis\frac{1}{N^{d-2}}$} \;\log \IP\big[|B(0,N) \backslash \cC^u_N| \ge \nu \,|B(0,N)|\big] \ge
\\
- \inf \Big\{ \fd \;\dis\int_{\IR^d} |\nabla \varphi |^2\,dz ; \, \varphi \ge 0, \, \varphi \in C^\infty_0(\IR^d), \;\strokedint_D \theta_0 \big((\sqrt{u} + \varphi)^2\big) \,dz > \nu\Big\}
\end{array}
\end{equation}
(with $D = [-1,1]^d$ and $\theta_0(v) = \IP [ 0 \Vv \infty ]$ as in (\ref{6.14})).

\n
Indeed, for any $R\ge0$  and large $N$, the $(0,R)$-disconnected sites (see (\ref{6.5})), contained in $B(0, N - R -1)$ belong to $B(0,N) \backslash \cC^u_N$. Applying the lower bound part of  (\ref{6.8}) with $\nu'$ in place of $\nu$, where $1>\nu'>\nu>\IP[0\Vu \infty]$($\ge  \theta_{0,R}(u)$), we find using the continuity of $K_{0,R} (\cdot)$, see (\ref{5.84b}), and letting successively $\nu'$ decrease to $\nu$ and $R$ tend to infinity that

\begin{equation}\label{6.34}
\liminf\limits_N \; \mbox{\f $\dis\frac{1}{N^{d-2}}$} \;\log \IP\big[|B(0,N) \backslash \cC^u_N| \ge \nu \,|B(0,N)|\big] \ge - K_{0} (\nu).
\end{equation}

\n
Now any $\varphi \ge 0$ in $C^\infty_0(\IR^d)$ such that $\strokedint_D \theta_0 \big((\sqrt{u} + \varphi)^2\big) \,dz > \nu$ satisfies for large $R$ $\strokedint_D \theta_{0,R} \big((\sqrt{u} + \varphi)^2\big) \,dz > \nu$. So,with (\ref{6.7}) and (\ref{6.11}),the right member of (\ref{6.33}) is smaller or equal to $- K_{0} (\nu)$, and (\ref{6.33}) follows.

The following questions are then natural. Can one replace the inequality in (\ref{6.33}) by an equality, and the $\liminf$ by a limit ? Can one attach non-negative minimizers $\varphi$ to the corresponding variational problems, and if so, get some insight into the nature, in particular the presence or absence, of sets where they reach the value $\sqrt{u_*} - \sqrt{u}$ \,? (When non-empty, such sets could be viewed as ``droplets''.) 

 \hfill $\square$
\end{remark}



\begin{thebibliography}{10} 

\bibitem{AsseScha17} A.~Asselah and B.~Schapira. 
Moderate deviations for the range of a transient walk: path concentration. {\em Ann. Sci. \'Ec. Norm. Sup\'er.}, 50(4), no. 3, 755-786, 2017.

\bibitem{AsseScha18} A.~Asselah and B.~Schapira. 
On the nature of the Swiss cheese in dimension 3.
{\em Ann. Probab.}, 48(2):1002--1013, 2020.

\bibitem{BoltHollVanb01} M.~van~den Berg, E.~Bolthausen, and F.~den Hollander.
Moderate deviations for the volume of the {W}iener sausage. {\em Ann. Math.(2)}, 153(2):355--406, 2001.

\bibitem{DebePopo15} D.~de~Bernardini and S.~Popov.
Russo's formula for random interlacements.  {\em J. Stat. Phys.}, 160(2):321--335, 2015.

\bibitem{Bodi99}
T.~Bodineau.
The {W}ulff construction in three and more dimensions. {\em Commun. Math. Phys.}, 207:197--229, 1999.

\bibitem{Cerf00}
R.~Cerf.
Large deviations for three dimensional supercritical percolation. {\em {\rm Ast\'erisque 267, Soci\'et\'e Math\'ematique de France}}, 2000.

\bibitem{CernTeix12}
J. \v{C}ern\'y and A. Teixeira. From random walk trajectories to random interlacements. Sociedade Brasileira de Matem\'atica, 23:1–78, 2012.


\bibitem{ChiaNitz19} A.~Chiarini and M.~Nitzschner.
Entropic repulsion for the occupation-time field of random interlacements conditioned on disconnection. {\em Ann.~Probab.}, 48(3):1317--1351, 2020.

\bibitem{ComeGallPopoVach13}
F.~Comets, Ch. Gallesco, S.~Popov, and M.Vachkovskaia.
On large deviations for the cover time of two-dimensional torus. {\em Electron. J. Probab.}, 96:1--18, 2013.


\bibitem{Dalm93} G.~Dal$\;$Maso.
{\em An Introduction to $\Gamma$-Convergence}, volume~8 of {\em  Progress in Nonlinear Differential Equations and their Applications.} Birkh\"auser, 1993.


\bibitem{DeusStro89} J.D. Deuschel and D.W. Stroock.
{\em Large Deviations}.  Academic Press, Boston, 1989.

\bibitem{Dono69} W.F. Donoghue. {\em Distributions and Fourier Transforms}, Academic Press, New York, 1969.


\bibitem{DrewRathSapo14c}
A. Drewitz, B. R\'ath, and A. Sapozhnikov. An Introduction to Random Interlace- ments. SpringerBriefs in Mathematics, Berlin, 2014. 


\bibitem{DrewPrevRodr}
A.~Drewitz, A.~Pr\'evost, and P.-F. Rodriguez. Geometry of {G}aussian free field sign clusters and random interlacements.
 {\em Preprint, {\rm also available at arXiv:1811.05970}}.



\bibitem{DumGosRodrSev}
H.~Duminil-Copin, S.~Goswami, P.-F. Rodriguez, and F. Severo. Equality of critical parameters for percolation of {G}aussian free field level-sets.
 {\em Preprint}, 2020, also available at arXiv:2002.07735.



\bibitem{Kren85} U.~Krengel.
{\em Ergodic Theorems}.  Walter de Gruyter, Berlin, 1985.

\bibitem{Lang72} S.~Lang. {\it Differential Manifolds}. Addison-Wesley Publishing Company, Reading, Mass., 1972.

\bibitem{Lawl91} G.F. Lawler.
{\em Intersections of Random Walks}.  Birkh\"auser, Basel, 1991.

\bibitem{Ledo01} M.~Ledoux.
{\em The Concentration of Measure Phenomenon.}. Mathematical Surveys and Monographs, 89, AMS, Providence, 2001.

\bibitem{Li17} X.~Li.
A lower bound for disconnection by simple random walk.  {\em Ann. Probab.}, 45(2):879--931, 2017.

\bibitem{LiSzni14} X.~Li and A.S. Sznitman.
A lower bound for disconnection by random interlacements. {\em Electron. J. Probab.}, 19(17):1--26, 2014.

\bibitem{LiSzni15} X.~Li and A.S. Sznitman.
Large deviations for occupation time profiles of random interlacements.
{\em Probab. Theory Relat. Fields}, 161:309--350, 2015.

\bibitem{LiebLoss01} E.~Lieb and M.~Loss.
{\em Analysis}, volume~14 of {\em Graduate Studies in Mathematics}. Second edition, AMS, Providence, 2001.

\bibitem{MarcRose06} M.B. Marcus and J.~Rosen.
{\em Markov Processes, Gaussian Processes, and Local Times}. Cambridge University Press, 2006.

\bibitem{Neve75}J.~Neveu.
{\em Discrete-Parameter Martingales}. North-Holland Publ. Company, Amsterdam, 1975.

\bibitem{PopoTeix15} S.~Popov and A.~Teixeira.
Soft local times and decoupling of random interlacements. {\em J. Eur. Math. Soc.}, 17(10):2545--2593, 2015.

\bibitem{PortSton78}
S.~Port and C.~Stone.
\newblock {\em Brownian Motion and Classical {P}otential {T}heory}.
\newblock Academic Press, New York, 1978.


\bibitem{Szni10} A.S. Sznitman.
Vacant set of random interlacements and percolation.  {\em Ann. Math.(2)}, 171(3):2039--2087, 2010.

\bibitem{Szni12d} A.S. Sznitman.
Random interlacements and the {G}aussian free field.  {\em Ann. Probab.}, 40(6):2400--2438, 2012.

\bibitem{Szni15} A.S. Sznitman.
Disconnection and level-set percolation for the {G}aussian free field.
{\em J. Math. Soc. Japan}, 67(4):1801--1843, 2015.

\bibitem{Szni17} A.S. Sznitman.
Disconnection, random walks, and random interlacements. {\em Probab. Theory Relat. Fields}, 167(1-2):1--44, 2017; the numbering quoted here in the text is the same as in arXiv:1412.3960 (the numbering of sections in the PTRF article is shifted by one unit).

\bibitem{Szni} A.S. Sznitman.
On macroscopic holes in some supercritical strongly dependent percolation models.
{\em Ann. Probab.}, 47(4):2459--2493, 2019.

\bibitem{Teix09b} A.~Teixeira.
On the uniqueness of the infinite cluster of the vacant set of random interlacements. {\em Ann. Appl. Probab.}, 19(1):454--466, 2009.



\end{thebibliography}
\end{document}